\documentclass[twoside,reqno]{amsart}
\usepackage{amssymb,amsmath,amsfonts,amsthm}
\usepackage{enumerate}
\usepackage{amsfonts}
\usepackage{array}

\newtheorem{theorem}{Theorem}
\newtheorem{remark}[theorem]{Remark}

\newtheorem{proposition}[theorem]{Proposition}
\newtheorem{lemma}[theorem]{Lemma}
\newtheorem{corollary}[theorem]{Corollary}
\newtheorem{definition}[theorem]{Definition}

\newcommand{\R}{{\mathbb R}}

\def\ddue{\partial_{x_2}}
\def\dtre{\partial_{x_3}}
\def\dt{\partial_{t}}

\begin{document}

\title[From $H^1$ to $L^2$]{Characteristic boundary value problems:\\
 Estimates from $H^1$ to $L^2$.}
\author[A.~Morando] {Alessandro Morando}
\address{DICATAM, Sezione di Matematica, \newline \indent
Universit\`a di Brescia,\newline \indent Via Valotti, 9, 25133 BRESCIA, Italy}
\email{alessandro.morando@ing.unibs.it, paolo.secchi@ing.unibs.it, paola.trebeschi@ing.unibs.it}

\author[P.~Secchi]{Paolo Secchi}

\author[P.~Trebeschi]{Paola Trebeschi}

\date{\today}

\begin{abstract}
Motivated by the study of certain non linear free-boundary value problems for hyperbolic systems of partial differential equations arising in Magneto-Hydrodynamics, in this paper we show that an a priori estimate of the solution to certain boundary value problems, in the conormal Sobolev space $H^1_{ tan}$, can be transformed into an $L^2$ a priori estimate of the same problem.
\end{abstract}


\maketitle

\section{Introduction and main results}\label{int}

The present paper is motivated by the study of certain non linear free boundary value problems for hyperbolic systems of partial differential equations arising in {\it Magneto-Hydrodynamics} (MHD).

The well-posedness of initial boundary value problems for hyperbolic PDEs was studied by Kreiss \cite{kreiss} for systems and Sakamoto \cite{sakamoto1,sakamoto2} for wave equations.
The theory was extended to free-boundary problems for a discontinuity by Majda \cite{majda2,majda1}. He related the discontinuity problem to a half-space problem by adding a new variable that describes the displacement of the discontinuity, and making a change of independent variables that \lq\lq flattens\rq\rq the discontinuity front.
The result is a system of hyperbolic PDEs that is coupled with an equation for the displacement of the discontinuity. Majda formulated analogs of the Lopatinski\u{\i} and uniform Lopatinski\u{\i} conditions for discontinuity problems, and proved a short-time, nonlinear existence and stability result for Lax shocks in solutions of hyperbolic conservation laws that satisfy the uniform Lopatinski\u{\i} condition (see \cite{benzoni-serre, metivier2} for further discussion).




Interesting and challenging problems arise when the discontinuity is weakly but not strongly stable, i.e. the Lopatinski\u{\i} condition only holds in weak form, because surface waves propagate along the discontinuity, see \cite{nappes,hunter}. A general theory for the evolution of such weakly stable discontinuities is lacking.

A typical difficulty in the analysis of  weakly stable problems is the loss of regularity in the a priori estimates of solutions. Short-time existence results have been obtained for various weakly stable nonlinear problems, typically by the use of a Nash-Moser scheme to compensate for the loss of derivatives in the linearized energy estimates, see \cite{alinhac,ChenWang,cs,SeTrNl,trakhinin09arma}.

A fundamental part of the general approach described above is given by the proof of the well-posedness of the linear boundary value problems (shortly written BVPs in the sequel) obtained from linearizing the nonlinear problem (in the new independent variables with \lq\lq flattened\rq\rq boundary) around a suitable basic state. This requires two things: the proof of a linearized energy estimate, and the existence of the solution to the linearized problem.

In case of certain problems arising in MHD, a spectral analysis of the linearized equations, as required by the Kreiss-Lopatinski\u{\i} theory, seems very hard to be obtained because of big algebraic difficulties. An alternative approach for the proof of the linearized a priori estimate is the energy method. This method has been applied successfully to the linearized MHD problems by Trakhinin (see \cite{trakhinin05,trakhinin10} and other references); typically the method gives an a priori estimate for the solution in the conormal Sobolev space $H^1_{ tan}$ (see Section \ref{fs} for the definition of this space) bounded by the norm of the source term in the same function space (or a space of higher order in case of loss of regularity).

Once given the a priori estimate, the next point requires the proof of the existence of the solution to the linearized problem. Here one finds a new difficulty. The classical duality method for the existence of a weak $L^2$ solution requires an a priori estimate for the dual problem (usually similar to the given linearized problem) of the form $L^2-L^2$ (from the data in the interior to the solution, disregarding for simplicity the boundary regularity).
In case of loss of derivatives, when for the problem it is given an estimate of the form $H^{1}-L^2$, one would need an estimate of the form $L^2-H^{-1}$ for the dual problem,
see \cite{coulombel05}.

The existence of a solution directly in $H^1_{ tan}$ would require an a priori estimate for the dual problem in the dual spaces $(H^1_{ tan})'-(H^1_{ tan})'$ (possibly of the form $(H^1_{ tan})'-(H^2_{ tan})'$ in case of loss of regularity), but it is not clear how to get it.

This difficulty motivates the present paper. We show that an a priori estimate of the solution to certain BVPs in the conormal Sobolev space $H^1_{ tan}$ can be transformed into an $L^2$ a priori estimate, with the consequence that the existence of a weak $L^2$ solution can be obtained by the classical duality argument.

The most of the paper is devoted to the proof of this result. In the Appendix we present some examples of free-boundary problems in MHD that fit in the general formulation described below.


\medskip
For a given integer $n\ge 2$, let $\mathbb{R}^{n}_+$ denote the $n-$dimensional positive half-space
$$
\mathbb{R}^{n}_+:=\{x=(x_1,\dots,x_n)\in\mathbb{R}^{n}\,:\,\, x_1>0\}.
$$
We also use the notation $x':=(x_2,\dots,x_{n})$. The boundary of $\mathbb{R}^{n}_+$ will be sistematically identified with $\mathbb{R}^{n-1}_{x'}$.
\newline
We are interested in a boundary value problem  of the following form
\begin{subequations}\label{sbvp}
\begin{align}
\mathcal{L}_\gamma u+\rho_\sharp u=F\,,&\quad{\rm in}\,\,\mathbb{R}^{n}_+\,,\label{sbvp1}\\
b_\gamma\psi+\mathcal M_\gamma u+b_\sharp\psi+\ell_\sharp u=g\,,&\quad{\rm on}\,\,\mathbb{R}^{n-1}\,\label{sbvp2}\,.
\end{align}
\end{subequations}
In \eqref{sbvp1}, $\mathcal L_\gamma$ is the first-order linear partial differential operator
\begin{equation}\label{pdo}
\mathcal L_\gamma=\mathcal L_\gamma(x,D):=\gamma I_N+\sum\limits_{j=1}^{n}A_j(x)\partial_j+B(x)\,,
\end{equation}
where the shortcut $\partial_j:=\frac{\partial}{\partial x_j}$, for $j=1,\dots,n$, is used hereafter and $I_N$ denotes the $N\times N$ identity matrix. The coefficients $A_j=A_j(x)$, $B=B(x)$ ($1\le j\le n$) are $N\times N$ real matrix-valued functions in $C^{\infty}_{(0)}(\mathbb{R}^{n}_+)$, the space of restrictions to $\mathbb{R}^{n}_+$ of functions of $C^{\infty}_0(\mathbb{R}^{n})$\footnote{With a slight abuse, the same notations $C^\infty_{(0)}(\mathbb R^n_+)$, $C^\infty_0(\mathbb R^n)$ are used throughout the paper to mean the space of functions taking either scalar or matrix values (possibly with different sizes). We adopt the same abuse for other function spaces later on.}.
\newline
In \eqref{sbvp2},
\begin{subequations}\label{oprt_bordo}
\begin{align}
b_\gamma=b_\gamma(x',D'):=\gamma b_0(x')+\sum\limits_{j=2}^n b_j(x')\partial_j+\beta(x')\,,\label{oprt_bordo1}\\
\mathcal M_\gamma=\mathcal M_\gamma(x',D'):=\gamma M_0(x')+\sum\limits_{j=2}^n M_j(x')\partial_j+M(x')\label{oprt_bordo2}
\end{align}
\end{subequations}
are first-order linear partial differential operators, acting on the tangential variables $x'\in\mathbb R^{n-1}$; for a given integer $1\le d\le N$, the coefficients $b_j$, $\beta$ and $M_j$, $M$ (for $j=0,2,\dots,n$) are functions in $C^{\infty}_0(\mathbb{R}^{n-1})$ taking values in the spaces $\mathbb{R}^{d}$ and $\mathbb{R}^{d\times N}$ respectively. Finally, $\rho_\sharp=\rho_\sharp(x,Z, \gamma)$ in \eqref{sbvp1} and $b_\sharp=b_\sharp(x',D',\gamma)$, $\ell_\sharp=\ell_\sharp(x',D',\gamma)$ in \eqref{sbvp2} stand for ``lower order operators'' of pseudo-differential type, acting ``tangentially'' on $(u,\psi)$, whose symbols belong to suitable symbol classes introduced in Section \ref{conormalcalculus}. The operators $\rho_\sharp=\rho_\sharp(x,Z, \gamma)$, $b_\sharp=b_\sharp(x',D',\gamma)$, $\ell_\sharp=\ell_\sharp(x',D',\gamma)$ must be understood as some ``lower order perturbations'' of the leading operators $\mathcal L_\gamma$, $b_\gamma$ and $\mathcal M_\gamma$ in the equations \eqref{sbvp}; in the following we assume that the problem \eqref{sbvp}, with given operators $\mathcal L_\gamma$, $b_\gamma$, $\mathcal M_\gamma$, obeys a suitable a priori estimate which has to be ``stable'' under the addition of arbitrary lower order terms $\rho_\sharp u$, $b_\sharp\psi$, $\ell_\sharp u$ in the interior equation \eqref{sbvp1} and the boundary condition \eqref{sbvp2} (see the assumptions $(H)_1$, $(H)_2$ below).
The structure of the operators \eqref{oprt_bordo} and $\rho_\sharp$, $b_\sharp$, $\ell_\sharp$ in \eqref{sbvp} will be better described later on.
\newline
The unknown $u$, as well as the source term $F$, are $\mathbb{R}^N-$valued functions of $x$, the unknown $\psi$ is a scalar function of $x'$ \footnote{In nonlinear free-boundary problems the scalar function $\psi$ describes the displacement of the discontinuity.} and the boundary datum $g$ is an $\mathbb{R}^d-$valued function of $x'$. We may assume that $u$ and $F$ are compactly supported in the unitary $n-$dimensional positive half-cylinder $\mathbb{B}^+:=\{x=(x_1,x')\,:\,\,0\le x_1< 1\,,\,\,|x'|<1\}$. Analogously, we assume that $\psi$ and $g$ are compactly supported in the unitary $(n-1)-$dimensional ball $\mathcal{B}(0;1):=\{|x'|<1\}$. For an arbitrary $0<\delta_0<1$, we also set $\mathbb{B}^+_{\delta_0}:=\{x=(x_1,x')\,:\,\,0\le x_1< \delta_0\,,\,\,|x'|<\delta_0\}$ and $\mathcal{B}(0;\delta_0):=\{|x'|<\delta_0\}$.
\newline
The BVP \eqref{sbvp} has {\it characteristic boundary of constant multiplicity} $1\le r<N$ in the following sense: the coefficient $A_1$ of the normal derivative in $\mathcal L_\gamma$ displays the block-wise structure
\begin{equation}\label{strutturaablocchi}
A_1(x)=\begin{pmatrix}
A_1^{I,I} & A^{I,II}_1\\
A_1^{II,I} & A^{II,II}_1
\end{pmatrix}\,,
\end{equation}
where $A_1^{I,I}$, $A_1^{I,II}$, $A_1^{II,I}$, $A_1^{II,II}$ are respectively $r\times r$, $r\times (N-r)$, $(N-r)\times r$, $(N-r)\times (N-r)$ sub-matrices, such that
\begin{equation}\label{singularity}
A^{I,II}_{1\,|\,x_1=0}=0\,,\quad A^{II,I}_{1\,|\,x_1=0}=0\,,\quad A^{II,II}_{1\,|\,x_1=0}=0\,,
\end{equation}
and $A^{I,I}_1$ is invertible over $\mathbb{B}^+$. According to the representation above, we split the unknown $u$ as $u=(u^I, u^{II})$;  $u^{I}:=(u_1,\dots,u_r)\in\mathbb{R}^r$ and $u^{II}:=(u_{r+1},\dots,u_N)\in\mathbb{R}^{N-r}$ are said to be respectively the {\it noncharacteristic} and the {\it characteristic} components of $u$.
\newline
Concerning the boundary condition \eqref{sbvp2}, we firstly assume that the number $d$ of scalar boundary conditions obeys the assumption $d\le r+1$. As regards to the structure of the boundary operator $\mathcal M_\gamma$ in \eqref{oprt_bordo2}, we require that actually it acts nontrivially only on the noncharacteristic component $u^I$ of $u$; moreover we assume that the first-order leading part $\mathcal M^s_\gamma$ of $\mathcal M_\gamma$ only applies to a {\it subset of components of the non characteristic vector $u^I$}, namely there exists an integer $s$, with $1\le s\le r$, such that the coefficients $M_j$, $M$ of $\mathcal M_\gamma$ take the form
\begin{equation}\label{strutturaMgamma}
M_j=\begin{pmatrix}M^s_j & 0\end{pmatrix}\,,\quad M=\begin{pmatrix}M^I & 0\end{pmatrix}\,,\quad j=0,2,\dots,n\,,
\end{equation}
where the matrices $M^s_j=M^s_j(x')$ ($j=0,2,\dots,n$) and $M^I=M^I(x')$ belong respectively to $\mathbb R^{d\times s}$ and $\mathbb R^{d\times r}$. If we set $u^{I,s}:=(u_1,\dots,u_s)$, then the operator $\mathcal M_\gamma$ in \eqref{oprt_bordo2} may be rewritten, according to \eqref{strutturaMgamma}, as
\begin{equation}\label{Mgamma}
\mathcal M_\gamma u=\mathcal M^s_\gamma u^{I,s}+M^Iu^I\,,
\end{equation}
where $\mathcal M^s_\gamma$ is the first-order leading operator
\begin{equation}\label{leading}
\mathcal M^s_\gamma(x',D'):=\gamma M^s_0+\sum\limits_{j=2}^n M^s_j\partial_j\,.
\end{equation}
As we just said, the operator $\ell_\sharp=\ell_\sharp(x',D',\gamma)$ must be understood as a lower order perturbation of the leading part $\mathcal M^s_\gamma$ of the boundary operator $\mathcal M_\gamma$ in \eqref{Mgamma}; hence, according to the form of $\mathcal M_\gamma$, we assume that $\ell_\sharp$ only acts on the component $u^{I,s}$ of the unknown vector $u$, that is
\begin{equation}\label{ellsharp}
\ell_\sharp(x',D',\gamma)u=\ell_\sharp(x',D',\gamma)u^{I,s}\,.
\end{equation}
A BVP of the form \eqref{sbvp}, under the structural assumptions \eqref{strutturaablocchi}-\eqref{Mgamma}, comes from the study of certain non linear free boundary value problems for {\it hyperbolic} systems of partial differential equations arising in {\it Magneto-hydrodynamics}. Such problems model the motion of a compressible inviscid fluid, under the action of a magnetic field, when the fluid may develop discontinuities along a moving unknown characteristic hypersurface. 
As we already said, to show the local-in-time existence of such a kind of piecewise discontinuous flows, the classical approach consists, firstly, of reducing the original free boundary problem to a BVP set on a fixed domain, performing a nonlinear change of coordinates that sends the front of the physical discontinuities into a fixed hyperplane of the space-time domain. Then, one starts to consider the well posedness of the linear BVP obtained from linearizing the found nonlinear BVP around a basic state provided by a particular solution. The resulting linear problem displays the structure of the problem \eqref{sbvp}, where the unknown $u$ represents the set of physical variables involved in the model, while the unknown $\psi$ encodes the moving discontinuity front. The solvability of the linear BVP firstly requires that a suitable a priori estimate can be attached to the problem.
\newline
Let the operators $\mathcal L_\gamma$, $b_\gamma$, $\mathcal M_\gamma$ be given, with structure described by formulas \eqref{pdo}, \eqref{oprt_bordo1}, \eqref{strutturaablocchi}-\eqref{leading} above. We assume that the two alternative hypotheses are satisfied:
\begin{itemize}
\item[$(H)_1$.]{\bf A priori estimate with loss of regularity in the interior term.} For all symbols $\rho_\sharp=\rho_\sharp(x,\xi,\gamma)\in\Gamma^0$, $b_{\sharp}=b_{\sharp}(x',\xi',\gamma)\in\Gamma^0$ and $\ell_\sharp=\ell_\sharp(x',\xi',\gamma)\in\Gamma^0$, taking values respectively in $\mathbb R^{N\times N}$, $\mathbb R^d$ and $\mathbb R^{d\times s}$, there exist constants $C_0>0$, $\gamma_0\ge 1$, depending only on the matrices $A_j$, $B$, $b_j$, $\beta$, $M_j^s$, $M^I$ in \eqref{pdo}, \eqref{oprt_bordo}, \eqref{Mgamma}, \eqref{leading} and a finite number of semi-norms of $\rho_\sharp$, $b_{\sharp}$, $\ell_\sharp$, such that for all functions $u\in C^\infty_{(0)}(\mathbb{R}^n_+)$, compactly supported on $\mathbb{B}^+$, $\psi\in C^\infty_0(\mathbb{R}^{n-1})$, compactly supported on $\mathcal{B}(0;1)$, and all $\gamma\ge\gamma_0$ the following a priori energy estimate is satisfied
    \begin{equation}\label{h1h2estimate}
    \begin{array}{ll}
    \displaystyle{\gamma\left(||u||^2_{H^1_{ tan\,,\gamma}(\mathbb{R}^n_+)}+||u^I_{|\,x_1=0}||^2_{H^{1/2}_\gamma(\mathbb{R}^{n-1})}\right)+\gamma^2||\psi||^2_{H^1_\gamma(\mathbb{R}^{n-1})}}
    \\\\
    \quad\quad\displaystyle{\le {C_0}\left( \frac{1}{\gamma^3}||F||^2_{H^{2}_{tan,\gamma}(\mathbb{R}^n_+)}+\frac{1}{\gamma}||g||^2_{H^{3/2}_{\gamma}(\mathbb{R}^{n-1})}\right)\,,}
    \end{array}
    \end{equation}
    where $F:=\mathcal L_\gamma u + \rho_\sharp(x,Z,\gamma) u$, $g:=b_\gamma\psi+\mathcal M_\gamma u+b_{\sharp}(x',D',\gamma)\psi+\ell_\sharp (x', D', \gamma)u^{I,s}$ and $\rho_\sharp(x,Z,\gamma)$, $b_\sharp(x',D',
    \gamma)$, $\ell_\sharp(x',D',\gamma)$ are respectively the pseudo-differential operators with symbols $\rho_\sharp$, $b_\sharp$, $\ell_\sharp$.
\item[$(H)_2$.]{\bf A priori estimate without loss of regularity in the interior term.} For all symbols $b_{\sharp}=b_{\sharp}(x',\xi',\gamma)\in\Gamma^0$ and $\ell_\sharp=\ell_\sharp(x',\xi',\gamma)\in\Gamma^0$, taking values respectively in $\mathbb R^d$ and $\mathbb R^{d\times s}$, there exist constants $C_0>0$, $\gamma_0\ge 1$, depending only on the matrices $A_j$, $B$, $b_j$, $\beta$, $M_j^s$, $M^I$ in \eqref{pdo}, \eqref{oprt_bordo}, \eqref{Mgamma}, \eqref{leading} and a finite number of semi-norms of $b_{\sharp}$, $\ell_\sharp$, such that for all functions $u\in C^\infty_{(0)}(\mathbb{R}^n_+)$, compactly supported on $\mathbb{B}^+$, $\psi\in C^\infty_0(\mathbb{R}^{n-1})$, compactly supported on $\mathcal{B}(0;1)$, and all $\gamma\ge\gamma_0$ the following a priori energy estimate is satisfied
    \begin{equation}\label{h1estimate}
    \begin{array}{ll}
    \displaystyle{\gamma\left(||u||^2_{H^1_{ tan\,,\gamma}(\mathbb{R}^n_+)}+||u^I_{|\,x_1=0}||^2_{H^{1/2}_\gamma(\mathbb{R}^{n-1})}\right)+\gamma^2||\psi||^2_{H^1_\gamma(\mathbb{R}^{n-1})}}
    \\\\
    \quad\quad\displaystyle{\le \frac{C_0}{\gamma}\left(||F||^2_{H^{1}_{tan,\gamma}(\mathbb{R}^n_+)}+||g||^2_{H^{3/2}_{\gamma}(\mathbb{R}^{n-1})}\right)\,,}
    \end{array}
    \end{equation}
    where $F:=\mathcal L_\gamma u$, $g:=b_\gamma\psi+\mathcal M_\gamma u+b_{\sharp}(x',D',\gamma)\psi+\ell_\sharp (x', D', \gamma)u^{I,s}$ and $b_\sharp(x',D',
    \gamma)$, $\ell_\sharp(x',D',\gamma)$ are respectively the pseudo-differential operators with symbols $b_\sharp$, $\ell_\sharp$.
\end{itemize}

\vspace{.2cm}
\noindent
The symbol class $\Gamma^0$ and the related pseudo-differential operators will be introduced in Section \ref{conormalcalculus}. The function spaces and the norms involved in the estimates  \eqref{h1h2estimate}, \eqref{h1estimate} will be described in Section \ref{fs}.
\newline
By the hypotheses $(H)_1$ and $(H)_2$, we require that an a priori estimate in the {\it tangential} Sobolev space  (see the next Section \ref{fs} and Definition \ref{spconormalireali} below) is enjoyed by the BVP \eqref{sbvp}. The structure of the estimate is justified by the physical models that we plan to cover (see the Appendix \ref{appB}). The inserting of the zeroth order terms $\rho_\sharp(x,Z,\gamma)u$, $b_\sharp(x',D',\gamma)\psi$, $\ell_\sharp(x',D',\gamma)u^{I,s}$ in the interior source term $F$ and the boundary datum $g$ is a property of {\it stability} of the estimates \eqref{h1h2estimate}, \eqref{h1estimate}, under lower order operators. We notice, in particular, that the addition of $\ell_\sharp(x',D',\gamma)u^{I,s}$ in the boundary condition \eqref{sbvp2} only modifies the zeroth order term $M^Iu^I$ for the part that applies to the components $u^{I,s}$ of the noncharacteristic unknown vector $u^I$, see \eqref{Mgamma}, \eqref{ellsharp}. This behavior of the boundary condition, under lower order perturbations, is inspired by the physical problems to which we address. It happens sometimes that the specific structure of some coefficients involved in the zeroth order part of the original "unperturbed" boundary operator \eqref{Mgamma} is needed in order to derive an a priori estimate of the type \eqref{h1h2estimate} or \eqref{h1estimate} for the corresponding BVP \eqref{sbvp}; hence these coefficients of the boundary operator must be kept unchanged by the addition of some lower order perturbations.
\newline
Note also that the two a priori estimates in \eqref{h1h2estimate}, \eqref{h1estimate} exhibit a different behavior with respect to the interior data: in \eqref{h1h2estimate} a loss of one tangential derivative from the interior data $F$ occurs, whereas in \eqref{h1estimate} no loss of interior regularity is assumed. According to this different behavior, a stability assumption under lower order perturbations $\rho_\sharp$ of the interior operator $\mathcal L_\gamma$ is only required in $(H)_1$.
\newline
Both the estimates exhibit the same loss of regularity from the boundary data.

\vspace{.2cm}
\noindent
The aim of this paper is to prove the following result.
\begin{theorem}\label{mainthm}
Assume that the operators $\mathcal L_\gamma$, $b_\gamma$, $\mathcal M_\gamma$ have the structure described in \eqref{pdo}, \eqref{oprt_bordo1}, \eqref{strutturaablocchi}-\eqref{leading}. Let $0<\delta_0<1$.
\begin{itemize}
\item[1.] If the assumption $(H)_1$ holds true, then for all symbols $\rho_\sharp$, $b_{\sharp}, \ell_\sharp\in\Gamma^0$ there exist constants $\breve{C}_0>0$, $\breve{\gamma}_0\ge 1$, depending only on the matrices $A_j$, $B$, $b_j$, $\beta$, $M_j^s$, $M^I$ in \eqref{pdo}, \eqref{oprt_bordo}, \eqref{Mgamma}, \eqref{leading}, $\delta_0$ and a finite number of semi-norms of $\rho_\sharp$, $b_{\sharp}$, $\ell_\sharp$, such that for all functions $u\in C^\infty_{(0)}(\mathbb{R}^n_+)$, compactly supported on $\mathbb{B}^+_{\delta_0}$, $\psi\in C^\infty_0(\mathbb{R}^{n-1})$, compactly supported on $\mathcal{B}(0;\delta_0)$, and all $\gamma\ge\breve{\gamma}_0$ the following a priori energy estimate is satisfied
    \begin{equation}\label{l2h1estimate}
    \displaystyle{\gamma\left(||u||^2_{L^2(\mathbb{R}^n_+)}+||u^I_{|\,x_1=0}||^2_{H^{-1/2}_\gamma(\mathbb{R}^{n-1})}\right)+\gamma^2||\psi||^2_{L^2(\mathbb{R}^{n-1})}
    \le\breve{C}_0\left(\frac1{\gamma^3}||F||^2_{H^1_{tan,\gamma}(\mathbb{R}^n_+)}+\frac1{\gamma}||g||^2_{H^{1/2}_{\gamma}(\mathbb{R}^{n-1})}\right)\,,}
    \end{equation}
    where $F:=\mathcal L_\gamma u+\rho_\sharp(x,Z,\gamma)u$ and $g:=b_\gamma\psi+\mathcal M_\gamma u+b_{\sharp}(x',D',\gamma)\psi+\ell_\sharp (x', D', \gamma)u^{I,s}$.
\item[2.] If the assumption $(H)_2$ holds true, then for every pair of symbols $b_{\sharp}, \ell_\sharp\in\Gamma^0$ there exist constants $\breve{C}_0>0$, $\breve{\gamma}_0\ge 1$, depending only on the matrices $A_j$, $B$, $b_j$, $\beta$, $M_j^s$, $M^I$ in \eqref{pdo}, \eqref{oprt_bordo}, \eqref{Mgamma}, \eqref{leading}, $\delta_0$ and a finite number of semi-norms of $b_{\sharp}$, $\ell_\sharp$, such that for all functions $u\in C^\infty_{(0)}(\mathbb{R}^n_+)$, compactly supported on $\mathbb{B}^+_{\delta_0}$, $\psi\in C^\infty_0(\mathbb{R}^{n-1})$, compactly supported on $\mathcal{B}(0;\delta_0)$, and all $\gamma\ge\breve{\gamma}_0$ the following a priori energy estimate is satisfied
    \begin{equation}\label{l2estimate}
    \displaystyle{\gamma\left(||u||^2_{L^2(\mathbb{R}^n_+)}+||u^I_{|\,x_1=0}||^2_{H^{-1/2}_\gamma(\mathbb{R}^{n-1})}\right)+\gamma^2||\psi||^2_{L^2(\mathbb{R}^{n-1})}
    \le \frac{\breve{C}_0}{\gamma}\left(||F||^2_{L^2(\mathbb{R}^n_+)}+||g||^2_{H^{1/2}_{\gamma}(\mathbb{R}^{n-1})}\right)\,,}
    \end{equation}
    where $F:=\mathcal L_\gamma u$ and $g:=b_\gamma\psi+\mathcal M_\gamma u+b_{\sharp}(x',D',\gamma)\psi+\ell_\sharp (x', D', \gamma)u^{I,s}$.
\end{itemize}
\end{theorem}
The paper is organized as follows. In Section \ref{fs} we introduce the function spaces to be used in the following and the main related notations. In Section \ref{tt} we collect some technical tools, and the basic concerned results, that will be useful for the proof of Theorem \ref{mainthm}, given in Section \ref{l2}.
\newline
The Appendix \ref{appA} contains the proof of the most of the technical results used in Section \ref{l2}. The Appendix \ref{appB} is devoted to present some free boundary problems in MHD, that can be stated within the general framework developed in the paper.
\section{Function Spaces}\label{fs}
The purpose of this Section is to introduce the main function spaces to be used in the following and collect their basic properties.
For $\gamma\ge 1$ and $s\in\mathbb{R}$, we set
\begin{equation}\label{weightfcts}
\lambda^{s,\gamma}(\xi):=(\gamma^2+|\xi|^2)^{s/2}
\end{equation}
and, in particular, $\lambda^{s}:=\lambda^{s,1}$.
\newline
The Sobolev space of order $s\in\mathbb{R}$ in $\mathbb{R}^n$ is defined to be the set of all tempered distributions $u\in\mathcal{S}^{\prime}(\mathbb{R}^n)$ such that $\lambda^{s}\widehat{u}\in L^2(\mathbb{R}^n)$, being $\widehat{u}$ the Fourier transform of $u$. For $s\in\mathbb{N}$, the Sobolev space of order $s$ reduces to the set of all functions $u\in L^2(\mathbb{R}^n)$ such that $\partial^{\alpha}u\in L^2(\mathbb{R}^n)$, for all multi-indices $\alpha\in\mathbb{N}^n$ with $|\alpha|\le s$, where we have set
$$
\partial^{\alpha}:=\partial^{\alpha_1}_1\dots\partial^{\alpha_n}_n\,,\quad \alpha=(\alpha_1,\dots,\alpha_n)\,,
$$
and $|\alpha|:=\alpha_1+\dots+\alpha_n$, as it is usual.
\newline
Throughout the paper, for real $\gamma\ge 1$, $H^s_{\gamma}(\mathbb{R}^n)$ will denote the Sobolev space of order $s$, equipped with the $\gamma-$depending norm $||\cdot||_{s,\gamma}$ defined by
\begin{equation}\label{normagamma}
||u||^2_{s,\gamma}:=(2\pi)^{-n}\int_{\mathbb{R}^n}\lambda^{2s,\gamma}(\xi)|\widehat{u}(\xi)|^2d\xi\,,
\end{equation}
($\xi=(\xi_1,\dots,\xi_n)$ are the dual Fourier variables of $x=(x_1,\dots,x_n)$). The norms defined by \eqref{normagamma}, with different values of the parameter $\gamma$, are equivalent each other. For $\gamma=1$ we set for brevity $||\cdot||_{s}:=||\cdot||_{s,1}$ (and, accordingly, $H^s(\mathbb{R}^n):=H^s_{1}(\mathbb{R}^n)$).
\newline
It is clear that, for $s\in\mathbb{N}$, the norm in \eqref{normagamma} turns out to be equivalent, {\it uniformly with respect to} $\gamma$, to the norm $||\cdot||_{H^s_{\gamma}(\mathbb{R}^n)}$ defined by
\begin{equation}\label{derivate}
||u||^2_{H^s_{\gamma}(\mathbb{R}^n)}:=\sum\limits_{|\alpha|\le s}\gamma^{2(s-|\alpha|)}||\partial^{\alpha}u||^2_{L^2(\mathbb{R}^n)}\,.
\end{equation}
\noindent
Another useful remark about the parameter depending norms defined in \eqref{normagamma} is provided by the following counterpart of the usual Sobolev imbedding inequality
\begin{equation}\label{gammaimbedding}
||u||_{s,\gamma}\le\gamma^{s-r}||u||_{r,\gamma}\,,
\end{equation}
for arbitrary $s\le r$ and $\gamma\ge 1$.
\begin{remark}\label{sobolevbordo}
{\rm In Section \ref{l2}, the ordinary Sobolev spaces, endowed with the weighted norms above, will be considered in $\mathbb{R}^{n-1}$ (interpreted as the boundary of the half-space $\mathbb{R}^n_+$) and used to measure the smoothness of functions on the boundary; regardless of the different dimension, the same notations and conventions as before will be used there.}
\end{remark}
The appropriate functional setting where one measures the internal smoothness of solutions to characteristic problems is provided by the anisotropic Sobolev spaces introduced by Shuxing Chen \cite{chen} and Yanagisawa, Matsumura \cite{yanagisawa91}, see also \cite{secchi00} . Indeed these spaces take account of the loss of normal regularity with respect to the boundary that usually occurs for characteristic problems.
\newline
Let $\sigma\in C^{\infty}([0,+\infty[)$ be a monotone increasing function such that $\sigma(x_1)=x_1$ in a neighborhood of the origin and $\sigma(x_1)=1$ for any $x_1$ large enough.
\newline
For $j=1,2,\dots,n$, we set
$$
Z_1:=\sigma(x_1)\partial_1\,,\quad Z_j:=\partial_j\,,\,\,{\rm for}\,\,j\ge 2\,.
$$
Then, for every multi-index $\alpha=(\alpha_1,\dots,\alpha_n)\in\mathbb{N}^n$, the differential operator $Z^{\alpha}$ in the tangential direction (conormal derivative) of order $|\alpha|$ is defined by
$$
Z^{\alpha}:=Z_1^{\alpha_1}\dots Z^{\alpha_n}_n\,.
$$
\noindent
Given an integer $m\geq 1$ the {\it anisotropic Sobolev space} $H^m_{\ast}(\mathbb{R}^n_+)$ of order $m$ is defined as the set of functions $u\in L^2(\mathbb{R}^n_+)$ such that
$Z^\alpha\partial_1^k u\in  L^2(\mathbb{R}^n_+)$, for all multi-indices $\alpha\in\mathbb{N}^n$ and $k\in\mathbb{N}$ with $|\alpha|+2k\le m$, see \cite{mosetre09} and the references therein. Agreeing with the notations set for the usual Sobolev spaces, for $\gamma\ge 1$, $H^m_{\ast,\gamma}(\mathbb{R}^n_+)$ will denote the anisotropic space of order $m$ equipped with the $\gamma-$depending norm
\begin{equation}\label{normaanisotropa}
||u||^2_{H^m_{\ast,\gamma}(\mathbb{R}^n_+)}:=\sum\limits_{|\alpha|+2k\le m}\gamma^{2(m-|\alpha|-2k)}||Z^{\alpha}\partial^k_1 u||^2_{L^2(\mathbb{R}^n_+)}\,.
\end{equation}
Similarly, the {\it conormal Sobolev space} $H^m_{tan}(\mathbb{R}^n_+)$ of order $m$ is defined to be the set of functions $u\in L^2(\mathbb{R}^n_+)$ such that
$Z^\alpha u\in  L^2(\mathbb{R}^n_+)$, for all multi-indices $\alpha$ with $|\alpha|\le m$. For $\gamma\ge 1$, $H^m_{tan,\gamma}(\mathbb{R}^n_+)$ denotes the conormal space of order $m$ equipped with the $\gamma-$depending norm
\begin{equation}\label{normaconormale}
||u||^2_{H^{m}_{tan,\gamma}(\mathbb{R}^n_+)}:=\sum\limits_{|\alpha|\le m}\gamma^{2(m-|\alpha|)}||Z^{\alpha}u||^2_{L^2(\mathbb{R}^n_+)}\,.
\end{equation}
In the end, we remark that the following identity $H^1_{\ast\,,\gamma}(\mathbb{R}^n_+)=H^1_{tan\,,\gamma}(\mathbb{R}^n_+)$ holds true. However, for a Sobolev order $m>1$ the continuous imbedding $H^m_{\ast, \gamma}(\mathbb R^n_+)\subset H^m_{tan, \gamma}(\mathbb R^n_+)$ is fulfilled with the strict inclusion relation.
\newline
Since the functions we are dealing with, throughout the paper, vanish for large $x_1$ (as they are compactly supported on $\mathbb{B}^+$), without the loss of generality we assume the conormal derivative $Z_1$ to coincide with the differential operator $x_1\partial_1$ from now on \footnote{Notice however that, for functions arbitrarily supported on $\mathbb{R}^n_+$, the conormal derivative $Z_1$ equals the singular operator $x_1\partial_1$ only locally near the boundary $\{x_1=0\}$; indeed, $Z_1$ behaves like the usual normal derivative $\partial_1$ far from the boundary, according to the properties of the weight $\sigma=\sigma(x_1)$.}. This reduction will make easier to implement on conormal spaces the technical machinery that will be introduced in the next Section.
\section{Preliminaries and technical tools}\label{tt}
We start by recalling the definition of two operators $\sharp$ and $\natural$, introduced by Nishitani and Takayama in \cite{nishitani00}, with the main property of mapping isometrically square integrable (resp. essentially bounded) functions over the half-space $\mathbb{R}^n_+$ onto square integrable (resp. essentially bounded) functions over the full space $\mathbb{R}^n$.
\newline
The mappings
$\sharp:L^2(\mathbb{R}^{n}_+)\rightarrow L^2(\mathbb{R}^{n})$ and
$\natural:L^{\infty}(\mathbb{R}^{n}_+)\rightarrow L^{\infty}(\mathbb{R}^{n})$
are respectively defined by
\begin{equation}\label{diesisbquadro}
w^{\sharp}(x):=w(e^{x_1},x')e^{x_1/2},\quad a^{\natural}(x)=a(e^{x_1},x')\,,\quad\,\forall\,x=(x_1,x')\in\mathbb{R}^n\,.
\end{equation}
They are both norm preserving bijections.
\newline
It is also useful to notice that the above operators can be extended to the set $\mathcal{D}^{\prime}(\mathbb{R}^n_+)$ of Schwartz distributions in $\mathbb{R}^n_+$. It is easily seen that both $\sharp$ and $\natural$ are topological isomorphisms of the space $C^{\infty}_0(\mathbb{R}^n_+)$ of test functions in $\mathbb{R}^n_+$ (resp. $C^{\infty}(\mathbb{R}^n_+)$) onto the space $C^{\infty}_0(\mathbb{R}^n)$ of test functions in $\mathbb{R}^n$ (resp. $C^{\infty}(\mathbb{R}^n)$). Therefore, a standard duality argument leads to define $\sharp$ and $\natural$ on $\mathcal{D}^{\prime}(\mathbb{R}^n_+)$, by setting for every $\varphi\in C^{\infty}_0(\mathbb{R}^n)$
\begin{eqnarray}
\langle u^{\sharp},\varphi\rangle:=\langle u,\varphi^{\sharp^{-1}}\rangle\,,\label{diesis}\\
\langle u^{\natural},\varphi\rangle:=\langle u,\varphi^{\flat}\rangle\label{biquadro}
\end{eqnarray}
($\langle\cdot,\cdot\rangle$ is used to denote the duality pairing between distributions and test functions either in the half-space $\mathbb{R}^n_+$ or the full space $\mathbb{R}^n$). In the right-hand sides of \eqref{diesis}, \eqref{biquadro}, $\sharp^{-1}$ is just the inverse operator of $\sharp$, that is
\begin{equation}\label{sharp-1}
\varphi^{\sharp^{-1}}(x)=\frac1{\sqrt{x_1}}\varphi(\log x_1,x')\,,\quad\forall x_1>0,\,\,x'\in\mathbb{R}^{n-1}\,,
\end{equation}
while the operator $\flat$ is defined by
\begin{equation}\label{flat}
\varphi^{\flat}(x)=\frac1{x_1}\varphi(\log x_1,x')\,,\quad\forall x_1>0,\,\,x'\in\mathbb{R}^{n-1}\,,
\end{equation}
for functions $\varphi\in C^{\infty}_0(\mathbb{R}^n)$. The operators $\sharp^{-1}$ and $\flat$ arise by explicitly calculating the formal adjoints of $\sharp$ and $\natural$ respectively.
\newline
Of course, one has that $u^{\sharp}, u^{\natural}\in\mathcal{D}^{\prime}(\mathbb{R}^n)$; moreover the following relations can be easily verified (cf. \cite{nishitani00})
\begin{eqnarray}
(\psi u)^{\sharp}=\psi^{\natural}u^{\sharp}\,,\label{calcolo1}\\
\partial_j (u^{\natural})=(Z_j u)^{\natural},\quad j=1,\dots,n\,,\label{calcolo2}\\
\partial_1 (u^{\sharp})=(Z_1 u)^{\sharp}+\frac12 u^{\sharp}\,,\label{calcolo3}\\
\partial_j (u^{\sharp})=(Z_j u)^{\sharp}\,,\quad j=2,\dots,n\,,\label{calcolo4}
\end{eqnarray}
\noindent
whenever $u\in\mathcal{D}^{\prime}(\mathbb{R}^n_+)$ and $\psi\in C^{\infty}(\mathbb{R}^n_+)$ (in \eqref{calcolo1} $u\in L^2(\mathbb{R}^n_+)$ and $\psi\in L^{\infty}(\mathbb{R}^n_+)$ are even allowed).
\newline
From formulas \eqref{calcolo3}, \eqref{calcolo4} and the $L^2-$boundedness of $\sharp$, it also follows that
\begin{equation}\label{Hmsharp}
\sharp: H^m_{tan,\gamma}(\mathbb{R}^n_+)\rightarrow H^m_{\gamma}(\mathbb{R}^n)
\end{equation}
is a topological isomorphism, for each integer $m\ge 1$ and real $\gamma\ge 1$.
\newline
The previous remarks give a natural way to extend the definition of the conormal spaces on $\mathbb{R}^n_+$ to an arbitrary real order $s$. More precisely we give the following
\begin{definition}\label{spconormalireali}
For $s\in\mathbb{R}$ and $\gamma\ge 1$, the space $H^s_{tan,\gamma}(\mathbb{R}^n_+)$ is defined as
$$
H^s_{tan,\gamma}(\mathbb{R}^n_+):=\{u\in\mathcal{D}'(\mathbb{R}^n_+):\,\,u^\sharp\in H^s_\gamma(\mathbb{R}^n)\}
$$
and is provided with the norm
\begin{equation}\label{snorma}
||u||^2_{s,tan,\mathbb{R}^n_+,\gamma}:=||u^\sharp||^2_{s,\gamma}=(2\pi)^{-n}\int_{\mathbb{R}^n}\lambda^{2s,\gamma}(\xi)|\widehat{u^{\sharp}}(\xi)|^2\,d\xi\,.
\end{equation}
\end{definition}
It is obvious that, like for the real order usual Sobolev spaces, $H^{s}_{tan,\gamma}(\mathbb{R}^n_+)$ is a Banach space for every real $s$; furthermore, the above definition reduces to the one given in Section \ref{fs} when $s$ is a positive integer. Finally, for all $s\in\mathbb{R}$, the $\sharp$ operator becomes a topological isomorphism of $H^s_{tan,\gamma}(\mathbb{R}^n_+)$ onto $H^s_\gamma(\mathbb{R}^n)$.
\newline
In the end, we observe that the following
$$
\sharp:C^{\infty}_{(0)}(\mathbb{R}^n_+)\rightarrow\mathcal{S}(\mathbb{R}^n)\,,\quad \natural:C^{\infty}_{(0)}(\mathbb{R}^n_+)\rightarrow C^{\infty}_b(\mathbb{R}^n)
$$
are linear continuous operators, where $\mathcal{S}(\mathbb{R}^n)$ denotes the Schwartz space of rapidly decreasing functions in $\mathbb{R}^n$ and $C^{\infty}_b(\mathbb{R}^n)$ the space of infinitely smooth functions in $\mathbb{R}^n$, with bounded derivatives of all orders; notice also that the last maps are not onto. Finally, we remark that
\begin{equation}\label{sharp-1}
\sharp^{-1}:\mathcal{S}(\mathbb{R}^n)\rightarrow C^\infty(\mathbb{R}^n_+)
\end{equation}
is a bounded operator.
\subsection{A class of conormal operators}\label{conormalcalculus}
The $\sharp$ operator, defined at the beginning of Section \ref{tt}, can be used to allow pseudo-differential operators in $\mathbb{R}^n$ {\it acting conormally} on functions only defined over the positive half-space $\mathbb{R}^n_+$. Then the standard machinery of pseudo-differential calculus (in the parameter depending version introduced in \cite{agranovic71}, \cite{chazarain-piriou}) can be re-arranged into a functional calculus properly behaved on conormal Sobolev spaces described in Section \ref{fs}. In Section \ref{l2}, this calculus will be usefully applied to derive from the estimate \eqref{h1h2estimate} or \eqref{h1estimate} associated to the BVP \eqref{sbvp} the corresponding estimate \eqref{l2h1estimate} or \eqref{l2estimate} of Theorem \ref{mainthm}.
\newline
Let us introduce the pseudo-differential symbols, with a parameter, to be used later; here we closely follow the terminology and notations of \cite{coulombeltesi}.
\begin{definition}\label{symbols}
A parameter-depending pseudo-differential symbol of order $m\in\mathbb{R}$ is a real (or complex)-valued measurable function $a(x,\xi,\gamma)$ on $\mathbb{R}^n\times\mathbb{R}^n\times[1,+\infty[$, such that $a$ is $C^{\infty}$ with respect to $x$ and $\xi$ and for all multi-indices $\alpha,\beta\in\mathbb{N}^n$ there exists a positive constant $C_{\alpha,\beta}$ satisfying:
\begin{equation}\label{symbolestimates}
|\partial^{\alpha}_{\xi}\partial^{\beta}_xa(x,\xi,\gamma)|\le C_{\alpha,\beta}\lambda^{m-|\alpha|,\gamma}(\xi)\,,
\end{equation}
for all $x,\xi\in\mathbb{R}^n$ and $\gamma\ge 1$.
\end{definition}
\noindent
The same definition as above extends to functions $a(x,\xi,\gamma)$ taking values in the space $\mathbb{R}^{N\times N}$ (resp. $\mathbb{C}^{N\times N}$) of $N\times N$ real (resp. complex)-valued matrices, for all integers $N>1$ (where the module $|\cdot|$ is replaced in \eqref{symbolestimates} by any equivalent norm in $\mathbb{R}^{N\times N}$ (resp. $\mathbb{C}^{N\times N}$)). We denote by $\Gamma^m$ the set of $\gamma-$depending symbols of order $m\in\mathbb{R}$ (the same notation being used for both scalar or matrix-valued symbols). $\Gamma^m$ is equipped with the obvious norms
\begin{equation}\label{semi-norma}
|a|_{m,k}:=\max\limits_{|\alpha|+|\beta|\le k}\sup\limits_{(x,\xi)\in\mathbb{R}^n\times\mathbb{R}^n\,,\,\,\gamma\ge 1}\lambda^{-m+|\alpha|,\gamma}(\xi)|\partial^{\alpha}_{\xi}\partial^{\beta}_x a(x,\xi,\gamma)|\,,\quad\forall\,k\in\mathbb{N}\,,
\end{equation}
which turn it into a Fr\'echet space. For all $m,m'\in\mathbb{R}$, with $m\le m'$, the continuous imbedding $\Gamma^m\subset\Gamma^{m'}$ can be easily proven.
\newline
For all $m\in\mathbb{R}$, the function $\lambda^{m,\gamma}$ is of course a (scalar-valued) symbol in $\Gamma^m$.
\newline
Any symbol $a=a(x,\xi,\gamma)\in\Gamma^m$ defines a {\it pseudo-differential operator} ${\rm Op}^{\gamma}(a)=a(x,D,\gamma)$ on the Schwartz space $\mathcal{S}(\mathbb{R}^n)$, by the standard formula
\begin{equation}\label{gammapsdo}
\forall\,u\in\mathcal{S}(\mathbb{R}^n)\,,\forall\,x\in\mathbb{R}^n\,,\,\,\,\,\,{\rm Op}^{\gamma}(a)u(x)=a(x,D,\gamma)u(x):=(2\pi)^{-n}\int_{\mathbb{R}^n}e^{ix\cdot\xi}a(x,\xi,\gamma)\widehat{u}(\xi)d\xi\,,
\end{equation}
where, of course, we denote $x\cdot\xi:=\sum\limits_{j=1}^nx_j\xi_j$. $a$ is called the symbol of the operator \eqref{gammapsdo}, and $m$ is its order. It comes from the classical theory that ${\rm Op}^{\gamma}(a)$ defines a linear bounded operator
\begin{equation*}
{\rm Op}^{\gamma}(a):\mathcal{S}(\mathbb{R}^n)\rightarrow\mathcal{S}(\mathbb{R}^n)\,;
\end{equation*}
moreover, the latter extends to a linear bounded operator on the space $\mathcal{S}^{\prime}(\mathbb{R}^n)$ of tempered distributions in $\mathbb{R}^n$.
\newline
Let us observe that, for a symbol $a=a(\xi,\gamma)$ independent of $x$, the integral formula \eqref{gammapsdo} defining the operator ${\rm Op}^\gamma(a)$ simply becomes
\begin{equation}\label{coeffcost}
{\rm Op}^\gamma(a)u=\mathcal{F}^{-1}(a(\cdot,\gamma)\widehat{u})=\mathcal{F}^{-1}(a(\cdot,\gamma))\ast u\,,\quad u\in\mathcal{S}'(\mathbb{R}^n)\,,
\end{equation}
\noindent
where $\mathcal{F}^{-1}$ denotes hereafter the inverse Fourier transform and $\ast$ is the convolution operator.
\newline
An exhaustive account of the symbolic calculus for pseudo-differential operators with symbols in $\Gamma^m$ can be found in \cite{chazarain-piriou} (see also \cite{coulombeltesi}). Here, we just recall the following result, concerning the composition and the commutator of two pseudo-differential operators.
\begin{proposition}\label{prodottoecommutatore}
Let $a\in\Gamma^m$ and $b\in\Gamma^{l}$, for $l,m\in\mathbb{R}$. Then the composed operator ${\rm Op}^{\gamma}(a){\rm Op}^{\gamma}(b)$ is a pseudo-differential operator with symbol in $\Gamma^{m+l}$; moreover, if we let $a\# b$ denote the symbol of the composition, one has for every integer $N\ge 1$
\begin{equation}\label{espansione1}
a\# b-\sum\limits_{|\alpha|< N }\frac{(-i)^{|\alpha|}}{\alpha !}\partial^{\alpha}_{\xi}a\partial^{\alpha}_x b\in\Gamma^{m+l-N}\,.
\end{equation}
Under the same assumptions, the commutator $\lbrack{\rm Op}^{\gamma}(a),{\rm Op}^{\gamma}(b)\rbrack:={\rm Op}^{\gamma}(a){\rm Op}^{\gamma}(b)-{\rm Op}^{\gamma}(b){\rm Op}^{\gamma}(a)$ is again a pseudo-differential operator with symbol $c\in\Gamma^{m+l}$. If we further assume that one of the two symbols $a$ or $b$ is scalar-valued (so that $a$ and $b$ commute in the point-wise product), then the symbol $c$ of $\lbrack{\rm Op}^{\gamma}(a),{\rm Op}^{\gamma}(b)\rbrack$ has order $m+l-1$.
\end{proposition}
\noindent
We point out that when the symbol $b\in\Gamma^l$ of the preceding statement does not depend on the $x$ variables (i.e. $b=b(\xi,\gamma)$) then the symbol $a\# b$ of ${\rm Op}^{\gamma}(a){\rm Op}^{\gamma}(b)$ reduces to the point-wise product of symbols $a$ and $b$, that is the {\it asymptotic} formula \eqref{espansione1} is replaced by the {\it exact} formula
\begin{equation}\label{prodottoaritmetico}
(a\# b)(x,\xi,\gamma)=a(x,\xi,\gamma)b(\xi,\gamma)\,.
\end{equation}
\begin{remark}\label{operatoribordo}
{\rm In the next Section \ref{l2}, in order to handle the boundary condition \eqref{sbvp2}, the algebra of pseudo-differential operators presented above will be used in the framework of $\mathbb{R}^{n-1}_{x'}$, considered as the boundary of the half-space $\mathbb{R}^n_+$. According to \eqref{gammapsdo}, for a boundary symbol $a=a(x',\xi',\gamma)$, $x',\xi'\in\mathbb{R}^{n-1}$, the related pseudo-differential operator will be denoted by ${\rm Op}^\gamma(a)$ or $a(x',D',\gamma)$. In particular, we will write $\lambda^{m,\gamma}(D')$ to mean the boundary operator with symbol $\lambda^{m,\gamma}(\xi')$ defined by \eqref{weightfcts} with $\xi'$ instead of $\xi$.}
\end{remark}
Starting from the symbolic classes $\Gamma^m$, $m\in\mathbb{R}$, we introduce now the class of {\it conormal operators} in $\mathbb{R}^n_+$, to be used in the sequel.
\newline
Let $a(x,\xi,\gamma)$ be a $\gamma-$depending symbol in $\Gamma^m$, $m\in\mathbb{R}$. The {\it conormal operator with symbol $a$}, denoted by ${\rm Op}^{\gamma}_{\sharp}(a)$ (or equivalently $a(x,Z,\gamma)$) is defined by setting
\begin{equation}\label{operatoreconormale}
\forall\,u\in C^{\infty}_{(0)}(\mathbb{R}^n_+)\,,\quad\left({\rm Op}^{\gamma}_{\sharp}(a)u\right)^{\sharp}=\left({\rm Op}^{\gamma}(a)\right)(u^{\sharp})\,.
\end{equation}
In other words, the operator ${\rm Op}^{\gamma}_{\sharp}(a)$ is the composition of mappings
\begin{equation}\label{composizione}
{\rm Op}^{\gamma}_{\sharp}(a)=\sharp^{-1}\circ {\rm Op}^{\gamma}(a)\circ\sharp\,.
\end{equation}
As we already noted, $u^{\sharp}\in\mathcal{S}(\mathbb{R}^n)$ whenever $u\in C^{\infty}_{(0)}(\mathbb{R}^n_+)$; hence formula \eqref{operatoreconormale} makes sense and gives that ${\rm Op}^{\gamma}_{\sharp}(a)u$ is a $C^{\infty}-$function in $\mathbb{R}^n_+$ (see also \eqref{sharp-1}). Also ${\rm Op}^{\gamma}_{\sharp}(a):C^{\infty}_{(0)}(\mathbb{R}^n_+)\to C^{\infty}(\mathbb{R}^n_+)$ is a linear bounded operator that extends to a linear bounded operator from the space of distributions $u\in\mathcal{D}^{\prime}(\mathbb{R}^n_+)$ satisfying $u^{\sharp}\in\mathcal{S}^{\prime}(\mathbb{R}^n)$ into $\mathcal{D}^{\prime}(\mathbb{R}^n_+)$ itself\footnote{In principle, ${\rm Op}^{\gamma}_{\sharp}(a)$ could be defined by \eqref{operatoreconormale} over all functions $u\in C^{\infty}(\mathbb{R}^n_+)$, such that $u^{\sharp}\in\mathcal{S}(\mathbb{R}^n)$. Then ${\rm Op}^{\gamma}_{\sharp}(a)$ defines a linear bounded operator on the latter function space, provided that it is equipped with the topology induced, via $\sharp$, from the Fr\'echet topology of $\mathcal{S}(\mathbb{R}^n)$.}. Throughout the paper, we continue to denote this extension by ${\rm Op}^{\gamma}_{\sharp}(a)$ (or $a(x,Z,\gamma)$ equivalently).
\newline
As an immediate consequence of \eqref{composizione}, we have that for all symbols $a\in\Gamma^m$, $b\in\Gamma^l$, with $m,l\in\mathbb{R}$, there holds
\begin{equation}\label{prodottoconormale}
\forall\,u\in C^{\infty}_{(0)}(\mathbb{R}^n_+)\,,\quad {\rm Op}^{\gamma}_{\sharp}(a){\rm Op}^{\gamma}_{\sharp}(b)u=\left({\rm Op}^{\gamma}(a){\rm Op}^{\gamma}(b)(u^{\sharp})\right)^{\sharp^{-1}}\,.
\end{equation}
\newline
Then, it is clear that a functional calculus of conormal operators can be straightforwardly borrowed from the corresponding pseudo-differential calculus in $\mathbb{R}^n$; in particular we find that products and commutators of conormal operators are still operators of the same type, and their symbols are computed according to the rules collected in Proposition \ref{prodottoecommutatore}.
\newline
Below, let us consider the main examples of conormal operators that will be met in Section \ref{l2}.
\newline
As a first example, we quote the multiplication by a matrix-valued function $B\in C^{\infty}_{(0)}(\mathbb{R}^n_+)$. It is clear that this makes an operator of order zero according to \eqref{operatoreconormale}; indeed \eqref{calcolo1} gives for any vector-valued $u\in C^{\infty}_{(0)}(\mathbb{R}^n_+)$
\begin{equation}
(Bu)^{\sharp}(x)=B^{\natural}(x)u^{\sharp}(x)\,,
\end{equation}
and $B^{\natural}$ is a $C^{\infty}-$function in $\mathbb{R}^n$, with bounded derivatives of any order, hence a symbol in $\Gamma^0$.
\newline
We remark that, when computed for $B^{\natural}$, the norm of order $k\in\mathbb{N}$, defined on symbols by \eqref{semi-norma}, just reduces to
\begin{equation}\label{semi-norma_ordine0}
|B^{\natural}|_{0,k}=\max\limits_{|\alpha|\le k}||\partial^{\alpha}B^{\natural}||_{L^{\infty}(\mathbb{R}^n)}=\max\limits_{|\alpha|\le k}||Z^{\alpha}B||_{L^{\infty}(\mathbb{R}^n_+)}\,,
\end{equation}
where the second identity above exploits formulas \eqref{calcolo2} and that $\natural$ maps isometrically $L^{\infty}(\mathbb{R}^n_+)$ onto $L^{\infty}(\mathbb{R}^n)$.
\newline
Now, let $\mathcal{L}:=\gamma I_N+\sum\limits_{j=1}^nA_j(x)Z_j$ be a first-order linear partial differential operator, with matrix-valued coefficients $A_j\in C^{\infty}_{(0)}(\mathbb{R}^n_+)$ for $j=1,\dots,n$ and $\gamma\ge 1$. Since the leading part of $\mathcal{L}$ only involves conormal derivatives, applying \eqref{calcolo1}, \eqref{calcolo3}, \eqref{calcolo4} then gives
\begin{equation*}
\left(\gamma u+\sum\limits_{j=1}^nA_jZ_ju\right)^{\sharp}=\left(\gamma I-\frac12 A^{\natural}_1\right) u^{\sharp}+\sum\limits_{j=1}^nA_j^{\natural}\partial_j u^{\sharp}={\rm Op}^{\gamma}(a)u^{\sharp}\,,
\end{equation*}
where $a=a(x,\xi,\gamma):=\left(\gamma I_N-\frac12A_1^{\natural}(x)\right)+i\sum\limits_{j=1}^nA^{\natural}_j(x)\xi_j$ is a symbol in $\Gamma^1$. Then $\mathcal{L}$ is a conormal operator of order $1$, according to \eqref{operatoreconormale}.
\newline
\subsection{Sobolev continuity of conormal operators}\label{sobolevcontinuity}
We recall the following classical Sobolev continuity property for ordinary pseudo-differential operators on $\mathbb{R}^n$.
\begin{proposition}\label{continuitasobolev}
If $s,m\in\mathbb{R}$ then for all $a\in\Gamma^m$ the pseudo-differential operator ${\rm Op}^{\gamma}(a)$ extends as a linear bounded operator from $H^{s+m}_{\gamma}(\mathbb{R}^n)$ into $H^s_{\gamma}(\mathbb{R}^n)$, and the operator norm of such an extension is uniformly bounded with respect to $\gamma$.
\end{proposition}
\noindent
We refer the reader to \cite{chazarain-piriou} for a detailed proof of Proposition \ref{continuitasobolev}. A thorough analysis shows that the norm of ${\rm Op}^{\gamma}(a)$, as a linear bounded operator from $H^{s+m}_{\gamma}(\mathbb{R}^n)$ to $H^s_{\gamma}(\mathbb{R}^n)$, actually depends only on a norm of type \eqref{semi-norma} of the symbol $a$, besides the Sobolev order $s$ and the symbolic order $m$ (cf. \cite{chazarain-piriou} for detailed calculations).
From the Sobolev continuity of pseudo-differential operators quoted above, and using that the operator $\sharp$ maps isomorphically conormal Sobolev spaces in $\mathbb{R}^n_+$ onto ordinary Sobolev spaces in $\mathbb{R}^n$ (see \eqref{Hmsharp} and Definition \ref{spconormalireali}), we easily derive the following result.
\begin{proposition}\label{continuitaconormale}
If $s,m\in\mathbb{R}$ and $a\in\Gamma^m$, then the conormal operator ${\rm Op}_{\sharp}^{\gamma}(a)$ extends to a linear bounded operator from $H^{s+m}_{tan,\gamma}(\mathbb{R}^n_+)$ to $H^{s}_{tan,\gamma}(\mathbb{R}^n_+)$; moreover the operator norm of such an extension is uniformly bounded with respect to $\gamma$.
\end{proposition}
\noindent
In order to perform the subsequent analysis, our interest will be mainly focused on the conormal operators of the type
\begin{equation}\label{conormaloprt}
\lambda^{m,\gamma}(Z):={\rm Op}^\gamma_{\sharp}(\lambda^{m,\gamma})\,,\quad m\in\mathbb{R}\,.
\end{equation}
Firstly, it is worth to remark that for each real $m$, the conormal operator $\lambda^{m,\gamma}(Z)$ is invertible, its two-sided inverse being provided by the operator  $\lambda^{-m,\gamma}(Z)$. Hence, applying Proposition \ref{continuitaconormale} to the operators $\lambda^{m,\gamma}(Z)$, $\lambda^{-m,\gamma}(Z)$ gives that the following
$$
\lambda^{m,\gamma}(Z):H^m_{tan,\gamma}(\mathbb{R}^n_+)\rightarrow L^2(\mathbb{R}^n_+)\,,\quad \lambda^{-m,\gamma}(Z): L^2(\mathbb{R}^n_+)\rightarrow H^m_{tan,\gamma}(\mathbb{R}^n_+)\,,
$$
are linear bounded operators. Notice also that, from Plancherel's identity, the norm \eqref{snorma} (with $m$ instead of $s$) on $H^m_{tan}(\mathbb{R}^n_+)$ can be restated in terms of the operator $\lambda^{m,\gamma}(Z)$ as
\begin{equation}\label{normalambda}
||u||_{m,tan,\mathbb{R}^n_+,\gamma}=||\lambda^{m,\gamma}(Z)u||_{L^2(\mathbb{R}^n_+)}\,.
\end{equation}
The relation \eqref{normalambda} will play an essential role in the proof of estimate \eqref{l2estimate}.
\section{Proof of Theorem \ref{mainthm}}\label{l2}
This Section is entirely devoted to the proof of Theorem \ref{mainthm}.
\subsection{The strategy of the proof.}\label{gs}
We closely follow the techniques developed in \cite{moseBVP} (see also \cite{moseIBVP}). In principle, for given smooth functions $u, \psi$ under the assumptions of Theorem \ref{mainthm}, we consider the problem analogous to \eqref{sbvp} solved by the functions $\lambda^{-1,\gamma}(Z)u$ and $\lambda^{-1,\gamma}(D')\psi$; \footnote{Actually, instead of $(\lambda^{-1,\gamma}(Z)u, \lambda^{-1,\gamma}(D')\psi)$ we will consider similar functions obtained by applying to $(u,\psi)$ a suitable modified version of the operators $\lambda^{-1,\gamma}(Z)$, $\lambda^{-1,\gamma}(D')$, that will be rigorously defined in Section \ref{mfm}. These new operators will be constructed in such a way to differ from $\lambda^{-1,\gamma}(Z)$, $\lambda^{-1,\gamma}(D')$ by suitable regularizing lower order reminders.} this problem is obtained by acting on the original BVP solved by $(u,\psi)$ by the operators $\lambda^{-1,\gamma}(Z)$, $\lambda^{-1,\gamma}(D')$ and making use of the rules of the symbolic calculus collected in Section \ref{conormalcalculus}. In the resulting equations, new terms appear, including the commutator between the differential operator $\mathcal L_\gamma$ and the conormal operator $\lambda^{-1,\gamma}(Z)$, in the interior equation, and similar commutators arising from the interaction of $\lambda^{-1,\gamma}(D')$ with the operators in the boundary condition. We apply the assumption $(H)_1$ (or $(H_2)$) to the problem for $(\lambda^{-1,\gamma}(Z)u, \lambda^{-1,\gamma}(D')\psi)$, writing for it the estimate \eqref{h1h2estimate} (or \eqref{h1estimate}). The structure of the estimates \eqref{h1h2estimate}, \eqref{h1estimate} allows to treat the commutator terms involved in the equations either as a part of the source terms or as lower order operators. The desired estimates \eqref{l2h1estimate}, \eqref{l2estimate} come respectively from \eqref{h1h2estimate}, \eqref{h1estimate} for $(\lambda^{-1,\gamma}(Z)u, \lambda^{-1,\gamma}(D')\psi)$, in view of the equivalence of norms \eqref{normalambda}, \eqref{normaconormale} in $H^m_{tan}(\mathbb{R}^n_+)$ and the similar equivalence of norms \eqref{normagamma}, \eqref{derivate} for ordinary Sobolev spaces on the boundary.
\subsection{A modified version of the conormal operator $\lambda^{-1,\gamma}(Z)$}\label{mfm}
As explained before, we are going to act on the equation \eqref{sbvp1}, written for a given smooth function $u$, by the conormal operator $\lambda^{-1,\gamma}(Z)$. To make possible the interaction between $\lambda^{-1,\gamma}(Z)$ and the term of $\mathcal L_\gamma$ involving the normal derivative $\partial_1$, we need to slightly modify the conormal operator $\lambda^{-1,\gamma}(Z)$. Here, we follow the ideas of \cite{moseBVP}.

\medskip
To be definite, let us illustrate the strategy for the operator $\lambda^{m,\gamma}(Z)$ with general order $m\in\mathbb{R}$.
The first step is to decompose the symbol $\lambda^{m,\gamma}$ as the sum of two contributions. To do so, we take an arbitrary positive, even function $\chi\in C^{\infty}(\mathbb{R}^n)$ with the following properties
\begin{equation}\label{chi}
0\le\chi(x)\le 1\,,\quad\forall\,x\in\mathbb{R}^n\,,\quad \chi(x)\equiv 1\,,\,\,{\rm for}\,\,|x|\le\frac{\varepsilon_0}{2}\,,\quad \chi(x)\equiv 0\,,\,\,{\rm for}\,\,|x|>\varepsilon_0\,,
\end{equation}
with a suitable $0<\varepsilon_0<1$ that will be specified later on, see Lemma \ref{supporto}.
Then, we set:
\begin{equation}\label{funzionepesomodificata}
\begin{array}{ll}
\lambda^{m,\gamma}_{\chi}(\xi):=\chi(D)(\lambda^{m,\gamma})(\xi)=(\mathcal{F}^{-1}\chi\ast\lambda^{m,\gamma})(\xi)\,,\\
\\
r_{m}(\xi,\gamma):=\lambda^{m,\gamma}(\xi)-\lambda^{m,\gamma}_{\chi}(\xi)=(I-\chi(D))(\lambda^{m,\gamma})(\xi)\,.
\end{array}
\end{equation}
The following result (see \cite[Lemma 4.1]{moseBVP}) shows that the function $\lambda^{m,\gamma}_{\chi}$ behaves, as a symbol, like $\lambda^{m,\gamma}$.
\begin{lemma}\label{lemmatecnico1}
Let the function $\chi\in C^{\infty}(\mathbb{R}^n)$ satisfy the assumptions in \eqref{chi}. Then $\lambda^{m,\gamma}_{\chi}$ is a symbol in $\Gamma^m$, i.e. for all $\alpha\in\mathbb{N}^n$ there exists a constant $C_{m,\alpha}>0$ such that:
\begin{equation}\label{lemmatecnicoeq1.0}
|\partial^{\alpha}_{\xi}\lambda^{m,\gamma}_{\chi}(\xi)|\le C_{m,\alpha}\lambda^{m-|\alpha|,\gamma}(\xi)\,,\quad\forall\,\xi\in\mathbb{R}^n\,.
\end{equation}
\end{lemma}
\noindent
An immediate consequence of Lemma \ref{lemmatecnico1} and \eqref{funzionepesomodificata} is that $r_{m}$ is also a $\gamma-$depending symbol in $\Gamma^{m}$.
\newline
Let us define, with the obvious meaning of the notations:
\begin{equation}
\begin{array}{ll}
\lambda^{m,\gamma}_{\chi}(D):={\rm Op}^{\gamma}(\lambda^{m,\gamma}_{\chi})\,,\quad r_{m}(D,\gamma):={\rm Op}^{\gamma}(r_{m})\,,\\
\\
\lambda^{m,\gamma}_{\chi}(Z):={\rm Op}^{\gamma}_{\sharp}(\lambda^{m,\gamma}_{\chi})\,,\quad r_{m}(Z,\gamma):={\rm Op}^{\gamma}_{\sharp}(r_{m})\,.
\end{array}
\end{equation}
A useful property of the modified operator $\lambda^{m,\gamma}_{\chi}(Z)$ is that it preserves the compact support of functions, as shown by the following
\begin{lemma}\label{supporto}
Let $0<\delta_0<1$ be fixed. There exists $\varepsilon_0=\varepsilon_0(\delta_0)>0$ such that, if $\chi\in C^\infty_0(\mathbb{R}^n)$ satisfies the assumption \eqref{chi} with the previous choice of $\varepsilon_0$, then for all $u\in C^{\infty}_{(0)}(\mathbb{R}^n_+)$, with ${\rm supp} u\subseteq\mathbb{B}^+_{\delta_0}$, we have
$$
{\rm supp}\lambda^{m,\gamma}_{\chi}(Z)u\subseteq\mathbb{B}^+\,.
$$
\end{lemma}
\begin{remark}
Note that the support of $\lambda^{m,\gamma}_{\chi}(Z)u$ is bigger than the support of $u$, depending on ${\rm supp}\,\chi$. Hence, if one wants that ${\rm supp}\,\lambda^{m,\gamma}(Z)u$ is contained in the fixed domain $\mathbb{B}^+$, one has to choose $\chi$ with sufficiently small support.
\end{remark}
\noindent
The second important result is concerned with the conormal operator $r_{m}(Z,\gamma)={\rm Op}^{\gamma}_{\sharp}(r_{m})$, and tells that it essentally behaves as a {\it regularizing} operator on conormal Sobolev spaces.
\begin{lemma}\label{lemmatecnico2}
i. For every $p\in\mathbb{N}$, the conormal operator $r_{m}(Z,\gamma)$ extends as a linear bounded operator, still denoted by $r_{m}(Z,\gamma)$, from $L^2(\mathbb{R}^n_+)$ to $H^p_{tan,\gamma}(\mathbb{R}^n_+)$.
\newline
ii. Moreover, for every $h\in\mathbb{N}$ there exists a positive constant $C_{p,h,n,\chi}$, depending only on $p, h$, $\chi$ and the dimension $n$, such that for all $\gamma\ge 1$ and $u\in L^2(\mathbb{R}^n_+)$:
\begin{equation}\label{kregolarizzazione}
||r_{m}(Z,\gamma)u||_{H^p_{tan,\gamma}(\mathbb{R}^n_+)}\le C_{p,h,n,\chi}\gamma^{-h}||u||_{L^2(\mathbb{R}^n_+)}\,.
\end{equation}
\end{lemma}
\noindent
The proof of Lemmata \ref{supporto}, \ref{lemmatecnico2} is postponed to Appendix \ref{appA}.
\newline
In the following sections, the above analysis will be applied to the operator $\lambda^{-1,\gamma}(Z)$. According to \eqref{funzionepesomodificata}, we decompose
\begin{equation}\label{splitting}
\lambda^{-1,\gamma}(Z)=\lambda^{-1,\gamma}_{\chi}(Z)+r_{-1}(Z,\gamma)\,.
\end{equation}
\subsection{A boundary operator}\label{bdrycndt}
As it was already explained in Section \ref{gs}, we need to derive the problem analogous to \eqref{sbvp} satisfied by $(\lambda^{-1,\gamma}(Z)u,\lambda^{-1,\gamma}(D')\psi)$ for given smooth functions $(u,\psi)$. Actually, as we said, $\lambda^{-1,\gamma}(Z)$ must be replaced by its modification $\lambda^{-1,\gamma}_\chi(Z)$ (see \eqref{splitting}). Analogously, we have to introduce an appropriate modification of $\lambda^{-1,\gamma}(D')$, to be used as a ``boundary counterpart'' of $\lambda^{-1,\gamma}_\chi(Z)$: this new operator comes from computing the value of $\lambda^{-1,\gamma}_{\chi}(Z)u$ on the boundary $\{x_1=0\}$. To this end, it is worthwhile to make an additional hypothesis about the smooth function $\chi$ involved in the definition of $\lambda^{-1,\gamma}_{\chi}(Z)$. We assume that $\chi$ has the form:
\begin{equation}\label{prodottotensoriale}
\forall\,x=(x_1,x')\in\mathbb{R}^n\,,\quad \chi(x)=\chi_1(x_1)\widetilde{\chi}(x')\,,
\end{equation}
where $\chi_1\in C^{\infty}(\mathbb{R})$ and $\widetilde{\chi}\in C^{\infty}(\mathbb{R}^{n-1})$ are given positive even functions, to be chosen in such a way that conditions \eqref{chi} are made satisfied.
\newline
As we did in Section \ref{mfm}, the result we are going to present here are stated for the general conormal operator $\lambda^{m,\gamma}(Z)$ with an arbitrary order $m\in\mathbb{R}$. All the proofs will be given in the Appendix \ref{appA}.
\newline
Following closely the arguments employed to prove \cite[Proposition 4.10]{moseBVP}, we are able to get the following
\begin{proposition}\label{bordo}
Assume that $\chi$ obeys the assumptions \eqref{chi}, \eqref{prodottotensoriale}. Then, for all $\gamma\ge 1$ and $m\in\mathbb{R}$ the function $b'_{m}(\xi',\gamma)$ defined by
\begin{equation}\label{simbolobordobis}
b'_{m}(\xi',\gamma):=(2\pi)^{-n}\int_{\mathbb{R}^n}\lambda^{m,\gamma}(\eta_1,\eta'+\xi')\left(e^{(\cdot)_1/2}\chi_1\right)^{\wedge_1}(\eta_1)\widehat{\widetilde{\chi}}(\eta')\,d\eta\,,\quad\forall\,\xi'\in\mathbb{R}^{n-1}\,,
\end{equation}
is a $\gamma-$depending symbol in $\mathbb{R}^{n-1}$ belonging to $\Gamma^{m}$, where $\wedge_1$ is used to denote the one-dimensional Fourier transformation with respect to $x_1$, while $\wedge$ denotes the $(n-1)-$dimensional Fourier transformation with respect to $x'$. Moreover, for all functions $u\in C^{\infty}_{(0)}(\mathbb{R}^n_+)$ there holds
\begin{equation}\label{formulatraccia}
\forall\,x'\in\mathbb{R}^{n-1}\,,\quad (\lambda^{m,\gamma}_{\chi}(Z)u)_{|\,x_1=0}(x')=b'_{m}(D',\gamma)(u_{|\,x_1=0})(x')\,.
\end{equation}
\end{proposition}
\noindent
The next Lemma shows that the boundary pseudo-differential operator $b'_{m}(D',\gamma)$ differs from the operator $\lambda^{m,\gamma}(D')$ by a lower order remainder.
\begin{lemma}\label{lemmatecnico5}
For $m\in\mathbb{R}$, let $b'_{m}(\xi',\gamma)$ be defined by \eqref{simbolobordobis}. Then there exists a symbol $\beta_{m}(\xi',\gamma)\in\Gamma^{m-2}$ such that:
\begin{equation}\label{splittingbordo}
b'_{m}(\xi',\gamma)=\lambda^{m,\gamma}(\xi')+\beta_{m}(\xi',\gamma)\,,\quad\forall\,\xi'\in\mathbb{R}^{n-1}\,.
\end{equation}
\end{lemma}
\noindent
As a consequence of Proposition \ref{bordo} and Lemma \ref{supporto}, we see now that, like $\lambda^{m,\gamma}_{\chi}(Z)$, the boundary operator $b'_m(D',\gamma)$ preserves the compactness of the support of functions on $\mathbb{R}^{n-1}$.
\begin{corollary}\label{supportobordo}
For all $m\in\mathbb{R}$ and $\psi\in C^\infty_0(\mathbb{R}^{n-1})$ with ${\rm supp}\,\psi\subset \mathcal{B}(0;\delta_0)$, then
\begin{equation}
{\rm supp}\,b'_{m}(D',\gamma)\psi\subset\mathcal{B}(0;1)\,.
\end{equation}
\end{corollary}
\noindent
In the following the results stated in Proposition \ref{bordo}, Lemma \ref{lemmatecnico5} and Corollary \ref{supportobordo} will be applied to the case of $m=-1$.
\subsection{Regularized BVP}\label{sec_regBVP}
From now on, we will focus on the proof of the estimate \eqref{l2h1estimate} stated in the first part of Theorem \ref{mainthm}, under the assumption $(H)_1$ about the BVP \eqref{sbvp}. The second part of Theorem \ref{mainthm} (estimate \eqref{l2estimate}, under the assumption $(H)_2$) follows by developing similar arguments to those explained here below; we will write in details only those steps which make the difference between the proof of the two statements $1$ and $2$ (see Section \ref{4.7}).
\newline
Let $(u,\psi)$ be given smooth functions obeying the assumptions of Theorem \ref{mainthm}. Given arbitrary symbols $\rho_\sharp=\rho_\sharp(x,\xi,\gamma)\in\Gamma^0$, $\ell_\sharp=\ell_\sharp(x',\xi',\gamma), b_\sharp=b_\sharp(x',\xi',\gamma)\in\Gamma^0$, let us set
\begin{equation}\label{tf}
F:=\mathcal L_\gamma u+\rho_\sharp(x,Z,\gamma)u\,,
\end{equation}
\begin{equation}\label{tb}
g:=b_\gamma\psi+\mathcal M^s_\gamma u^{I,s}+M^Iu^I+b_{\sharp}(x',D',\gamma)\psi+\ell_\sharp(x',D',\gamma)u^{I,s}\,.
\end{equation}
We are going to derive a corresponding BVP for the pair of functions $(\lambda^{-1,\gamma}_{\chi}(Z)u,b'_{-1}(D',\gamma)\psi)$, to which the a priori estimate \eqref{h1h2estimate} will be applied. Notice that, in view of Lemma \ref{supporto} and Corollary \ref{supportobordo}, the functions $\lambda^{-1,\gamma}_{\chi}(Z)u$, $b'_{-1}(D',\gamma)\psi$ are supported on $\mathbb{B}^+$ and $\mathcal{B}(0;1)$, as required in the hypothesis $(H)_1$, provided the function $\chi$ satisfies the assumptions \eqref{chi} with a sufficiently small $0<\varepsilon_0<1$.
\subsubsection{The interior equation}\label{interioreqt}
We follow the strategy already explained in Section \ref{gs}, where now the role of the operator $\lambda^{-1,\gamma}(Z)$ is replaced by $\lambda^{-1,\gamma}_{\chi}(Z)$. Thus, for a given smooth function $u\in C^\infty_{(0)}(\mathbb{R}^n_+)$, supported on $\mathbb{B}^+_{\delta_0}$, from \eqref{tf}, we find that
\begin{equation}\label{sistemaregolarizzato}
\mathcal L_\gamma(\lambda^{-1,\gamma}_{\chi}(Z)u)+\rho_\sharp(\lambda^{-1,\gamma}_{\chi}(Z)u)+[\lambda^{-1,\gamma}_{\chi}(Z),\mathcal L_\gamma+\rho_\sharp]u=\lambda^{-1,\gamma}_{\chi}(Z)F\,,\quad{\rm in}\,\,\mathbb{R}^n_+\,,
\end{equation}
where here and in the rest of this section, it is written $\rho_\sharp$ instead of $\rho_\sharp(x,Z,\gamma)$, in order to shorten formulas.
\newline
We will see that the commutator term $[\lambda^{-1,\gamma}_{\chi}(Z),\mathcal L_\gamma+\rho_\sharp]u$, involved in the left-hand side of the above equation, can be restated as a lower order pseudo-differential operator of conormal type with respect to $\lambda^{-1,\gamma}_{\chi}(Z)u$, up to some ``smoothing reminder'' to be treated as a part of the source term in the right-hand side of the equation.
\newline
To this end, we proceed as follows. Firstly, we decompose the commutator term in the left-hand side of \eqref{sistemaregolarizzato} as the sum of two contributions corresponding respectively to the {\it tangential} and {\it normal} components of $\mathcal L_\gamma$.
\newline
In view of \eqref{strutturaablocchi}, \eqref{singularity}, we may write the coefficient $A_1$ of the normal derivative $\partial_1$ in the expression \eqref{pdo} of $\mathcal L_\gamma$ as
\begin{equation}\label{strutturaA1}
A_1=A^1_1+A^2_1\,,\quad A^1_1:=\begin{pmatrix}A^{I,I}_1 & 0\\ 0 & 0\end{pmatrix}\,,\,\,A^2_{1\,|\,\,x_1=0}=0\,,
\end{equation}
hence
$$
A^2_1\partial_1=H_1Z_1\,,
$$
where $H_1(x)=x_1^{-1}A_1^2(x)\in C^{\infty}_{(0)}(\mathbb{R}^n_+)$. Accordingly, we split $\mathcal L_\gamma$ as
\begin{equation}\label{Ltan}
\mathcal L_\gamma=A^1_1\partial_1+\mathcal L_{tan, \gamma}\,,\quad \mathcal L_{tan, \gamma}:=\gamma I_N+H_1Z_1+\sum\limits_{j=2}^n A_jZ_j+B\,.
\end{equation}
Consequently, we have:
\begin{equation}\label{decomposizione}
[\lambda^{-1,\gamma}_{\chi}(Z),\mathcal L_\gamma+\rho_\sharp]u=[\lambda^{-1,\gamma}_{\chi}(Z),A_1^1\partial_1]u+[\lambda^{-1,\gamma}_{\chi}(Z),\mathcal L_{tan, \gamma}+\rho_\sharp]u\,.
\end{equation}
Note that $\mathcal L_{tan,\gamma}+\rho_\sharp$ is just a conormal operator of order $1$, according to the terminology introduced in Section \ref{conormalcalculus}.
\subsubsection{The tangential commutator}\label{commutatoretangenziale}
 Concerning the {\it tangential commutator} term $[\lambda^{-1,\gamma}_{\chi}(Z),\mathcal L_{tan, \gamma}+\rho_\sharp]u$, we use the identity
\begin{equation}\label{identita1}
\lambda^{1,\gamma}(Z)\lambda^{-1,\gamma}(Z)=Id\,,
\end{equation}
and formula \eqref{splitting} to rewrite it as follows
\begin{equation}\label{splitting1}
\begin{array}{ll}
\displaystyle[\lambda^{-1,\gamma}_{\chi}(Z),\mathcal L_{tan, \gamma}+\rho_\sharp]u=[\lambda^{-1,\gamma}_{\chi}(Z),\mathcal L_{tan, \gamma}+\rho_\sharp]\lambda^{1,\gamma}(Z)\lambda^{-1,\gamma}(Z)u\\
\\
\displaystyle\quad =[\lambda^{-1,\gamma}_{\chi}(Z),\mathcal L_{tan, \gamma}+\rho_\sharp]\lambda^{1,\gamma}(Z)\left(\lambda^{-1,\gamma}_\chi(Z)u\right)+[\lambda^{-1,\gamma}_{\chi}(Z),\mathcal L_{tan, \gamma}+\rho_\sharp]\lambda^{1,\gamma}(Z)r_{-1}(Z,\gamma)u\,.
\end{array}
\end{equation}
Since $\lambda^{-1,\gamma}_\chi$ is a scalar symbol, from the symbolic calculus (see Proposition \ref{prodottoecommutatore}) we know that
\begin{equation}\label{rho0tan}
\rho_{0, tan}(x,Z,\gamma):=[\lambda^{-1,\gamma}_{\chi}(Z),\mathcal L_{tan, \gamma}+\rho_\sharp]\lambda^{1,\gamma}(Z)
\end{equation}
is a conormal pseudo-differential operator with symbol $\rho_{0, tan}(x,\xi,\gamma)\in\Gamma^0$. Hence, the first term in the decomposition provided by \eqref{splitting1} can be regarded as an additional lower order term with respect to $\lambda^{-1,\gamma}_\chi(Z)u$, besides $\rho_\sharp(\lambda^{-1,\gamma}_\chi(Z)u)$, in the equation \eqref{sistemaregolarizzato} (see formula \eqref{tilderho}). On the other hand, from Lemma \ref{lemmatecnico2}, the second term in the decomposition \eqref{splitting1}
\begin{equation}\label{R-1}
R_{-1}(x,Z,\gamma)u:=[\lambda^{-1,\gamma}_{\chi}(Z),\mathcal L_{tan, \gamma}+\rho_\sharp]\lambda^{1,\gamma}(Z)r_{-1}(Z,\gamma)u
\end{equation}
can be moved to the right-hand side of the equation \eqref{sistemaregolarizzato} and treated as a part of the source term (see Section \ref{sectintcomm}).
\subsubsection{The normal commutator}\label{commutatorenormale}
We consider now the {\it normal commutator} term $[\lambda^{-1,\gamma}_{\chi}(Z),A_1^1\partial_1]u$ involved in \eqref{decomposizione}. With respect to the tangential term studied in Section \ref{commutatoretangenziale}, here the analysis is  little more technical.
\newline
First of all, we notice that, due to the structure of the matrix $A^1_1$ (see \eqref{strutturaA1}), the commutator $[\lambda^{-1,\gamma}_{\chi}(Z),A_1^1\partial_1]$ acts non trivially only on the noncharacteristic component of the vector function $u$; namely we have:
\begin{equation}\label{componenti}
[\lambda^{-1,\gamma}_{\chi}(Z),A_1^1\partial_1]u=
\begin{pmatrix}
[\lambda^{-1,\gamma}_{\chi}(Z),A^{I,I}_1\partial_1]u^I\\
0
\end{pmatrix}\,.
\end{equation}
Therefore, we focus on the study of the first nontrivial component of the commutator term. Note that the commutator $[\lambda^{-1,\gamma}_{\chi}(Z),A_1^{I,I}\partial_1]$ cannot be merely treated by the tools of the conormal calculus developed in Section \ref{conormalcalculus}, because of the presence of the effective normal derivative $\partial_1$ (recall that $A_1^{I,I}$ is invertible).
This section is devoted to the study of the normal commutator $[\lambda^{-1,\gamma}_{\chi}(Z),A^{I,I}_1\partial_1]u^{I}$. The following result is of fundamental importance for the sequel. Here again, for the sake of generality, the result is given with a general order $m$.
\begin{proposition}\label{normale}
For all $m\in\mathbb{R}$, there exists a symbol $q_{m}(x,\xi,\gamma)\in\Gamma^{m-1}$ such that
\begin{equation}\label{identitacommutatore}
[\lambda^{m,\gamma}_{\chi}(Z),A^{I,I}_1\partial_1]w=q_{m}(x,Z,\gamma)(\partial_1 w)\,,\quad\forall\,w\in C^{\infty}_{(0)}(\mathbb{R}^n_+)\,,\,\,\forall\,\gamma\ge 1\,.
\end{equation}
\end{proposition}
\noindent
{\it Proof}.\,\,The proof follows the same lines of that of \cite[Proposition 4.8]{moseBVP}.
\newline
For given $w\in C^{\infty}_{(0)}(\mathbb{R}^n_+)$, let us explicitly compute $\left([\lambda^{m,\gamma}_{\chi}(Z),A^{I,I}_1\partial_1]w\right)^{\sharp}$; using the identity $(\partial_1w)^{\sharp}=e^{-x_1}(Z_1w)^{\sharp}$ and that $\lambda^{m,\gamma}_{\chi}(Z)$ and $Z_1$ commute, we find for every $x\in\mathbb{R}^n$:
\begin{equation}
\begin{array}{ll}
\displaystyle{\left([\lambda^{m,\gamma}_{\chi}(Z),A^{I,I}_1\partial_1]w\right)^{\sharp}(x)}\\
\displaystyle{=\left(\lambda^{m,\gamma}_\chi(Z)\left(A^{I,I}_1\partial_1w\right)-A^{I,I}_1\partial_1\left(\lambda^{m,\gamma}_\chi(Z)w\right)\right)^{\sharp}(x)}\\
\displaystyle{=\lambda^{m,\gamma}_\chi(D)\left(A^{I,I,\natural}_1\left(\partial_1w\right)^{\sharp}\right)(x)-A^{I,I,\natural}_1(x)\left(\partial_1\left(\lambda^{m,\gamma}_\chi(Z)w\right)\right)^{\sharp}(x)}\\
\displaystyle{=\lambda^{m,\gamma}_{\chi}(D)\left(A^{I,I,\natural}_1 e^{-(\cdot)_1}(Z_1w)^{\sharp}\right)(x)
-A^{I,I,\natural}_1(x)e^{-x_1}\left(Z_1 \lambda^{m,\gamma}_{\chi}(Z)w\right)^{\sharp}(x)}\\
\displaystyle{=\lambda^{m,\gamma}_{\chi}(D)\left(A^{I,I,\natural}_1 e^{-(\cdot)_1}(Z_1w)^{\sharp}\right)(x)
-A^{I,I,\natural}_1(x)e^{-x_1}\left(\lambda^{m,\gamma}_{\chi}(Z)Z_1 w\right)^{\sharp}(x)}\\
\displaystyle{=\lambda^{m,\gamma}_{\chi}(D)\left(A^{I,I,\natural}_1 e^{-(\cdot)_1}(Z_1w)^{\sharp}\right)(x)
-A^{I,I,\natural}_1(x)e^{-x_1}\lambda^{m,\gamma}_{\chi}(D)(Z_1 w)^{\sharp}(x)\,.}
\end{array}
\end{equation}
Observing that $\lambda^{m,\gamma}_{\chi}(D)$ acts on the space $\mathcal{S}(\mathbb{R}^n)$ as the convolution by the inverse Fourier transform of $\lambda^{m,\gamma}_{\chi}$ (see \eqref{coeffcost}), the preceding expression can be equivalently restated as follows:
\begin{equation}\label{identitacommutatore1}
\begin{array}{ll}
\displaystyle{\left([\lambda^{m,\gamma}_{\chi}(Z),A^{I,I}_1\partial_1]w\right)^{\sharp}(x)}\\
\displaystyle{=\mathcal{F}^{-1}\left(\lambda^{m,\gamma}_{\chi}\right)\ast A^{I,I,\natural}_1 e^{-(\cdot)_1}(Z_1w)^{\sharp}(x)
-A^{I,I,\natural}_1(x)e^{-x_1}\mathcal{F}^{-1}\left(\lambda^{m,\gamma}_{\chi}\right)\ast(Z_1w)^{\sharp}}\\
\displaystyle{=\left\langle\mathcal{F}^{-1}\left(\lambda^{m,\gamma}_{\chi}\right)\,,\,A^{I,I,\natural}_1(x-\cdot)e^{-(x_1-(\cdot)_1)}(Z_1w)^{\sharp}(x-\cdot)\right\rangle}\displaystyle{-A^{I,I,\natural}_1(x)e^{-x_1}\langle\mathcal{F}^{-1}\left(\lambda^{m,\gamma}_{\chi}\right)\,,\,(Z_1w)^{\sharp}(x-\cdot)\rangle}\\
\displaystyle{=\left\langle\eta^{m,\gamma}\,,\,\chi(\cdot)A^{I,I,\natural}_1(x-\cdot)e^{-(x_1-(\cdot)_1)}(Z_1w)^{\sharp}(x-\cdot)\right\rangle}\displaystyle{-\left\langle\eta^{m,\gamma}\,,\,\chi(\cdot)A^{I,I,\natural}_1(x)e^{-x_1}(Z_1w)^{\sharp}(x-\cdot)\right\rangle}\\
\displaystyle{=\left\langle\eta^{m,\gamma}\,,\,\chi(\cdot)A^{I,I,\natural}_1(x-\cdot)(\partial_1w)^{\sharp}(x-\cdot)\right\rangle-\left\langle\eta^{m,\gamma}\,,\,\chi(\cdot)A^{I,I,\natural}_1(x)e^{-(\cdot)_1}(\partial_1w)^{\sharp}(x-\cdot)\right\rangle}\\
\displaystyle{=\left\langle\eta^{m,\gamma}\,,\,\chi(\cdot)\left(A^{I,I,\natural}_1(x-\cdot)-A^{I,I,\natural}_1(x)e^{-(\cdot)_1}\right)(\partial_1w)^{\sharp}(x-\cdot)\right\rangle\,,}
\end{array}
\end{equation}
where $\eta^{m,\gamma}:=\mathcal{F}^{-1}\left(\lambda^{m,\gamma}\right)$, and the identity $\mathcal{F}^{-1}\left(\lambda^{m,\gamma}_{\chi}\right)=\chi\eta^{m,\gamma}$ (following at once from \eqref{funzionepesomodificata}) has been used. Just for brevity, let us further set
\begin{equation}\label{nucleo}
\mathcal{K}(x,y):=\left(A^{I,I,\natural}_1(x-y)-A^{I,I,\natural}_1(x)e^{-y_1}\right)\chi(y)\,.
\end{equation}
Thus the identity above reads as
\begin{equation}\label{identitacommutatore2}
\left([\lambda^{m,\gamma}_{\chi}(Z),A^{I,I}_1\partial_1]w\right)^{\sharp}(x)=\left\langle\eta^{m,\gamma}\,,\,\mathcal{K}(x,\cdot)(\partial_1w)^{\sharp}(x-\cdot)\right\rangle\,,
\end{equation}
where the ``kernel'' $\mathcal{K}(x,y)$ is a bounded function in $C^{\infty}(\mathbb{R}^n\times\mathbb{R}^n)$, with bounded derivatives of all orders.
This regularity of $\mathcal{K}$ is due to the presence of the function $\chi$ in formula \eqref{nucleo}; actually the vanishing of $\chi$ at infinity prevents the blow-up of the exponential factor $e^{-y_1}$, as $y_1\to -\infty$. We point out that this is precisely the step of our analysis of the normal commutator, where this function $\chi$ is needed.
\newline
After \eqref{nucleo}, we also have that $\mathcal{K}(x,0)=0$; then, by a Taylor expansion with respect to $y$, we can represent the kernel $\mathcal{K}(x,y)$ as follows
\begin{equation}\label{taylor}
\mathcal{K}(x,y)=\sum\limits_{k=1}^n b_k(x,y)y_k\,,
\end{equation}
where $b_k(x,y)$ are given bounded functions in $C^{\infty}(\mathbb{R}^n\times\mathbb{R}^n)$, with bounded derivatives; it comes from \eqref{nucleo} and \eqref{chi} that $b_k$ can be defined in such a way that for all $x\in\mathbb{R}^n$ there holds
\begin{equation}\label{supportobk}
{\rm supp}\,b_k(x,\cdot)\subseteq\{|y|\le 2\varepsilon_0\}\,.\footnote{This can be made by multiplying $\mathcal{K}(x,y)$ by a suitable cut off function $\varphi=\varphi(y)\in C^\infty_0(\mathbb{R}^n)$ such that $\varphi(y)=1$ for $|y|\le 2\varepsilon_0$. This multiplication does not modify $\mathcal{K}$, since $\mathcal{K}$ is supported on $\{|y|\le\varepsilon_0\}$ with respect to $y$. Thus the equality \eqref{taylor} still holds, where the functions $b_k(x,y)$ are replaced by $b_k(x,y)\varphi(y)$ compactly supported with respect to $y$.}
\end{equation}
Inserting \eqref{taylor} in \eqref{identitacommutatore2} and using standard properties of the Fourier transform we get
\begin{equation}\label{identitacommutatore3}
\begin{array}{ll}
\displaystyle{\left([\lambda^{m,\gamma}_{\chi}(Z),A^{I,I}_1\partial_1]w\right)^{\sharp}(x)=\left\langle\eta^{m,\gamma}\,,\,\sum\limits_{k=1}^n b_k(x,\cdot)(\cdot)_k(\partial_1w)^{\sharp}(x-\cdot)\right\rangle}\\
\displaystyle{=\sum\limits_{k=1}^n\left\langle(\cdot)_k\mathcal{F}^{-1}\left(\lambda^{m,\gamma}\right)\,,\,b_k(x,\cdot)(\partial_1w)^{\sharp}(x-\cdot)\right\rangle}\\
\displaystyle{=i\sum\limits_{k=1}^n\left\langle\mathcal{F}^{-1}\left(\partial_{k}\lambda^{m,\gamma}\right)\,,\,b_k(x,\cdot)(\partial_1w)^{\sharp}(x-\cdot)\right\rangle}\\
\displaystyle{=i\sum\limits_{k=1}^n\left\langle \partial_{k}\lambda^{m,\gamma}\,,\,\mathcal{F}^{-1}\left(b_k(x,\cdot)(\partial_1w)^{\sharp}(x-\cdot)\right)\right\rangle}\\
\displaystyle{=i\sum\limits_{k=1}^n \int_{\mathbb{R}^n}\partial_{k}\lambda^{m,\gamma}(\xi)\mathcal{F}^{-1}\big(b_k(x,\cdot)(\partial_1w)^{\sharp}(x-\cdot)\big)(\xi)d\xi}\\
\displaystyle{=i\sum\limits_{k=1}^n (2\pi)^{-n}\int_{\mathbb{R}^n}\partial_{k}\lambda^{m,\gamma}(\xi)\left(\int_{\mathbb{R}^n}e^{i\xi\cdot y}b_k(x,y)(\partial_1w)^{\sharp}(x-y)dy\right)d\xi\,.}\\
\end{array}
\end{equation}
Note that for $w\in C^{\infty}_{(0)}(\mathbb{R}^n_+)$ and any $x\in\mathbb{R}^n$ the function $b_k(x,\cdot)(\partial_1w)^{\sharp}(x-\cdot)$ belongs to $\mathcal{S}(\mathbb{R}^n)$; hence the last expression in \eqref{identitacommutatore3} makes sense. Henceforth, we replace $(\partial_1w)^{\sharp}$ by any function $v\in\mathcal{S}(\mathbb{R}^n)$. Our next goal is writing the integral  operator
\begin{equation}\label{operatoreintegrale}
(2\pi)^{-n}\int_{\mathbb{R}^n}\partial_{k}\lambda^{m,\gamma}(\xi)\left(\int_{\mathbb{R}^n}e^{i\xi\cdot y}b_k(x,y)v(x-y)dy\right)d\xi
\end{equation}
as a pseudo-differential operator.
\newline
Firstly, we make use of the inversion formula for the Fourier transformation and Fubini's theorem to recast \eqref{operatoreintegrale} as follows:
\begin{equation}\label{operatorebk}
\begin{array}{ll}
\displaystyle{\int_{\mathbb{R}^n}e^{i\xi\cdot y}b_k(x,y)v(x-y)dy}
\displaystyle{=(2\pi)^{-n}\int_{\mathbb{R}^n}e^{i\xi\cdot y}b_k(x,y)\left(\int_{\mathbb{R}^n}e^{i(x-y)\cdot\eta}\widehat{v}(\eta)d\eta\right)dy}\\
\displaystyle{=(2\pi)^{-n}\int_{\mathbb{R}^n}e^{i x\cdot\eta}\left(\int_{\mathbb{R}^n}e^{-iy\cdot(\eta-\xi)}b_k(x,y)dy\right)\widehat{v}(\eta)d\eta}
\displaystyle{=(2\pi)^{-n}\int_{\mathbb{R}^n}e^{i x\cdot\eta}\widehat{b}_k(x,\eta-\xi)\widehat{v}(\eta)d\eta\,;}
\end{array}
\end{equation}
for every index $k$, $\widehat{b}_k(x,\zeta)$ denotes the partial Fourier transform of $b_k(x,y)$ with respect to $y$. Then, inserting \eqref{operatorebk} into \eqref{operatoreintegrale} we obtain
\begin{equation}\label{operatoreintegrale1}
\begin{array}{ll}
\displaystyle{(2\pi)^{-n}\int_{\mathbb{R}^n}\partial_{k}\lambda^{m,\gamma}(\xi)\left(\int_{\mathbb{R}^n}e^{i\xi\cdot y}b_k(x,y)v(x-y)dy\right)d\xi}\\
\displaystyle{=(2\pi)^{-2n}\int_{\mathbb{R}^n}\partial_{k}\lambda^{m,\gamma}(\xi)\left(\int_{\mathbb{R}^n}e^{i x\cdot\eta}\widehat{b}_k(x,\eta-\xi)\widehat{v}(\eta)d\eta\right)d\xi\,.}
\end{array}
\end{equation}
Recall that for each $x\in\mathbb{R}^n$, the function $y\mapsto b_k(x,y)$ belongs to $C^{\infty}_{0}(\mathbb{R}^n)$ (and its compact support does not depend on $x$, see \eqref{supportobk}); thus, for each $x\in\mathbb{R}^n$, $\widehat{b}_k(x,\zeta)$ is rapidly decreasing in $\zeta$.
\newline
Because $\lambda^{m,\gamma}\in\Gamma^m$ and since $\widehat{v}(\eta)$ is also rapidly decreasing, Fubini's theorem can be used to change the order of the integrations within \eqref{operatoreintegrale1}. So we get
\begin{equation}\label{operatoreintegrale2}
\begin{array}{ll}
\displaystyle{(2\pi)^{-2n}\int_{\mathbb{R}^n}\partial_{k}\lambda^{m,\gamma}(\xi)\left(\int_{\mathbb{R}^n}e^{i x\cdot\eta}\widehat{b}_k(x,\eta-\xi)\widehat{v}(\eta)d\eta\right)d\xi}\\
\displaystyle{=(2\pi)^{-2n}\int_{\mathbb{R}^n}e^{i x\cdot\eta}\left(\int_{\mathbb{R}^n}\widehat{b}_k(x,\eta-\xi)\partial_{k}\lambda^{m,\gamma}(\xi)d\xi\right)\widehat{v}(\eta)d\eta}\\
\displaystyle{=(2\pi)^{-n}\int_{\mathbb{R}^n}e^{i x\cdot\eta}q_{k,m}(x,\eta,\gamma)\widehat{v}(\eta)d\eta\,,}
\end{array}
\end{equation}
where we have set
\begin{equation}\label{simbolonormalek}
q_{k,m}(x,\xi,\gamma):=(2\pi)^{-n}\int_{\mathbb{R}^n}\widehat{b}_k(x,\eta)\partial_{k}\lambda^{m,\gamma}(\xi-\eta)d\eta\,.
\end{equation}
Notice that formula \eqref{simbolonormalek} defines $q_{k,m}$ as the convolution of the functions $\widehat{b}_k(x,\cdot)$ and $\partial_{k}\lambda^{m,\gamma}$; hence $q_{k,m}$ is a well defined $C^{\infty}-$function in $\mathbb{R}^n\times\mathbb{R}^n$.
\newline
The proof of Proposition \ref{normale} will be accomplished, once the following Lemma will be proved.
\begin{lemma}\label{stimesimbolonormale}
For every $m\in\mathbb{R}$, $k=1,\dots,n$, $q_{k,m}\in\Gamma^{m-1}$, i.e. for all $\alpha,\beta\in\mathbb{N}^n$ there exists a positive constant $C_{k,m,\alpha,\beta}$, independent of $\gamma$, such that
\begin{equation}\label{stimeqk}
|\partial^{\alpha}_{\xi}\partial^{\beta}_{x}q_{k,m}(x,\xi,\gamma)|\le C_{k,m,\alpha,\beta}\lambda^{m-1-|\alpha|,\gamma}(\xi)\,,\quad\forall\,x,\,\xi\in\mathbb{R}^n\,.
\end{equation}
\end{lemma}
\noindent
The proof of Lemma \ref{stimesimbolonormale} is postponed to Appendix \ref{appA}.
\newline
Now, we continue the proof of Proposition \ref{normale}

\vspace{.2cm}
\noindent
{\it End of the proof of Proposition \ref{normale}.}\,\,The last row of \eqref{operatoreintegrale2} provides the desired representation of \eqref{operatoreintegrale} as a pseudo-differential operator; actually it gives the identity
\begin{equation*}
(2\pi)^{-n}\int_{\mathbb{R}^n}\!\! \partial_{k}\lambda^{m,\gamma}(\xi)\left(\int_{\mathbb{R}^n}\!\! e^{i\xi\cdot y}b_k(x,y)v(x-y)dy\right)d\xi={\rm Op}^{\gamma}(q_{k,m})v(x)\,,
\end{equation*}
for every $v\in\mathcal{S}(\mathbb{R}^n)$.
\newline
Inserting the above formula (with $v=(\partial_1 w)^{\sharp}$) into \eqref{identitacommutatore3} finally gives
\begin{equation}\label{rappresentazionefinale}
\begin{array}{ll}
\left([\lambda^{m,\gamma}_{\chi}(Z),A^{I,I}_1\partial_1]w\right)^{\sharp}(x)={\rm Op}^{\gamma}(q_{m})(\partial_1 w)^{\sharp}(x)\,,
\end{array}
\end{equation}
where $q_{m}=q_{m}(x,\xi,\gamma)$ is the symbol in $\Gamma^{m-1}$ defined by
\begin{equation}\label{simbolonormalefinale}
q_{m}(x,\xi,\gamma):=i\sum\limits_{k=1}^n q_{k,m}(x,\xi,\gamma)\,.
\end{equation}
Of course, formula \eqref{identitacommutatore} is equivalent to \eqref{rappresentazionefinale}, in view of \eqref{operatoreconormale}. This ends the proof of Proposition \ref{normale}.\qed

\vspace{.2cm}
\noindent
We come back to the analysis of the normal commutator term $[\lambda^{-1,\gamma}_{\chi}(Z),A^{I,I}_1\partial_1]u^{I}$. To estimate it, we apply the result of Proposition \ref{normale} for $m=-1$ and $w=u^I$. Then we find the representation formula
\begin{equation}\label{forma1commutatore}
[\lambda^{-1,\gamma}_{\chi}(Z),A^{I,I}_1\partial_1]u^I=q_{-1}(x,Z,\gamma)(\partial_1 u^I)\,,
\end{equation}
where the symbol $q_{-1}\in\Gamma^{-2}$ is defined by \eqref{simbolonormalefinale}. Since $A^{I,I}_1$ is invertible, from \eqref{tf}, $\partial_1u^I$ can be represented in terms of tangential derivatives of $u$ and $F$, as follows
\begin{equation}\label{partial_1}
\partial_1u^I=(A^{I,I}_1)^{-1}F^I+\mathcal{T}_\gamma u\,,
\end{equation}
where $\mathcal{T}_\gamma=\mathcal{T}_\gamma(x,Z)$ denotes the tangential partial differential operator
\begin{equation}\label{normtan}
\mathcal{T}_\gamma u:=-(A^{I,I}_1)^{-1}\left[\gamma u^I+H_1Z_1u^{II}+\left(\sum\limits_{j=2}^n A_jZ_ju+Bu+\rho_\sharp u\right)^{I}\right]
\end{equation}
and we have set $H_1:=x_1^{-1}A^{I,II}_1$ (recall that $H_1\in C^{\infty}_{(0)}(\mathbb{R}^n_+)$ since $A^{I,II}_{1\,|\,x_1=0}=0$). Inserting \eqref{partial_1} into \eqref{forma1commutatore} leads to
\begin{equation}\label{ordinezero1}
[\lambda^{-1,\gamma}_{\chi}(Z),A^{I,I}_1\partial_1]u^{I}=q_{-1}(x,Z,\gamma)((A^{I,I}_1)^{-1}F^I)+q_{-1}(x,Z,\gamma)\mathcal{T}_\gamma u\,.
\end{equation}
The first term in the right-hand side of \eqref{ordinezero1} is moved to the right-hand side of equation \eqref{sistemaregolarizzato} and incorporated into the source term. As for the second term $q_{-1}(x,Z,\gamma)\mathcal{T}_\gamma u$, a similar analysis to the one performed about the tangential commutator term in the right-hand side of \eqref{decomposizione} can  be applied to rewrite it as the sum of a lower order operator acting on $\lambda^{-1,\gamma}_\chi(Z)u$ and some smoothing reminder. More precisely, applying again the identities \eqref{identita1} and \eqref{splitting} we get
\begin{equation}\label{splitting2}
q_{-1}(x,Z,\gamma)\mathcal{T}_\gamma u=q_{-1}(x,Z,\gamma)\mathcal{T}_\gamma\lambda^{1,\gamma}(Z)\left(\lambda^{-1,\gamma}_\chi(Z)u\right)+q_{-1}(x,Z,\gamma)\mathcal{T}_\gamma\lambda^{1,\gamma}(Z)r_{-1}(Z,\gamma)u\,.
\end{equation}
Combining \eqref{ordinezero1}, \eqref{splitting2} and \eqref{componenti} we decompose the normal commutator term in \eqref{decomposizione} as the sum of the following contributions
\begin{equation}\label{splitting3}
[\lambda^{-1,\gamma}_{\chi}(Z),A_1^1\partial_1]u=\left(\begin{array}{cc}q_{-1}(x,Z,\gamma)((A^{I,I}_1)^{-1}F^I)\\ 0\end{array}\right)+\rho_{0, nor}(x,Z,\gamma)(\lambda^{-1,\gamma}_{\chi}(Z)u)+S_{-1}(x,Z,\gamma)u\,.
\end{equation}
In the representation provided by \eqref{splitting3}, the conormal operator
\begin{equation}\label{rho0nor}
\rho_{0, nor}(x, Z,\gamma):=\left(\begin{array}{cc}q_{-1}(x,Z,\gamma)\mathcal{T}_\gamma\lambda^{1,\gamma}(Z)\\ 0\end{array}\right)
\end{equation}
has symbol in $\Gamma^0$(in view of Proposition \ref{prodottoecommutatore}), and hence it must be treated as an additional lower order operator, besides $\rho_\sharp$ and $\rho_{0, tan}$, within the equation \eqref{sistemaregolarizzato} (see \eqref{tilderho}); on the other hand
\begin{equation}\label{S-1}
S_{-1}(x,Z,\gamma)u:=\left(\begin{array}{cc}q_{-1}(x,Z,\gamma)\mathcal{T}_\gamma\lambda^{1,\gamma}(Z)r_{-1}(Z,\gamma)u\\ 0\end{array}\right)
\end{equation}
can be regarded as a smoothing reminder and then moved to the right-hand side of the equation \eqref{sistemaregolarizzato} to be treated as a part of the source term, in view of Lemma \ref{lemmatecnico2} (see Section \ref{sectintcomm}).
\subsubsection{The boundary condition}\label{bdrysect}
We are going to write a boundary condition to be coupled to \eqref{sistemaregolarizzato}.
\newline
Firstly we notice that, by Proposition \ref{bordo} for $m=-1$:
\begin{equation}\label{bordo-1}
(\lambda^{-1,\gamma}_{\chi}(Z)u)_{|\,x_1=0}=b'_{-1}(D',\gamma)(u_{|\,x_1=0})\,,
\end{equation}
where the symbol $b'_{-1}\in\Gamma^{-1}$ on $\mathbb{R}^{n-1}$ is defined by \eqref{simbolobordobis}. Then we apply the operator $b'_{-1}=b'_{-1}(D',\gamma)$ to \eqref{tb} and we obtain
\begin{equation}\label{bordoregolarizzato}
\begin{array}{ll}
b_\gamma(b'_{-1}\psi)+\mathcal M^s_\gamma\left(b'_{-1}u^{I,s}_{|\,x_1=0}\right)+M^I\left(b'_{-1}u^{I}_{|\,x_1=0}\right)+b_{\sharp}(b'_{-1}\psi)+\ell_\sharp\left(b'_{-1}u^{I,s}_{|\,x_1=0}\right)\\
\\
\,\,+[b'_{-1},b_\gamma]\psi+[b'_{-1},b_{\sharp}]\psi+[b'_{-1},\mathcal M^s_\gamma](u^{I,s}_{|\,x_1=0})+[b'_{-1},\ell_\sharp](u^{I,s}_{|\,x_1=0})+[b'_{-1},M^I](u^I_{|\,x_1=0})=b'_{-1}g\,,\,\,\text{on}\,\,\mathbb{R}^{n-1}\,.
\end{array}
\end{equation}
where, for simplicity, we have dropped the explicit dependence on $x'$, $D'$ and $\gamma$ in the operators.
We observe that, in view of the symbolic calculus (see Proposition \ref{prodottoecommutatore}), the commutators appearing above are all pseudo-differential operators on $\mathbb{R}^{n-1}$; more precisely, since $b'_{-1}(\xi',\gamma)$ is a scalar symbol we have that
\begin{equation}\label{commutatori-2}
\begin{array}{ll}
\lbrack b'_{-1},b_{\sharp}\rbrack =\lbrack b'_{-1}(D',\gamma),b_{\sharp}(x',D',\gamma)\rbrack \\
\\
\lbrack b'_{-1},\ell_{\sharp}\rbrack =\lbrack b'_{-1}(D',\gamma),\ell_{\sharp}(x',D',\gamma)\rbrack\\
\\
\lbrack b'_{-1},M^I\rbrack =\lbrack b'_{-1}(D',\gamma),M^I\rbrack
\end{array}
\end{equation}
are operators with symbol in $\Gamma^{-2}$, while
\begin{equation}\label{commutatori-1}
\begin{array}{ll}
\lbrack b'_{-1},b_{\gamma}\rbrack =\lbrack b'_{-1}(D',\gamma),b_{\gamma}(x',D')\rbrack \\
\\
\lbrack b'_{-1},\mathcal M^s_\gamma\rbrack =\lbrack b'_{-1}(D',\gamma),\mathcal M^s_\gamma(x',D')\rbrack
\end{array}
\end{equation}
are operators with symbol in $\Gamma^{-1}$.
\newline
Since the a priori estimate in assumption $(H)_1$ displays a loss of regularity from the boundary data, the above operators must be treated in two different ways. The commutators in \eqref{commutatori-2} can be moved to the right-hand side and treated as additional forcing terms. On the contrary, the commutators in \eqref{commutatori-1} cannot be regarded as a part of the source term in the equation \eqref{bordoregolarizzato} without loosing derivatives on the unknowns $u$ and $\psi$. These operators require a more careful analysis that essentially relies on similar arguments to those used to study the commutator term appearing in the interior equation \eqref{sistemaregolarizzato} (see Sections \ref{commutatoretangenziale}, \ref{commutatorenormale}).
\newline
We use Lemma \ref{lemmatecnico5} and the identity $\lambda^{1,\gamma}(D')\lambda^{-1,\gamma}(D')=Id$ to write
\begin{equation}\label{bordo0}
\begin{array}{ll}
\displaystyle{[b'_{-1}(D',\gamma),b_{\gamma}]\psi=[b'_{-1}(D',\gamma),b_{\gamma}]\lambda^{1,\gamma}(D')\lambda^{-1,\gamma}(D')\psi}\\\\
\displaystyle{=[b'_{-1}(D',\gamma),b_{\gamma}]\lambda^{1,\gamma}(D')(b'_{-1}(D',\gamma)-\beta_{-1}(D',\gamma))\psi}\\\\
\displaystyle{=\bigg([b'_{-1}(D',\gamma),b_{\gamma}]\lambda^{1,\gamma}(D')\bigg)(b'_{-1}(D',\gamma)\psi)-\bigg([b'_{-1}(D',\gamma),b_{\gamma}]\lambda^{1,\gamma}(D')\bigg)\beta_{-1}(D',\gamma)\psi}\\\\
\displaystyle{=d_0(x', D',\gamma)(b'_{-1}(D',\gamma)\psi)+d_{-3}(x', D',\gamma)\psi}\,,
\end{array}
\end{equation}
where
\begin{equation}\label{d0}
d_0(x', D',\gamma):=[b'_{-1}(D',\gamma),b_{\gamma}]\lambda^{1,\gamma}(D')
\end{equation}
has symbol in $\Gamma^0$ and
\begin{equation}\label{d-3}
d_{-3}(x', D',\gamma):==-[b'_{-1}(D',\gamma),b_{\gamma}]\lambda^{1,\gamma}(D')\beta_{-1}(D',\gamma)
\end{equation}
has symbol in $\Gamma^{-3}$, since $\beta_{-1}(\xi',\gamma)\in\Gamma^{-3}$.
\newline
Analogously, we can treat the term in $u$ involving the commutator $[b'_{-1}, \mathcal M^s_\gamma]$, namely we find:
\begin{equation}\label{bordoMs}
\begin{array}{ll}
\displaystyle{[b'_{-1}(D',\gamma),\mathcal M^s_{\gamma}]u^{I,s}_{|\,x_1=0}}
\displaystyle{=e_0(x', D',\gamma)\left(b'_{-1}(D',\gamma)u^{I,s}_{|\,x_1=0}\right)+e_{-3}(x', D',\gamma)u^{I,s}_{|\,x_1=0}}\,,
\end{array}
\end{equation}
where
\begin{equation}\label{e0}
e_0(x',D',\gamma):=[b'_{-1}(D',\gamma),\mathcal M^s_{\gamma}]\lambda^{1,\gamma}(D')
\end{equation}
has symbol in $\Gamma^0$ and
\begin{equation}\label{e-3}
e_{-3}(x',D',\gamma):=-[b'_{-1}(D',\gamma),\mathcal M^s_{\gamma}]\lambda^{1,\gamma}(D')\beta_{-1}(D',\gamma)
\end{equation}
has symbol in $\Gamma^{-3}$.
\newline
Thanks to the stability of the estimate \eqref{h1h2estimate} with respect to zero-th order terms in $\psi$ and $u^{I,s}$, the operators $d_0(x',D',\gamma)$ and $e_0(x',D',\gamma)$ in the representations \eqref{bordo0}, \eqref{bordoMs} can be just regarded as an additional lower order terms in $b'_{-1}(D',\gamma)\psi$ and $b'_{-1}(D',\gamma)u^{I,s}_{|\,x_1=0}$, together with $b_{\sharp}(x',D',\gamma)(b'_{-1}(D',\gamma)\psi)$, $\ell_{\sharp}(x',D',\gamma)\left(b'_{-1}(D',\gamma)u^{I,s}_{|\,x_1=0}\right)$ in the equation \eqref{bordoregolarizzato} (see formulas \eqref{tildeb}, \eqref{tildel} below). The terms involving $d_{-3}(x',D',\gamma)$, $e_{-3}(x',D',\gamma)$ can be just moved to the right-hand side of \eqref{bordoregolarizzato} and absorbed into the boundary datum (see \eqref{Gcal}).
\begin{remark}\label{rmktraccia}
Let us notice that in view of Proposition \ref{bordo} (and using that the operator $\lambda^{-1,\gamma}_\chi(Z)$ acts component-wise on functions) the following identities hold
\begin{equation}\label{traccia}
b'_{-1}(D',\gamma)\left(u^{I}_{|\,x_1=0}\right)=\left(\lambda^{-1,\gamma}_\chi(Z)u^{I}\right)_{|\,x_1=0}=\left(\lambda^{-1,\gamma}_\chi(Z)u\right)^{I}_{|\,x_1=0}
\end{equation}
and similarly for $u^{I,s}$.
\end{remark}
\subsubsection{Final form of the regularized BVP}
Summarizing the calculations performed in the previous Section \ref{sec_regBVP} and in view of Remark \ref{rmktraccia}, the functions $(\lambda^{-1,\gamma}_{\chi}(Z)u, b'_{-1}(D',\gamma)\psi)$ satisfy the system
\begin{equation}\label{fsystem}
\begin{cases}
\mathcal L_\gamma(\lambda^{-1,\gamma}_{\chi}(Z)u)+\tilde\rho(x,Z,\gamma)(\lambda^{-1,\gamma}_\chi(Z)u)=\mathcal{F}\,\,\quad{\rm in}\,\,\mathbb{R}^{n}_+\,\\
\\
\begin{array}{ll}
\displaystyle b_\gamma(b'_{-1}(D',\gamma)\psi)+\mathcal M^s_\gamma\left(\lambda^{-1,\gamma}_\chi(Z)u\right)^{I,s}_{|\,x_1=0}+M^I\left(\lambda^{-1,\gamma}_\chi(Z)u\right)^{I}_{|\,x_1=0}\\
\displaystyle\qquad +\tilde{b}(x',D',\gamma)(b'_{-1}(D',\gamma)\psi)+\tilde{\ell}(x',D',\gamma)\left(\lambda^{-1,\gamma}_\chi(Z)u\right)^{I,s}_{|\,x_1=0}=\mathcal{G}  \,\quad{\rm on}\,\,\mathbb{R}^{n-1}\,,
\end{array}
\end{cases}
\end{equation}
where
\begin{equation}\label{tilderho}
\tilde\rho(x,Z,\gamma):=\rho_\sharp(x,Z,\gamma)+\rho_{0, tan}(x, Z,\gamma)u+\rho_{0, nor}(x,Z,\gamma)u\,,
\end{equation}
\begin{equation}\label{tildeb}
\tilde b(x',D',\gamma):=b_{\sharp}(x',D',\gamma)+d_0(x',D',\gamma)\,,
\end{equation}
\begin{equation}\label{tildel}
\tilde \ell(x',D',\gamma):=\ell_{\sharp}(x',D',\gamma)+e_0(x',D',\gamma)\,,
\end{equation}
\begin{equation}\label{Fcal}
\begin{array}{ll}
\displaystyle\mathcal{F}:= \lambda^{-1,\gamma}_{\chi}(Z)F-\left(\begin{array}{cc}q_{-1}(x,Z,\gamma)((A^{I,I}_1)^{-1}F^I)\\ 0\end{array}\right)-R_{-1}(x,Z,\gamma)u-S_{-1}(x,Z,\gamma)u\,,
\end{array}
\end{equation}
\begin{equation}\label{Gcal}
\begin{array}{ll}
\displaystyle\mathcal{G}:=b'_{-1}(D',\gamma)g-[b'_{-1}(D',\gamma),b_{\sharp}(x',D',\gamma)]\psi-d_{-3}(x',D',\gamma)\psi\\
\\
\displaystyle\qquad -[b'_{-1}(D',\gamma),M^I](u^I_{|\,x_1=0})-[b'_{-1}(D',\gamma),\ell_\sharp(x',D',\gamma)](u^{I,s}_{|\,x_1=0})-e_{-3}(x',D',\gamma)u^{I,s}_{|\,x_1=0}\,,
\end{array}
\end{equation}
and the operators $\rho_{0, tan}$, $\rho_{0, nor}$, $d_0$, $e_0$, $R_{-1}$, $S_{-1}$, $d_{-3}$, $e_{-3}$ are defined in the preceding Sections \ref{interioreqt}-\ref{bdrysect}.
\subsection{The estimate associated to the regularized BVP}
From assumption $(H)_1$, we know that there exist constants $C_0>0$, $\gamma_0\ge 1$, depending only on the coefficients of the operator $\mathcal L_\gamma$ and a finite number of semi-norms of $\tilde\rho=\tilde\rho(x,\xi,\gamma)\in\Gamma^0$, $\tilde\ell=\tilde\ell(x',\xi',\gamma), \tilde b=\tilde b(x',\xi',\gamma)\in\Gamma^0$, such that for all $\gamma\ge\gamma_0$ the following estimate holds for the functions $(\lambda^{-1,\gamma}_{\chi}(Z)u, b'_{-1}(D',\gamma)\psi)$
\begin{equation}\label{lambdaestimate}
    \begin{array}{ll}
    \displaystyle{\gamma\left(||\lambda^{-1,\gamma}_\chi(Z)u||^2_{H^1_{ tan\,,\gamma}(\mathbb{R}^n_+)}+||(\lambda^{-1,\gamma}_\chi(Z)u^I)_{|\,x_1=0}||^2_{H^{1/2}_\gamma(\mathbb{R}^{n-1})}\right)+\gamma^2||b'_{-1}(D',\gamma)\psi||^2_{H^1_\gamma(\mathbb{R}^{n-1})}}
    \\\\
    \quad\quad\displaystyle{\le C_0\left(\frac1{\gamma^3}||\mathcal{F}||^2_{H^{2}_{tan,\gamma}(\mathbb{R}^n_+)}+\frac1{\gamma}||\mathcal{G}||^2_{H^{3/2}_{\gamma}(\mathbb{R}^{n-1})}\right)\,.}
    \end{array}
    \end{equation}
We start analyzing the terms appearing in the left-hand side of \eqref{lambdaestimate}.
\newline
In view of \eqref{normalambda}, \eqref{splitting} we compute
\begin{equation*}
\begin{array}{ll}
\displaystyle{||\lambda^{-1,\gamma}_\chi(Z)u||_{H^1_{ tan\,,\gamma}(\mathbb{R}^n_+)}=||\lambda^{1,\gamma}(Z)\,\lambda^{-1,\gamma}_\chi(Z)u ||_{L^2(\mathbb{R}^n_+)}= \left|\left|\lambda^{1,\gamma}(Z)\bigg(\,\lambda^{-1,\gamma}(Z) - r_{-1}(Z,\gamma)\bigg)u \right|\right|_{L^2(\mathbb{R}^n_+)}}\\
\\
= \left|\left|u  - \lambda^{1,\gamma}(Z)\, r_{-1}(Z,\gamma)u \right|\right|_{L^2(\mathbb{R}^n_+)}\geq ||u||_{L^2(\mathbb{R}^n_+)} - || \lambda^{1,\gamma}(Z)\,r_{-1}(Z,\gamma)u ||_{L^2(\mathbb{R}^n_+)}\\
\\ =  ||u||_{L^2(\mathbb{R}^n_+)} - || r_{-1}(Z,\gamma)u ||_{H^1_{ tan\,,\gamma}(\mathbb{R}^n_+)}.
\end{array}
\end{equation*}
Using Lemma \ref{lemmatecnico2} with $h=1$, there exists a constant $C_{1}$, independent on $\gamma$, such that
$$
|| r_{-1}(Z,\gamma)u ||_{H^1_{ tan\,,\gamma}(\mathbb{R}^n_+)}\leq\frac{C_{1}}{\gamma}||u||_{L^2(\mathbb{R}^n_+)}\,, \quad \forall\gamma\geq 1.
$$
Hence
\begin{equation}\label{l2lambda}
 ||\lambda^{-1,\gamma}_\chi(Z)u||_{H^1_{ tan\,,\gamma}(\mathbb{R}^n_+)}\geq  ||u||_{L^2(\mathbb{R}^n_+)} -\frac{C_{1}}{\gamma}||u||_{L^2(\mathbb{R}^n_+)}\geq \frac1{2} ||u||_{L^2(\mathbb{R}^n_+)}\,, \quad \forall\gamma\geq \gamma_1
\end{equation}
with large enough $\gamma_1\geq 1$.
\newline
Using  Proposition \ref{bordo} and  Lemma \ref{lemmatecnico5} we get
$$
(\lambda^{-1,\gamma}_\chi(Z)u^I)_{|\,x_1=0}=b'_{-1}(D', \gamma)(u^I_{|\, x_1=0})= \lambda^{-1,\gamma}(D')(u^I_{|\, x_1=0}) + \beta_{-1}(D',\gamma)(u^I_{|\, x_1=0}).
$$
Again by Lemma \ref{lemmatecnico5} we derive that $\beta_{-1}(\xi',\gamma)\in \Gamma^{-3}$, hence by Proposition \ref{continuitasobolev} and \eqref{gammaimbedding},  we get
\begin{equation}\label{l2traccia}
\begin{array}{ll}
||(\lambda^{-1,\gamma}_\chi(Z)u^I)_{|\,x_1=0}||_{H^{1/2}_\gamma(\mathbb{R}^{n-1})}=
|| \lambda^{-1,\gamma}(D')(u^I_{|\, x_1=0}) + \beta_{-1}(D',\gamma)(u^I_{|\, x_1=0})||_{H^{1/2}_\gamma(\mathbb{R}^{n-1}) }\\
\\\geq ||\lambda^{1/2,\gamma}(D')  \lambda^{-1,\gamma}(D')(u^I_{|\, x_1=0})||_{L^2(\mathbb{R}^n_+)} - || \beta_{-1}(D',\gamma)(u^I_{|\, x_1=0})||_{H^{1/2}_\gamma(\mathbb{R}^{n-1})}\\
\\ \geq || u^I_{|\, x_1=0}||_{H^{-1/2}_\gamma(\mathbb{R}^{n-1})} - C|| u^I_{|\, x_1=0}||_{H^{-5/2}_\gamma(\mathbb{R}^{n-1})}\\
\\ \displaystyle{\geq \left(1-\frac{C}{\gamma^2}\right)|| u^I_{|\, x_1=0}||_{H^{-1/2}_\gamma(\mathbb{R}^{n-1})}\geq \frac1{2}|| u^I_{|\, x_1=0}||_{H^{-1/2}_\gamma(\mathbb{R}^{n-1})}, \quad \forall \gamma\geq \gamma_1}
\end{array}
\end{equation}
with large enough $\gamma_1\geq 1$, and $C$ a  positive constant independent of $\gamma$.
As regards to the term $||b'_{-1}(D',\gamma)\psi||^2_{H^1_\gamma(\mathbb{R}^{n-1})}$ in \eqref{lambdaestimate} we write again, by Lemma \ref{lemmatecnico5},
$$
b'_{-1}(D', \gamma)\psi= \lambda^{-1,\gamma}(D')\psi+\beta_{-1}(D',\gamma)\psi\,.
$$
Arguing as above we obtain
\begin{equation}\label{l2psi}
\begin{array}{ll}
||b'_{-1}(D', \gamma)\psi||_{H^1_\gamma(\mathbb{R}^{n-1})}\geq  ||\lambda^{-1,\gamma}(D')\psi||_{H^1_\gamma(\mathbb{R}^{n-1})}  -||\beta_{-1}(D',\gamma)\psi||_{H^1_\gamma(\mathbb{R}^{n-1})}\\
\\
\displaystyle{\geq ||\psi||_{L^2(\mathbb{R}^{n-1})} -C ||\psi||_{H^{-2}_\gamma(\mathbb{R}^{n-1})}\geq \left(1-\frac{C}{\gamma^2}\right) ||\psi||_{L^2(\mathbb{R}^{n-1})} \geq \frac{1}{2}
||\psi||_{L^2(\mathbb{R}^{n-1})}, \quad \forall \gamma\geq \gamma_1}
\end{array}
\end{equation}
with $\gamma_1\geq 1$ large enough, and $C$ a  positive constant independent on $\gamma$.
\newline
To conclude the estimate, we need to analyze the different commutator terms involved in the data $\mathcal{F}, \mathcal{G}$ in  right-hand side of \eqref{lambdaestimate}. The next two sections are devoted to the study of these commutator terms.
\subsubsection{The estimate of the internal source term $\mathcal F$}\label{sectintcomm}
To provide an estimate of the $H^2_{tan}-$norm of the source term $\mathcal F$ in the internal equation of the BVP \eqref{fsystem}, we need to estimate in $H^2_{tan}(\mathbb R^n_+)$ the different terms involving $F$ and the function $u$ in the right-hand side of \eqref{Fcal}.
\newline
Concerning the terms in the right-hand side of \eqref{Fcal} containing the function $u$, from Lemma \ref{lemmatecnico2} and the fact that the operators $[\lambda^{-1,\gamma}_{\chi}(Z),\mathcal L_{tan, \gamma}+\rho_\sharp]\lambda^{1,\gamma}(Z)$ and $q_{-1}(x,Z,\gamma)\mathcal{T}_\gamma\lambda^{1,\gamma}(Z)$ involved in the definition of $R_{-1}$, $S_{-1}$ are of order zero (see \eqref{R-1}, \eqref{S-1}), we get
\begin{equation}\label{stimeR1S1}
\begin{array}{ll}
\displaystyle||R_{-1}u||_{H^2_{tan, \gamma}(\mathbb R^n_+)}\le C||r_{-1}(Z,\gamma)u||_{H^2_{tan,\gamma}(\mathbb R^n_+)}\le C_1||u||_{L^2(\mathbb R^n_+)}\,,\\
\\
\displaystyle||S_{-1}u||_{H^2_{tan, \gamma}(\mathbb R^n_+)}\le C||r_{-1}(Z,\gamma)u||_{H^2_{tan,\gamma}(\mathbb R^n_+)}\le C_1||u||_{L^2(\mathbb R^n_+)}\,,
\end{array}
\end{equation}
for suitable positive constants $C, C_1$ independent of $\gamma\ge 1$.
\newline
As regards to the terms in the right-hand side of \eqref{Fcal} that contain the function $F$, since the operator $q_{-1}(x,Z,\gamma)$ has symbol in $\Gamma^{-2}$ (cf. Proposition \ref{normale}), we immediately find that
\begin{equation}\label{stimeFprimi}
\begin{array}{ll}
\displaystyle||\lambda^{-1,\gamma}_\chi(Z) F||_{H^2_{tan, \gamma}(\mathbb R^n_+)}\le C||F||_{H^1_{tan,\gamma}(\mathbb R^n_+)}\,,\\
\\
\displaystyle||q_{-1}(x,Z,\gamma)((A_1^{I, I})^{-1}F^{I})||_{H^2_{tan, \gamma}(\mathbb R^n_+)}\le C||F^I||_{L^2(\mathbb R^n_+)}\le\frac{C}{\gamma}||F^I||_{H^1_{tan, \gamma}(\mathbb R^n_+)}\,,
\end{array}
\end{equation}
for a suitable positive $C$, independent of $\gamma$.
\newline
Collecting estimates \eqref{stimeR1S1}, \eqref{stimeFprimi} we obtain
\begin{equation}\label{stimaFcal}
||\mathcal F||_{H^2_{tan, \gamma}(\mathbb R^n_+)}\le C\left\{||F||_{H^1_{tan, \gamma}(\mathbb R^n_+)}+||u||_{L^2(\mathbb R^n_+))}\right\}\,,
\end{equation}
where again $C$ is some positive constant independent of $\gamma$.
\subsubsection{The estimate of the boundary data $\mathcal G$}\label{commbordo}
In this section we provide an estimate of the $H^{3/2}-$norm of the boundary data $\mathcal G$ in the right-hand side of $\eqref{fsystem}_2$, as it is required by the estimate \eqref{lambdaestimate}; in particular, we need to consider the commutator terms involved in \eqref{Gcal}.
\newline
From Section \ref{bdrysect} we know that the commutators in \eqref{commutatori-2} are pseudo-differential with symbols in $\Gamma^{-2}$. Hence from Proposition \ref{continuitasobolev}, there exists a constant $C>0$ such that, $\forall\,\gamma\ge 1$,
\begin{equation}\label{commbordo1}
\begin{array}{ll}
\displaystyle{||[b'_{-1}(D',\gamma),b_{\sharp}(x',D',\gamma)]\psi||_{H^{3/2}_{\gamma}(\mathbb{R}^{n-1})}\le C||\psi||_{H^{-1/2}_{\gamma}(\mathbb{R}^{n-1})}\le\frac{C}{\gamma^{1/2}}||\psi||_{L^2(\mathbb{R}^{n-1})}\,,}\\\\
\displaystyle{||[b'_{-1}(D',\gamma),M^I]u^I_{|\,x_1=0}||_{H^{3/2}_{\gamma}(\mathbb{R}^{n-1})}\le C||u^I_{|\,x_1=0}||_{H^{-1/2}_{\gamma}(\mathbb{R}^{n-1})}\,,}\\\\
\displaystyle{||[b'_{-1}(D',\gamma),\ell_{\sharp}(x',D',\gamma)]u^{I,s}_{|\,x_1=0}||_{H^{3/2}_{\gamma}(\mathbb{R}^{n-1})}\le C||u^{I,s}_{|\,x_1=0}||_{H^{-1/2}_{\gamma}(\mathbb{R}^{n-1})}\,.}
\end{array}
\end{equation}
Finally, since $d_{-3}(x',D',\gamma)$ and $e_{-3}(x', D', \gamma)$ have symbol in $\Gamma^{-3}$ (see \eqref{d-3} and \eqref{e-3}) we obtain
\begin{equation}\label{commbordo2}
||d_{-3}(x', D',\gamma)\psi||_{H^{3/2}_\gamma(\mathbb{R}^{n-1})}\le C||\psi||_{H^{-3/2}_{\gamma}(\mathbb{R}^{n-1})}\le \frac{C}{\gamma^{3/2}}||\psi||_{L^2(\mathbb{R}^{n-1})}\,,\quad\forall\,\gamma\ge 1\,,
\end{equation}
\begin{equation}\label{commbordo2}
||e_{-3}(x', D',\gamma)u^{I,s}_{|\, x_1=0}||_{H^{3/2}_\gamma(\mathbb{R}^{n-1})}\le C||u^{I,s}_{|\, x_1=0}||_{H^{-3/2}_{\gamma}(\mathbb{R}^{n-1})}\le \frac{C}{\gamma}||u^{I,s}_{|\, x_1=0}||_{H^{-1/2}(\mathbb{R}^{n-1})}\,,\quad\forall\,\gamma\ge 1\,,
\end{equation}
with $\gamma-$independent positive constant $C$.
Collecting the preceding estimates \eqref{commbordo1}, \eqref{commbordo2} and using \eqref{Gcal} we obtain
\begin{equation}\label{normaGcal}
\begin{array}{ll}
\displaystyle||\mathcal{G}||_{H^{3/2}_\gamma(\mathbb{R}^{n-1})}\le C\left(||b'_{-1}(D',\gamma)g||_{H^{3/2}_\gamma(\mathbb{R}^{n-1})}+||[b'_{-1}(D',\gamma),b_{\sharp}(x',D',\gamma)]\psi||_{H^{3/2}_\gamma(\mathbb{R}^{n-1})} \right.\\
\\
\displaystyle\quad\quad\left.+ ||[b'_{-1}(D',\gamma),\ell_{\sharp}(x',D',\gamma)]u^{I,s}_{|\, x_1=0}||_{H^{3/2}_\gamma(\mathbb{R}^{n-1})}+||d_{-3}(x', D',\gamma)\psi||_{H^{3/2}_\gamma(\mathbb{R}^{n-1})}\right.\\
\\
\displaystyle\quad\quad\left.+||[b'_{-1}(D',\gamma),M^I]u^I_{|\,x_1=0}||_{H^{3/2}_\gamma(\mathbb{R}^{n-1})} + ||e_{-3}(x', D',\gamma)u^{I,s}_{|\, x_1=0}||_{H^{3/2}_\gamma(\mathbb{R}^{n-1})}\right)\\
\\
\displaystyle{\quad\quad \le C\left(||g||_{H^{1/2}_\gamma(\mathbb{R}^{n-1})}+\frac1{\gamma^{1/2}}||\psi||_{L^2(\mathbb{R}^{n-1})}+||u^I_{|\,x_1=0}||_{H^{-1/2}_{\gamma}(\mathbb{R}^{n-1})}\right)\,,\quad\forall\,\gamma\ge 1\,,}
\end{array}
\end{equation}
with $\gamma-$independent positive constant $C$.
\subsection{Proof of estimate \eqref{l2h1estimate}}
We start from \eqref{lambdaestimate} and use \eqref{l2lambda}, \eqref{l2traccia}, \eqref{l2psi}, \eqref{stimaFcal}, \eqref{normaGcal} to get
\begin{equation*}
\begin{array}{ll}
\gamma\left(||u||^2_{L^2(\mathbb{R}^n_+)}+||u^I_{|\,x_1=0}||^2_{H^{-1/2}_\gamma(\mathbb{R}^{n-1})}\right)+\gamma^2||\psi||^2_{L^2(\mathbb{R}^{n-1})}\\
\\
\displaystyle{\le\frac{C}{\gamma^3}\left(||F||^2_{H^1_{tan, \gamma}(\mathbb{R}^n_+)}+||u||^2_{L^2(\mathbb{R}^n_+)}\right)+\frac{C}{\gamma}\left(||g||^2_{H^{1/2}_\gamma(\mathbb{R}^{n-1})}+\frac1{\gamma}||\psi||^2_{L^2(\mathbb{R}^{n-1})}+||u^I_{|\,x_1=0}||^2_{H^{-1/2}_{\gamma}(\mathbb{R}^{n-1})}\right)}
\end{array}
\end{equation*}
for all $\gamma\ge\gamma_1$, with $\gamma_1\ge 1$ large enough, and $C>0$ independent of $\gamma$.
\newline
Then estimate \eqref{l2h1estimate} follows by absorbing into the left-hand side the terms involving the functions $u$, $\psi$ in the right-hand side of the above inequality. This ends the proof of the statement $1$ of Theorem \ref{mainthm}.
\subsection{Proof of estimate \eqref{l2estimate}, statement $2$ of Theorem \ref{mainthm}.}\label{4.7}
In the end, let us shortly discuss the proof of the estimate \eqref{l2estimate} in Theorem \ref{mainthm}, statement $2$, under the assumption $(H)_2$ about the BVP \eqref{sbvp}.
\newline
As it was done in Section \ref{sec_regBVP}, for given smooth functions $(u, \psi)$ we firstly define the data
\begin{equation}\label{tfb1}
\begin{array}{ll}
F:=\mathcal L_\gamma u\,,\\
g:=b_\gamma\psi+\mathcal M^s_\gamma u^{I,s}+M^Iu^I+b_{\sharp}(x',D',\gamma)\psi+\ell_\sharp(x',D',\gamma)u^{I,s}\,.
\end{array}
\end{equation}
Notice that, differently from the case of statement $1$ (see formulas \eqref{tf}, \eqref{tb}), no lower order term in $u$ is involved in the definition of the interior source term $F$ in \eqref{tfb1}; this agrees with the assumption $(H)_2$, about the BVP \eqref{sbvp}, where no stability assumption under lower order interior operators is required for the estimate \eqref{h1estimate}.
\newline
Then, following the strategy already explained in Section \ref{gs}, we apply the operator $\lambda^{-1,\gamma}_\chi(Z)$ to the first equation in \eqref{tfb1} and we find
\begin{equation}\label{sistemaregolarizzato2}
\mathcal L_\gamma(\lambda^{-1,\gamma}_{\chi}(Z)u)=\lambda^{-1,\gamma}_{\chi}(Z)F-[\lambda^{-1,\gamma}_{\chi}(Z),\mathcal L_\gamma]u\,,\quad{\rm in}\,\,\mathbb{R}^n_+\,.
\end{equation}
Compared to the analogous equation \eqref{sistemaregolarizzato}, in the left-hand side of the above equation there is no lower order operator $\rho_\sharp(x,Z,\gamma)$. Moreover, we notice that the term involving the commutator $[\lambda^{-1,\gamma}_{\chi}(Z),\mathcal L_\gamma]$ has been put in the right-hand side of the equation \eqref{sistemaregolarizzato2}, which means that this term can be just regarded as a part of the source term of such an equation. This is a consequence of the fact that the a priori estimate \eqref{h1estimate}, that is associated to the BVP \eqref{sbvp} under the assumption $(H)_2$, does not lose derivatives from the interior source term $F$: the $H^1_{tan}-$norm of the unknown $u$ is measured by the $H^1_{tan}-$norm of $F$.
\newline
Concerning the boundary condition, the same arguments developed in the Section \ref{bdrysect} give that the function $(\lambda^{-1,\gamma}_\chi(Z)u, b'_{-1}(D',\gamma)\psi)$ satisfy the equation $\eqref{fsystem}_2$ on the boundary.
\newline
Applying the estimate \eqref{h1estimate} to the BVP \eqref{sistemaregolarizzato2}, $\eqref{fsystem}_2$ we find again that $(\lambda^{-1,\gamma}_\chi(Z)u, b'_{-1}(D',\gamma)\psi)$ obey the estimate
\begin{equation}\label{lambdaestimate1}
\begin{array}{ll}
\displaystyle{\gamma\left(||\lambda^{-1,\gamma}_\chi(Z)u||^2_{H^1_{ tan\,,\gamma}(\mathbb{R}^n_+)}+||(\lambda^{-1,\gamma}_\chi(Z)u^I)_{|\,x_1=0}||^2_{H^{1/2}_\gamma(\mathbb{R}^{n-1})}\right)+\gamma^2||b'_{-1}(D',\gamma)\psi||^2_{H^1_\gamma(\mathbb{R}^{n-1})}}
\\\\
\quad\quad\displaystyle{\le \frac{C_0}{\gamma}\left(||\mathcal{F}||^2_{H^{1}_{tan,\gamma}(\mathbb{R}^n_+)}+||\mathcal{G}||^2_{H^{3/2}_{\gamma}(\mathbb{R}^{n-1})}\right)\,,}
\end{array}
\end{equation}
where the interior source term $\mathcal F$ is defined now as
\begin{equation}\label{new_Fcal}
\mathcal F:=\lambda^{-1,\gamma}_{\chi}(Z)F-[\lambda^{-1,\gamma}_{\chi}(Z),\mathcal L_\gamma]u\,,
\end{equation}
while the boundary datum $\mathcal G$ is given by \eqref{Gcal}.
\newline
To conclude the proof, it remains to provide an estimate of the Sobolev norms of $\mathcal F$ and $\mathcal G$ appearing in the right-hand side of \eqref{lambdaestimate1}. The estimate of $\mathcal G$ is exactly the estimate \eqref{normaGcal} obtained in Section \ref{commbordo}.
\newline
Concerning the estimate of $\mathcal F$, from \eqref{new_Fcal} we firstly get
\begin{equation}\label{stimanewFcal0}
\begin{array}{ll}
\displaystyle||\mathcal F||_{H^1_{tan,\gamma}(\mathbb R^n_+)}\le\left\{||\lambda^{-1,\gamma}_\chi(Z)F||_{H^1_{tan,\gamma}(\mathbb R^n_+)}+||[\lambda^{-1,\gamma}_{\chi}(Z),\mathcal L_\gamma]u||_{H^1_{tan,\gamma}(\mathbb R^n_+)}\right\}\\
\\
\displaystyle\quad\le C\left\{||F||_{L^2(\mathbb R^n_+)}+||[\lambda^{-1,\gamma}_{\chi}(Z),\mathcal L_\gamma]u||_{H^1_{tan,\gamma}(\mathbb R^n_+)}\right\}\,,
\end{array}
\end{equation}
for a positive constant $C$ independent of $\gamma\ge 1$. In order to estimate the norm of the commutator term $[\lambda^{-1,\gamma}_{\chi}(Z),\mathcal L_\gamma]u$ involved in the right-hand side of \eqref{stimanewFcal0}, the same analysis performed in Sections \ref{commutatoretangenziale}, \ref{commutatorenormale} leads to the formula
\begin{equation}\label{decomposizione1}
\begin{array}{ll}
[\lambda^{-1,\gamma}_{\chi}(Z),\mathcal L_\gamma]u=[\lambda^{-1,\gamma}_{\chi}(Z),A_1^1\partial_1]u+[\lambda^{-1,\gamma}_{\chi}(Z),\mathcal L_{tan, \gamma}]u\\
\\
\quad =\left(\begin{array}{cc}q_{-1}(x,Z,\gamma)(\partial_1u^I)\\ 0\end{array}\right)+[\lambda^{-1,\gamma}_{\chi}(Z),\mathcal L_{tan, \gamma}]u\,,
\end{array}
\end{equation}
where the result of Proposition \ref{normale} (see also \eqref{componenti}) has been used to get the second equality above and $\mathcal L_{tan,\gamma}$ is the tangential differential operator defined in \eqref{Ltan}.
\newline
Since, in view of Proposition \ref{prodottoecommutatore}, $[\lambda^{-1,\gamma}_{\chi}(Z),\mathcal L_{tan, \gamma}]$ is a conormal operator with symbol in $\Gamma^{-1}$, Proposition \ref{continuitaconormale} yields
\begin{equation}\label{normacommtan}
||[\lambda^{-1,\gamma}_{\chi}(Z),\mathcal L_{tan, \gamma}]u||_{H^1_{tan,\gamma}(\mathbb{R}^n_+)}\le C||u||_{L^2(\mathbb{R}^n_+)}\,,
\end{equation}
with some positive $\gamma-$independent constant $C$.
\newline
As for $q_{-1}(x,Z,\gamma)$, it is a conormal operator with symbol in $\Gamma^{-2}$. Writing again $\partial_1u^I$ is terms of conormal derivatives of $u$ and $F$ as in \eqref{partial_1} gives
\begin{equation*}
q_{-1}(x,Z,\gamma)(\partial_1u^I)=q_{-1}(x,Z,\gamma)\left((A^{I,I}_1)^{-1}F^I+\mathcal{T}_\gamma u\right)\,,
\end{equation*}
where $\mathcal{T}_{\gamma}$ is the conormal operator of order $1$ defined in \eqref{normtan} (with $\rho_\sharp=0$). Hence in view of Proposition \ref{continuitaconormale} we get
\begin{equation}\label{normacommnor}
\begin{array}{ll}
\displaystyle{||[\lambda^{-1,\gamma}_{\chi}(Z),A^{I,I}_1\partial_1]u^{I}||_{H^1_{tan,\gamma}(\mathbb R^n_+)}=\left|\left|\lambda^{1,\gamma}(Z)\left(q_{-1}(x,Z,\gamma)((A^{I,I}_1)^{-1}F^I+\mathcal{T}_\gamma u)\right)\right|\right|_{L^2(\mathbb{R}^n_+)}}\\
\\
\displaystyle{\le \left|\left|\lambda^{1,\gamma}(Z)\left(q_{-1}(x,Z,\gamma)((A^{I,I}_1)^{-1}F^I)\right)\right|\right|_{L^2(\mathbb{R}^n_+)}+||\lambda^{1,\gamma}(Z)q_{-1}(x,Z,\gamma)\mathcal{T}_\gamma u||_{L^2(\mathbb{R}^n_+)}}\\
\\
\displaystyle{\le C_0\left(||F^I||_{H^{-1}_{tan,\gamma}(\mathbb{R}^n_+)}+||u||_{L^2(\mathbb{R}^n_+)}\right)\le C_0\left(\frac1{\gamma}||F^I||_{L^2(\mathbb{R}^n_+)}+||u||_{L^2(\mathbb{R}^n_+)}\right)\,.}
\end{array}
\end{equation}
Collecting estimates \eqref{stimanewFcal0}, \eqref{normacommtan}, \eqref{normacommnor}, we finally get
\begin{equation}\label{stimanewFcal}
\begin{array}{ll}
||\mathcal{F}||_{H^1_{tan,\gamma}(\mathbb{R}^n_+)}\le C\left(||F||_{L^2(\mathbb{R}^n_+)}+||u||_{L^2(\mathbb{R}^n_+)}\right)\,,\quad\forall\,\gamma\ge 1\,,
\end{array}
\end{equation}
with $\gamma-$independent positive constant $C$.
\newline
The estimate \eqref{l2estimate} follows at once by combining \eqref{lambdaestimate1} with \eqref{normaGcal} and \eqref{stimanewFcal}.


\appendix
\section{Proof of some technical lemmata}\label{appA}
\subsection{Proof of Lemma \ref{supporto}}\label{appA.1}
For a given smooth function $u\in C^\infty_{(0)}(\mathbb{R}^n_+)$, an explicit calculation gives that
$$
\lambda^{m,\gamma}_{\chi}(Z)u(x)=\left\langle\mathcal{F}^{-1}\lambda^{m,\gamma}(\cdot),\chi(\cdot)e^{-\frac{(\cdot)_1}{2}}u(x_1e^{-(\cdot)_1},x'-(\cdot)') \right\rangle\,,\quad\forall\,x=(x_1,x')\in\mathbb{R}^n_+\,.
$$
We have to prove that, under a suitable choice of $\varepsilon_0$, if $x\notin\mathbb{B}^+$ then $\lambda^{m,\gamma}_{\chi}(Z)u(x)=0$. This is true if
$$
y\mapsto v_x(y):=\chi(y)e^{-\frac{y_1}{2}}u(x_1e^{-y_1},x'-y')
$$
is identically zero as long as $x\notin\mathbb{B}^+$.
\newline
Since $\mathbb{R}^n_+\setminus\mathbb{B}^+=\{x=(x_1,x'):\,x_1\ge 1\,,\,\,\forall\,x'\in\mathbb{R}^{n-1}\}\cup\{x=(x_1,x'):\,|x'|\ge 1\,,\,\,\forall\,x_1\in[0,+\infty[\}$ we need to analyze the following two cases.
\newline
{$1^{st}$ case}: $x_1\ge 1$.
\newline
Let $y\in\mathbb{R}^n$ be arbitrarily fixed. If $y\notin{\rm supp}\chi$, then $\chi(y)=0$, which implies $v_x(y)=0$. If $y\in{\rm supp}\chi$, then we have $-\varepsilon_0\le y_1\le\varepsilon_0$ and $|y'|\le\varepsilon_0$. Hence, we derive that $e^{-\varepsilon_0}\le e^{-y_1}\le e^{\varepsilon_0}$ and, since $x_1\ge 1$, $x_1e^{-y_1}\ge e^{-y_1}\ge e^{-\varepsilon_0}$. Since $u(x_1,x')=0$ when $x_1\ge\delta_0$, if we choose $\varepsilon_0>0$ such that $e^{-\varepsilon_0}>\delta_0$ (that is equivalent to $\varepsilon_0<\log(1/\delta_0)$), then we get that
$$
\forall\,y\in{\rm supp}\chi\,,\,\,\forall\,x_1\ge 1\,:\quad u(x_1e^{-y_1},x'-y')=0\,,
$$
which gives $v_x(y)=0$.
\newline
{$2^{nd}$ case}: $|x'|\ge 1$.
\newline
Again, if $y\notin{\rm supp}\chi$, then $v_x(y)=0$. If $y\in{\rm supp}\chi$ then $|x'-y'|\ge |x'|-|y'|\ge 1-|y'|\ge 1-\varepsilon_0$. To conclude, in this case it is sufficient to choose $\varepsilon_0>0$ such that $1-\varepsilon_0>\delta_0$ in order to have again $v_x(y)=0$.
\newline
Finally, the result is proved if we choose $0<\varepsilon_0\le\min\{\log(1/\delta_0),1-\delta_0\}$.
\subsection{Proof of Lemma \ref{lemmatecnico2}}\label{appA.2}
For arbitrary $u\in L^2(\mathbb{R}^n_+)$, we observe that in view of \eqref{coeffcost}, \eqref{operatoreconormale}
\begin{equation}
(r_m(Z,\gamma)u)^{\sharp}=r_m(D,\gamma)u^\sharp=\mathcal{F}^{-1}(r_m(\cdot,\gamma))\ast u^\sharp\,;
\end{equation}
then, for arbitrary $\beta\in\mathbb{N}^n$:
$$
\partial^{\beta}(r_m(Z,\gamma)u)^{\sharp}=(\partial^{\beta}\mathcal{F}^{-1}(r_m(\cdot,\gamma))\ast u^\sharp\,.
$$
Since $H^{p}_{tan,\gamma}(\mathbb{R}^n_+)$ is topologically isomorphic to $H^p_{\gamma}(\mathbb{R}^n)$ for all positive integers $p$, via the $\sharp$ operator, and $u^\sharp\in L^2(\mathbb{R}^n)$, then $r_{m}(Z,\gamma)u\in H^p_{tan,\gamma}(\mathbb{R}^n_+)$ is proven provided that $\partial^{\beta}\mathcal{F}^{-1}(r_m(\cdot,\gamma))$ belongs to $L^1(\mathbb{R}^n)$ for all $\beta\in\mathbb{N}^n$ with $|\beta|\le p$.
\newline
On the other hand, by the standard properties of the Fourier transform and by \eqref{funzionepesomodificata}, we get
\begin{equation}\label{resto}
\begin{array}{ll}
\mathcal{F}^{-1}(r_m(\cdot,\gamma))=\mathcal{F}^{-1}((I-\chi(D))\lambda^{m,\gamma})=\mathcal{F}^{-1}(\mathcal{F}^{-1}((1-\chi)\widehat{\lambda^{m,\gamma}}))\\
\\
=(2\pi)^{-n}\widetilde{((1-\chi)\widehat{\lambda^{m,\gamma}})}=(1-\chi)\mathcal{F}^{-1}(\lambda^{m,\gamma}),
\end{array}
\end{equation}
where we have used the identity $\mathcal{F}^{-1}g= (2\pi)^{-n}\widetilde{\widehat{g}},$ with $\widetilde g(x)=g(-x)$, and that $\chi$ is an even function.
\newline
Let us firstly focus on $\mathcal{F}^{-1}(\lambda^{m,\gamma})$. For arbitrary positive integers $N, k$ and $\beta\in\mathbb{N}^n$ one computes
\begin{equation}\label{F-1}
\begin{array}{ll}
\displaystyle{|z|^{2(N+k)}\partial^\beta_z\mathcal{F}^{-1}(\lambda^{m,\gamma})(z)=i^{|\beta|}\sum\limits_{|\alpha|=N+k}\frac{(N+k)!}{\alpha!}z^{2\alpha}\mathcal{F}^{-1}(\xi^\beta\lambda^{m,\gamma})(z)}\\
\\
\displaystyle{=i^{|\beta|}(-1)^{N+k}\sum\limits_{|\alpha|=N+k}\frac{(N+k)!}{\alpha!}\mathcal{F}^{-1}\left(\partial^{2\alpha}_\xi(\xi^\beta\lambda^{m,\gamma})\right)(z)\,.}
\end{array}
\end{equation}
On the other hand, since $\lambda^{m,\gamma}\in\Gamma^m$, for $|\alpha|=N+k$ we get
$$
\begin{array}{ll}
|\partial^{2\alpha}_\xi(\xi^\beta\lambda^{m,\gamma}(\xi))|\le C_{\alpha,\beta}\lambda^{m+|\beta|-2|\alpha|,\gamma}(\xi)= C_{\alpha,\beta}\lambda^{m+|\beta|-2(N+k),\gamma}(\xi)\\
\\
=C_{\alpha,\beta}\lambda^{-2k,\gamma}(\xi)\lambda^{m+|\beta|-2N,\gamma}(\xi)\le C_{\alpha,\beta}\gamma^{-2k}\lambda^{m+|\beta|-2N,\gamma}(\xi) \,,\quad\forall\,\xi\in\mathbb{R}^n\,,\,\,\forall\,\gamma\ge 1\,.
\end{array}
$$
For fixed $\beta$, we choose the integer $N_{\beta}=N$ such that $2N\ge m+|\beta|+1+n$; then
$$
\lambda^{m+|\beta|-2N,\gamma}(\xi)\le\lambda^{-(1+n),\gamma}(\xi)\le (1+|\xi|^2)^{-\frac{n+1}{2}}\,,\quad\forall\,\xi\in\mathbb{R}^n\,,\,\,\forall\,\gamma\ge 1
$$
yields
$$
|\partial^{2\alpha}_\xi(\xi^\beta\lambda^{m,\gamma}(\xi))|\le C_{\alpha,\beta}\gamma^{-2k}(1+|\xi|^2)^{-\frac{n+1}{2}}\,,\quad\forall\,\xi\in\mathbb{R}^n\,,\,\,\forall\,\gamma\ge 1\,;
$$
hence $\partial^{2\alpha}_\xi(\xi^\beta\lambda^{m,\gamma}(\xi))\in L^1(\mathbb{R}^n)$ and, from Riemann-Lebesgue Theorem, $\mathcal{F}^{-1}(\partial^{2\alpha}_\xi(\xi^\beta\lambda^{m,\gamma}(\xi)))\in L^\infty(\mathbb{R}^n)\cap C^0(\mathbb{R}^n)$ and we have
$$
\begin{array}{ll}
\displaystyle{||\mathcal{F}^{-1}(\partial^{2\alpha}_\xi(\xi^\beta\lambda^{m,\gamma}(\xi)))||_{L^\infty(\mathbb{R}^n)}\le\int_{\mathbb{R}^n}|\partial^{2\alpha}_\xi(\xi^\beta\lambda^{m,\gamma}(\xi))|\,d\xi}\\
\\
\qquad\displaystyle{\le C_{\alpha,\beta}\gamma^{-2k}\int_{\mathbb{R}^n}(1+|\xi|^2)^{-\frac{n+1}{2}}d\xi\le C_{\alpha,\beta,n}\gamma^{-2k}\,,\quad \forall\,\gamma\ge 1\,.}
\end{array}
$$
Therefore, in view of \eqref{F-1},
$$
|z|^{2(N+k)}\partial^\beta_z\mathcal{F}^{-1}(\lambda^{m,\gamma})(z)\in  L^\infty(\mathbb{R}^n)\cap C^0(\mathbb{R}^n)
$$
and
$$
|z|^{2(N+k)}|\partial^\beta_z\mathcal{F}^{-1}(\lambda^{m,\gamma})(z)|\le C_{k,N,\beta,n}\gamma^{-2k}\,,\quad \forall\,z\in\mathbb{R}^n\,,\gamma\ge 1\,,
$$
where the constant $C_{k,N,\beta,n}$ is independent of $\gamma$.
\newline
Summarizing, we have proved that:
$$
\begin{array}{ll}
\displaystyle{\forall\,\beta\in\mathbb{N}^n\,,\forall\,k,N\in\mathbb{N},\,\text{with}\,\, k\ge 1\,,\,\,\,N\ge\frac{m+|\beta|+1+n}{2}\,,\,\,\,\exists\,C=C_{k,N,\beta,n}>0:}\\
\\
\displaystyle{i.\quad |z|^{2(N+k)}\partial^\beta_z\mathcal{F}^{-1}(\lambda^{m,\gamma})(z)\in  L^\infty(\mathbb{R}^n)\cap C^0(\mathbb{R}^n)}\\\\
\displaystyle{ii.\quad |z|^{2(N+k)}|\partial^\beta_z\mathcal{F}^{-1}(\lambda^{m,\gamma})(z)|\le C_{k,N,\beta,n}\gamma^{-2k}\,,\quad \forall\,z\in\mathbb{R}^n\,,\gamma\ge 1\,.}
\end{array}
$$
For arbitrary $\beta\in\mathbb{N}^n$, we consider $\partial^{\beta}\mathcal{F}^{-1}(r_m(\cdot,\gamma))$. From \eqref{resto} we compute, by Leibniz formula,
\begin{equation}\label{resto1}
\partial^{\beta}\mathcal{F}^{-1}(r_m(\cdot,\gamma))(z)=-\sum\limits_{\nu<\beta}\begin{pmatrix}\beta\\\nu\end{pmatrix}\partial^{\beta-\nu}_z\chi(z)\partial^\nu_z\mathcal{F}^{-1}(\lambda^{m,\gamma})(z)+(1-\chi)(z)\partial^{\beta}_z\mathcal{F}^{-1}(\lambda^{m,\gamma})(z)\,.
\end{equation}
Note that $\partial^{\beta-\nu}\chi$, for all $\nu<\beta$, and $1-\chi$ are identically zero on a neighbourhood of $z=0$. Then, from {\it i}, {\it ii} above we derive that
$$
\begin{array}{ll}
\displaystyle{\forall\,\beta\in\mathbb{N}^n\,,\forall\,k,N\in\mathbb{N},\,\text{with}\,\, k\ge 1\,,\,\,\,N\ge\frac{m+|\beta|+1+n}{2}\,,\,\,\,\exists\,C=C_{k,N,\beta,\chi,n}>0:}\\
\\
\displaystyle{iii.\quad \partial^{\beta-\nu}_z\chi(z)\partial^\nu_z\mathcal{F}^{-1}(\lambda^{m,\gamma})(z),\,  (1-\chi)(z)\partial^\beta_z\mathcal{F}^{-1}(\lambda^{m,\gamma})(z) \in  L^\infty(\mathbb{R}^n)\cap C^0(\mathbb{R}^n)}\,,\quad\forall\,\nu<\beta\,;\\\\
\displaystyle{\begin{array}{ll}iv.\quad |\partial^{\beta-\nu}_z\chi(z)\partial^\nu_z\mathcal{F}^{-1}(\lambda^{m,\gamma})(z)|\le C_{k,N,\beta,\chi,n}\gamma^{-2k}(1+|z|^2)^{-N}\,,\\
\\
\quad\quad |(1-\chi)(z)\partial^\beta_z\mathcal{F}^{-1}(\lambda^{m,\gamma})(z)|\le C_{k,N,\beta,\chi,n}\gamma^{-2k}(1+|z|^2)^{-N}\,,\,\,\forall\,z\in\mathbb{R}^n\,,\,\,\nu<\beta\,,\,\,\gamma\ge 1\,.\end{array}}
\end{array}
$$
Thus, applying {\it iv} for $N\ge\max\left\{\displaystyle{\frac{n+1}{2}},\displaystyle{\frac{m+|\beta|+n+1}{2}}\right\}$, from \eqref{resto1} we obtain that $\partial^{\beta}\mathcal{F}^{-1}(r_m(\cdot,\gamma))\in L^1(\mathbb{R}^n)$ and for all $\gamma\ge 1$:
\begin{equation}\label{resto2}
||\partial^{\beta}\mathcal{F}^{-1}(r_m(\cdot,\gamma))||_{L^1(\mathbb{R}^n)}\le C_{N,k,n,\beta,\chi}\gamma^{-2k}\int_{\mathbb{R}^n}(1+|z|^2)^{-N}dz\le C_{k,n,\beta,\chi}\gamma^{-2k}\le C_{k,n,\beta,\chi}\gamma^{-k} \,,
\end{equation}
where the constant $C_{k,n,\beta,\chi}$ is independent of $\gamma$.
\newline
For every positive integer $p$, applying the above result to all multi-indices $\beta\in\mathbb{N}^n$ with $|\beta|\le p$ gives that $\partial^\beta (r_{m}(Z,\gamma)u)^\sharp=\partial^\beta\mathcal{F}^{-1}(r_m(\cdot,\gamma))\ast u^\sharp$ belongs to $L^2(\mathbb{R}^n)$ with
\begin{equation}\label{derivatabeta}
||\partial^\beta (r_{m}(Z,\gamma)u)^\sharp||_{L^2(\mathbb{R}^n)}\le ||\partial^{\beta}\mathcal{F}^{-1}(r_m(\cdot,\gamma))||_{L^1(\mathbb{R}^n)}||u^{\sharp}||_{L^2(\mathbb{R}^n)}\le C_{k,n,\beta,\chi}\gamma^{-k}||u||_{L^2(\mathbb{R}^n_+)}\,.
\end{equation}
This gives that $r_m(Z,\gamma)u\in H^p_{tan,\gamma}(\mathbb{R}^n_+)$. Furthermore, for an arbitrary positive integer $h$ we apply \eqref{derivatabeta} for each $\beta\in\mathbb{N}^n$ with $|\beta|\le p$ for $k=p-|\beta|+h$ to get
\begin{equation}
\begin{array}{ll}
||r_m(Z,\gamma)u||^2_{H^p_{tan,\gamma}(\mathbb{R}^n_+)}\le C_p||(r_m(Z,\gamma)u)^\sharp||^2_{H^p_\gamma(\mathbb{R}^n)}=\sum\limits_{|\beta|\le p}\gamma^{2(p-|\beta|)}||\partial^\beta(r_{m}(Z,\gamma)u)^\sharp||^2_{L^2(\mathbb{R}^n)}\\\\
\le\sum\limits_{|\beta|\le p}\gamma^{2(p-|\beta|)} C_{h,p,n,\beta,\chi}\gamma^{-2(p-|\beta|+h)}||u||^2_{L^2(\mathbb{R}^n_+)}\le C_{h,p,n,\chi}\gamma^{-2h}||u||^2_{L^2(\mathbb{R}^n_+)}\,,
\end{array}
\end{equation}
for a suitable $\gamma-$independent positive constant $C_{h,p,n,\chi}$. This shows the estimate \eqref{kregolarizzazione} and completes the proof.
\subsection{Proof of Proposition \ref{bordo}}\label{appA.3}
Let $u\in C^{\infty}_{(0)}(\mathbb{R}^n_+)$; to find a symbol $b'_{m}$ satisfying \eqref{formulatraccia}, from \eqref{funzionepesomodificata} we firstly compute
\begin{equation*}
\begin{array}{ll}
\displaystyle{(\lambda^{m,\gamma}_{\chi}(Z)u)^{\sharp}(x)=\lambda^{m,\gamma}_{\chi}(D)(u^{\sharp})(x)=(\mathcal{F}^{-1}(\lambda^{m,\gamma}_{\chi})\ast u^{\sharp})(x)=\langle\mathcal{F}^{-1}(\lambda^{m,\gamma}_{\chi}),u^{\sharp}(x-\cdot)\rangle}\\
\\
\displaystyle{=\langle\mathcal{F}^{-1}(\lambda^{m,\gamma}),\chi(\cdot)e^\frac{{x_1-(\cdot)_1}}{2}u(e^{x_1-(\cdot)_1},x'-(\cdot)')\rangle\,,\quad\forall\,(x_1,x')\in\mathbb{R}^n\,,}
\end{array}
\end{equation*}
hence, by \eqref{sharp-1},
\begin{equation*}
\begin{array}{ll}
\displaystyle{\lambda^{m,\gamma}_{\chi}(Z)u(x)=\langle\mathcal{F}^{-1}(\lambda^{m,\gamma}),\chi(\cdot)e^\frac{{x_1-(\cdot)_1}}{2}u(e^{x_1-(\cdot)_1},x'-(\cdot)')\rangle^{\sharp^{-1}}}\\\\ \displaystyle{=\frac1{\sqrt{x_1}}\left\langle \langle\mathcal{F}^{-1}(\lambda^{m,\gamma}),\chi(\cdot)e^\frac{{\log x_1-(\cdot)_1}}{2}u(e^{\log x_1-(\cdot)_1},x'-(\cdot)')\right\rangle}\\\\
\displaystyle{=\left\langle \mathcal{F}^{-1}(\lambda^{m,\gamma}),\chi(\cdot)e^{\frac{-(\cdot)_1}{2}}u(x_1e^{-(\cdot)_1},x'-(\cdot)')\right\rangle}\\\\
\displaystyle{=\left\langle \lambda^{m,\gamma},\mathcal{F}^{-1}\left(\chi(\cdot)e^{\frac{-(\cdot)_1}{2}}u(x_1e^{-(\cdot)_1},x'-(\cdot)')\right)\right\rangle}\\\\
\displaystyle{=(2\pi)^{-n}\int\lambda^{m,\gamma}(\xi)\left(\int e^{i\xi\cdot y}\chi(y)e^{-\frac{y_1}{2}}u(x_1e^{-y_1},x'-y')dy\right)d\xi\,,\quad\forall\,x_1>0\,,\,\,\forall\,x'\in\mathbb{R}^{n-1}\,.}
\end{array}
\end{equation*}
The regularity of $u$ legitimates all the above calculations. Setting $x_1=0$ in the last expression above, we deduce the corresponding expression for the trace on the boundary of $\lambda^{m,\gamma}_{\chi}(Z)u$
\begin{equation}\label{formulatraccia0}
(\lambda^{m,\gamma}_{\chi}(Z)u)_{|\,x_1=0}(x')=(2\pi)^{-n}\int\lambda^{m,\gamma}(\xi)\left(\int e^{i\xi\cdot y}\chi(y)e^{-\frac{y_1}{2}}(u_{|\,x_1=0})(x'-y')dy\right)d\xi\,.
\end{equation}
\noindent
Now we substitute \eqref{prodottotensoriale} into the $y-$integral appearing in the last expression above; then Fubini's theorem gives
\begin{equation}\label{fubini1}
\begin{array}{ll}
\displaystyle{\int e^{i\xi\cdot y}\chi_1(y_1)\widetilde{\chi}(y')e^{-\frac{y_1}{2}}(u_{|\,x_1=0})(x'-y')dy}\\
\displaystyle{=\int e^{i\xi'\cdot y'}\left(\int e^{i\xi_1y_1}e^{-\frac{y_1}{2}}\chi_1(y_1)dy_1\right)\widetilde{\chi}(y')(u_{|\,x_1=0})(x'-y')dy'}\\
\displaystyle{=\int e^{i\xi'\cdot y'}\left(\int e^{-i\xi_1(-y_1)}e^{-\frac{y_1}{2}}\chi_1(y_1)dy_1\right)\widetilde{\chi}(y')(u_{|\,x_1=0})(x'-y')dy'}\\
\displaystyle{=\int e^{i\xi'\cdot y'}\left(\int e^{-i\xi_1(-y_1)}e^{-\frac{y_1}{2}}\chi_1(-y_1)dy_1\right)\widetilde{\chi}(y')(u_{|\,x_1=0})(x'-y')dy'}\\
\displaystyle{=\left(e^{\frac{(\cdot)_1}{2}}\chi_1\right)^{\wedge_1}(\xi_1)\int e^{i\xi'\cdot y'}\widetilde{\chi}(y')(u_{|\,x_1=0})(x'-y')dy'\,,}
\end{array}
\end{equation}
where we have used that $\chi_1$ is even and $\wedge_1$ denotes the one-dimensional Fourier transformation with respect to $y_1$. Writing, by the inversion formula, $(u_{|\,x_1=0})(x'-y')=(2\pi)^{-n+1}\int e^{i(x'-y')\cdot\eta'}\widehat{u_{|\,x_1=0}}(\eta')d\eta'$ and using once more Fubini's theorem and that $\widetilde{\chi}$ is even, we further obtain
\begin{equation}\label{fubini2}
\begin{array}{ll}
\displaystyle{\int e^{i\xi'\cdot y'}\widetilde{\chi}(y')(u_{|\,x_1=0})(x'-y')dy'=(2\pi)^{-n+1}\int e^{i\xi'\cdot y'}\widetilde{\chi}(y')\left(\int e^{i(x'-y')\cdot\eta'}\widehat{u_{|\,x_1=0}}(\eta')d\eta'\right)dy'}\\
\displaystyle{=\int e^{ix'\cdot\eta'}\left((2\pi)^{-n+1}\int e^{i(\xi'-\eta')\cdot y'}\widetilde{\chi}(y')dy'\right)\widehat{u_{|\,x_1=0}}(\eta')d\eta'}\\
\displaystyle{=\int e^{ix'\cdot\eta'}\left((2\pi)^{-n+1}\int e^{-i(\xi'-\eta')\cdot (-y')}\widetilde{\chi}(-y')dy'\right)\widehat{u_{|\,x_1=0}}(\eta')d\eta'}\\
\displaystyle{=(2\pi)^{-n+1}\int e^{ix'\cdot\eta'}\widehat{\widetilde{\chi}}(\xi'-\eta')\widehat{u_{|\,x_1=0}}(\eta')d\eta'\,;}
\end{array}
\end{equation}
here $\wedge$ is used here to denote the $(n-1)-$dimensional Fourier transformation with respect to $x'$. Inserting \eqref{fubini1}, \eqref{fubini2} into \eqref{formulatraccia0} then leads to
\begin{equation}\label{formulatraccia1}
\begin{array}{ll}
\displaystyle{(\lambda^{m,\gamma}_{\chi}(Z)u)_{|\,x_1=0}(x')}\\
\displaystyle{=(2\pi)^{-n}\int\lambda^{m,\gamma}(\xi)\left(e^{\frac{(\cdot)_1}{2}}\chi_1\right)^{\wedge_1}(\xi_1)\left((2\pi)^{-n+1}\int e^{ix'\cdot\eta'}\widehat{\widetilde{\chi}}(\xi'-\eta')\widehat{u_{|\,x_1=0}}(\eta')d\eta'\right)d\xi\,.}
\end{array}
\end{equation}
Because $\left(e^{\frac{(\cdot)_1}{2}}\chi_1\right)^{\wedge_1}\in\mathcal{S}(\mathbb{R})$, $\widehat{\widetilde{\chi}}\in\mathcal{S}(\mathbb{R}^{n-1})$ and $\widehat{u_{|\,x_1=0}}\in\mathcal{S}(\mathbb{R}^{n-1})$, the double integral
\begin{equation*}
\int\int e^{ix'\cdot\eta'} \lambda^{m,\gamma}(\xi)\left(e^{\frac{(\cdot)_1}{2}}\chi_1\right)^{\wedge_1}(\xi_1)\widehat{\widetilde{\chi}}(\xi'-\eta')\widehat{u_{|\,x_1=0}}(\eta')d\eta'd\xi
\end{equation*}
converges absolutely; hence Fubini's theorem allows to exchange the order of the integrations in \eqref{formulatraccia1} and find
\begin{equation}\label{formulatraccia2}
(\lambda^{m,\gamma}_{\chi}(Z)u)_{|\,x_1=0}(x')=(2\pi)^{-n+1}\int e^{ix'\cdot\eta'} b'_{m}(\eta',\gamma)\widehat{u_{|\,x_1=0}}(\eta')d\eta'\,,
\end{equation}
where $b'_{m}(\eta',\gamma)$ is defined by \eqref{simbolobordobis}.
This shows the identity \eqref{formulatraccia}.
\subsection{Proof of Lemma \ref{lemmatecnico5}}\label{appA.4}
We follow the same lines of the proof of \cite[Lemma 4.11]{moseBVP}.
Setting for short
\begin{equation}\label{nucleoabbreviato}
\phi(x):=e^{x_1/2}\chi_1(x_1)\widetilde{\chi}(x')\,,
\end{equation}
the symbol \eqref{simbolobordobis} can be re-written as
\begin{equation}\label{simbolobordotris}
b'_{m}(\xi',\gamma)=(2\pi)^{-n}\int\lambda^{m,\gamma}(\eta_1,\eta'+\xi')\widehat{\phi}(\eta)\,d\eta\,.
\end{equation}
Substituting in \eqref{simbolobordotris} the function $\eta\mapsto\lambda^{m,\gamma}(\eta_1,\eta'+\xi')$ by its Taylor expansion about $\eta=0$
\begin{equation}\label{taylor2}
\lambda^{m,\gamma}(\eta_1,\eta'+\xi')=\sum\limits_{|\alpha|<N}\frac{(\partial^{\alpha}\lambda^{m,\gamma})(0,\xi')}{\alpha!}\eta^\alpha+N\sum\limits_{|\alpha|=N}\frac{\eta^\alpha}{\alpha!}
\int_0^1(\partial^\alpha\lambda^{m,\gamma})(t\eta_1,\xi'+t\eta')(1-t)^{N-1}dt
\end{equation}
for $N=2$, we get
\begin{equation}\label{splitting0}
\begin{array}{ll}
\displaystyle{b'_{m}(\xi',\gamma)}\\
\displaystyle{=(2\pi)^{-n}\int\left[\lambda^{m,\gamma}(\xi')+\sum\limits_{j=1}^n\eta_j(\partial_j\lambda^{m,\gamma})(0,\xi')+2\sum\limits_{|\alpha|=2}\frac{\eta^\alpha}{\alpha!}
\int_0^1(\partial^\alpha\lambda^{m,\gamma})(t\eta_1,\xi'+t\eta')(1-t)dt\right]\widehat{\phi}(\eta)\,d\eta}\\
\displaystyle{=(2\pi)^{-n}\lambda^{m,\gamma}(\xi')\int\widehat{\phi}(\eta)\,d\eta-i(2\pi)^{-n}\sum\limits_{j=1}^n(\partial_j\lambda^{m,\gamma})(0,\xi')\int\widehat{\partial_j\phi}(\eta)d\eta}\\ \displaystyle{-2(2\pi)^{-n}\sum\limits_{|\alpha|=2}\frac1{\alpha!}\int\left(\int_0^1\partial^{\alpha}\lambda^{m,\gamma}(t\eta_1,t\eta'+\xi')(1-t)\,dt\right)\widehat{\partial^{\alpha}\phi}(\eta)\,d\eta\,.}
\end{array}
\end{equation}
From Plancherel's identity and \eqref{nucleoabbreviato} (cf. also \eqref{chi}, \eqref{prodottotensoriale}) we compute
\begin{equation}\label{valoriphi}
\begin{array}{ll}
\displaystyle{(2\pi)^{-n}\int\widehat{\phi}(\eta)\,d\eta=\phi(0)=1\,,}\\
\displaystyle{(2\pi)^{-n}\int\widehat{\partial_1\phi}(\eta)\,d\eta=\partial_1\phi(0)=-\frac12\,,}\\ \displaystyle{(2\pi)^{-n}\int\widehat{\partial_j\phi}(\eta)\,d\eta=\partial_j\phi(0)=0\,,\quad j\ge 2\,.}
\end{array}
\end{equation}
On the other hand, from \eqref{weightfcts} one trivially computes that $(\partial_1\lambda^{m,\gamma})(0,\xi')=0$ for all $\xi'\in\mathbb{R}^{n-1}$. Inserting the last relation and \eqref{valoriphi} into \eqref{splitting0} then gives \eqref{splittingbordo}, where we set
\begin{equation}\label{beta}
\beta_{m,\delta}(\xi',\gamma):=-2(2\pi)^{-n}\sum\limits_{|\alpha|=2}\frac1{\alpha!}\int\left(\int_0^1\partial^\alpha\lambda^{m,\gamma}(t\eta_1,t\eta'+\xi')(1-t)\,dt\right)\widehat{\partial^\alpha\phi}(\eta)\,d\eta\,.
\end{equation}
To prove that $\beta_m$ belongs to $\Gamma^{m-2}$, differentiation under the integral sign of \eqref{beta} gives, for an arbitrary $\nu'\in\mathbb{N}^{n-1}$,
\begin{equation}\label{diff0}
\begin{array}{ll}
\displaystyle{\partial^{\nu'}_{\xi'}\beta_{m,\delta}(\xi',\gamma)=-2(2\pi)^{-n}\sum\limits_{|\alpha|=2}^n\frac1{\alpha!}\int\left[\partial^{\nu'}_{\xi'}\left(\int_0^1(\partial^\alpha\lambda^{m,\gamma})(t\eta_1,t\eta'+\xi')(1-t)\,dt\right)\right]\widehat{\partial^\alpha\phi}(\eta)\,d\eta}\\
\displaystyle{=-2(2\pi)^{-n}\sum\limits_{|\alpha|=2}^n\frac1{\alpha!}\int\left[\int_0^1(\partial^{\alpha+(0,\nu')}\lambda^{m,\gamma})(t\eta_1,t\eta'+\xi')(1-t)\,dt\right]\widehat{\partial^\alpha\phi}(\eta)\,d\eta\,;}
\end{array}
\end{equation}
hence from $\lambda^{m,\gamma}\in\Gamma^m$ we obtain
\begin{equation}\label{diff1}
|\partial^{\nu'}_{\xi'}\beta_{m,\delta}(\xi',\gamma)|
\le C_{m,\nu'}\sum\limits_{|\alpha|=2}\int\left(\int_0^1\lambda^{m-2-|\nu'|,\gamma}(t\eta_1,t\eta'+\xi')\,dt\right)|\widehat{\partial^\alpha\phi}(\eta)|\,d\eta\,,
\end{equation}
for a suitable $\gamma-$independent positive constant $C_{m,\nu'}$.
\newline
Recall that, for all $s\in\mathbb{R}$, $\gamma\ge 1$ and $\xi,\eta\in\mathbb{R}^n$
\begin{equation}\label{gammapeetre}
\lambda^{s,\gamma}(\xi)\le 2^{|s|}\lambda^{s,\gamma}(\xi-\eta)\lambda^{|s|}(\eta)\,,
\end{equation}
see \cite{chazarain-piriou}, \cite[Lemma 1.18]{saintraymond91}. Then, we apply \eqref{gammapeetre} (for $s=m-2-|\nu'|$) to estimate $\lambda^{m-2-|\nu'|,\gamma}(t\eta_1,t\eta'+\xi')$ within the right-hand side of \eqref{diff1} by
$$
\begin{array}{ll}
\displaystyle{
\lambda^{m-2-|\nu'|,\gamma}(t\eta_1,t\eta'+\xi')\le 2^{|m-2-|\nu'||}\lambda^{m-2-|\nu'|,\gamma}(\xi')\lambda^{|m-2-|\nu'||}(t \eta)}\\
\qquad\qquad\qquad\qquad\displaystyle{\le 2^{|m-2-|\nu'||}\lambda^{m-2-|\nu'|,\gamma}(\xi')\lambda^{|m-2-|\nu'||}(\eta)
\,,\quad\forall\,\xi'\in\mathbb{R}^{n-1},\,\eta\in\mathbb{R}^n\,,t\in[0,1]\,,}
\end{array}
$$
and combine with \eqref{diff1} to finally get
\begin{equation}\label{diff2}
|\partial^{\nu'}_{\xi'}\beta_{m}(\xi',\gamma)|\le C'_{m,\nu'}\lambda^{m-2-|\nu'|,\gamma}(\xi')\sum\limits_{|\alpha|=2}\int \lambda^{|m-2-|\nu'||}(\eta)|\widehat{\partial^\alpha\phi}(\eta)|d\eta\le C''_{m,\nu'}\lambda^{m-2-|\nu'|,\gamma}(\xi')\,,
\end{equation}
for $C'_{m,\nu'}, C''_{m,\nu'}$ suitable positive constants independent of $\gamma$ (notice in particular that the integrals in the sum involved in the right-hand side of the first inequality in \eqref{diff2} are absolutely convergent, because $\widehat{\partial^{\alpha}{\phi}}\in\mathcal{S}(\mathbb{R}^n)$ for all $|\alpha|=2$).
\subsection{Proof of Corollary \ref{supportobordo}}\label{appA.5}
For all $\psi\in C^\infty_0(\mathbb{R}^{n-1})$ under the above assumptions, let $\Psi\in C^\infty_{(0)}(\mathbb{R}^n_+)$ be chosen in such a way that
\begin{equation}\label{rilevamento}
{\rm supp}\,\Psi\subseteq\mathbb{B}^+_{\delta_0}\,,\quad \Psi_{|\,x_1=0}=\psi\,.
\end{equation}
Such a function $\Psi$ could be for instance obtained as
$$
\Psi(x_1,x'):=\eta(x_1)\psi(x')\,,\quad\forall\,x_1\ge 0\,,\,\,x'\in\mathbb{R}^{n-1}\,,
$$
with $\eta=\eta(x_1)\in C^\infty_{(0)}([0,+\infty[)$ such that
$$
\eta(x_1)=1\,,\quad 0\le x_1<\frac{\delta_0}{2}\,,\quad \eta(x_1)=0\,,\quad x_1>\delta_0\,.
$$
Then, in view of Proposition \ref{bordo} one has
$$
b'_{m}(D',\gamma)\psi=b'_{m}(D',\gamma)(\Psi_{|\,x_1=0})=(\lambda^{m,\gamma}_{\chi}(Z)\Psi)_{|\,x_1=0}\,.
$$
Then, from \eqref{rilevamento} and Lemma \ref{supporto},
$$
{\rm supp}\,b'_{m}(D',\gamma)\psi\subset\mathbb{B}^+\cap\{x_1=0\}=\mathcal{B}(0;1)\,.
$$
\subsection{Proof of Lemma \ref{stimesimbolonormale}}\label{appA.6}
Recall that we have defined for each $k=1,\dots,n$
\begin{equation}\label{qk}
q_{k,m}(x,\xi,\gamma):=(2\pi)^{-n}\int_{\mathbb{R}^n}\widehat{b}_k(x,\eta)\partial_k\lambda^{m,\gamma}(\xi-\eta)\,d\eta\,,
\end{equation}
where the functions $b_k=b_k(x,y)$ (cf. \eqref{taylor}) are given in $C^{\infty}(\mathbb{R}^n\times\mathbb{R}^n)$, have bounded derivatives in $\mathbb{R}^n\times\mathbb{R}^n$, and satisfy for all $x\in\mathbb{R}^n$
\begin{equation*}
{\rm supp}\,b_k(x,\cdot)\subseteq\{|y|\le 2\varepsilon_0\}\,.
\end{equation*}
Recall also that $\widehat{b}_k(x,\zeta)$ denotes the partial Fourier transform of $b_k(x,y)$ with respect to $y$.
\newline
The following lemma is concerned with the behavior at infinity of $\widehat{b}_k(x,\zeta)$.
\begin{lemma}
Let the function $b_k=b_k(x,y)\in C^{\infty}(\mathbb{R}^n\times\mathbb{R}^n)$ obey all of the preceding assumptions. Then, for every positive integer $N$ and all multi-indices $\alpha\in\mathbb{N}^n$ there exists a positive constant $C_{N,\alpha}$ such that
\begin{equation}\label{trasformataparziale}
(1+|\zeta|^2)^N|\partial^{\alpha}_x\widehat{b}_k(x,\zeta)|\le C_{N,\alpha}\,,\quad\forall\,x\,,\,\zeta\in\mathbb{R}^n\,.
\end{equation}
\end{lemma}
\begin{proof}
Since for each $x\in\mathbb{R}^n$, the function $b_k(x,\cdot)$ has compact support (independent of $x$), integrating by parts we get for an arbitrary integer $N>0$
\begin{equation}\label{intperparti}
\begin{array}{ll}
\displaystyle{(1+|\zeta|^2)^N\widehat{b}_k(x,\zeta)=\sum\limits_{|\alpha|\le N}\frac{N!}{\alpha!(N-|\alpha|)!}\int_{\{|y|\le 2\varepsilon_0\}} \zeta^{2\alpha} e^{-i\zeta\cdot y}b_k(x,y)\,dy}\\
\displaystyle{=\sum\limits_{|\alpha|\le N}\frac{N!}{\alpha!(N-|\alpha|)!}(-1)^{|\alpha|}\int_{\{|y|\le 2\varepsilon_0\}} \partial^{2\alpha}_{y}(e^{-i\zeta\cdot y})b_k(x,y)\,dy}\\
\displaystyle{=\sum\limits_{|\alpha|\le N}\frac{N!}{\alpha!(N-|\alpha|)!}(-1)^{|\alpha|}\int_{\{|y|\le 2\varepsilon_0\}} e^{-i\zeta\cdot y}\partial^{2\alpha}_y b_k(x,y)\,dy\,,}
\end{array}
\end{equation}
from which \eqref{trasformataparziale} trivially follows, using that $y-$derivatives of $b_k(x,y)$ are bounded in $\mathbb{R}^n\times\mathbb{R}^n$ by a positive constant independent of $x$.
\end{proof}
\noindent
We are going now to analyze the behavior at infinity of the derivatives of $q_{k,m}(x,\xi,\gamma)$ defined as in \eqref{qk}. For all multi-indices $\alpha,\beta\in\mathbb{N}^n$, differentiation under the integral sign in \eqref{qk} gives
\begin{equation}\label{derivatasottosgn}
\begin{array}{ll}
\partial^{\alpha}_{\xi}\partial^{\beta}_xq_{k,m}(x,\xi,\gamma)=(2\pi)^{-n}\int\partial^{\beta}_x\widehat{b}_k(x,\eta)\partial^{\alpha+e^k}\lambda^{m,\gamma}(\xi-\eta)\,d\eta\,,
\end{array}
\end{equation}
where $e^k:=(0,\dots,\underbrace{1}_{k},\dots,0)$. Then using that $\lambda^{m,\gamma}$ is a symbol of order $m$ together with \eqref{trasformataparziale} and combining with \eqref{gammapeetre}, for $s=m-1-|\alpha|$, we obtain
\begin{equation}\label{stimediffsottosgn}
\begin{array}{ll}
\displaystyle{|\partial^{\alpha}_{\xi}\partial^{\beta}_x q_{k,m}(x,\xi,\gamma)|\le C_{N,\beta}C_{m,\alpha}\int \lambda^{-2N}(\eta)\lambda^{m-1-|\alpha|,\gamma}(\xi-\eta)\,d\eta}\\
\displaystyle{\le C_{N,m,\alpha,\beta}\lambda^{m-1-|\alpha|,\gamma}(\xi)\int\lambda^{|m-1-|\alpha||-2N}(\eta)\,d\eta\,,}
\end{array}
\end{equation}
where the integral in the last line is finite, provided that the integer $N$ is taken to be sufficiently large. This provides the estimate \eqref{stimeqk}, with constant $C_{N,m,\alpha,\beta}\int\lambda^{|m-1-|\alpha||-2N}(\eta)\,d\eta$ independent of $\gamma$.


\section{Some examples from MHD}\label{appB}
\subsection{Current-vortex sheets}\label{appB.1}

Consider the equations of ideal compressible MHD:
\begin{equation}
\left\{
\begin{array}{l}
\partial_t\rho  +{\rm div}\, (\rho {v} )=0,\\[3pt]
\partial_t(\rho {v} ) +{\rm div}\,(\rho{v}\otimes{v} -{H}\otimes{H} ) +
{\nabla}q=0, \\[3pt]
\partial_t{H} -{\nabla}\times ({v} {\times}{H})=0,\\[3pt]
\partial_t\bigl( \rho e +\frac{1}{2}|{H}|^2\bigr)+
{\rm div}\, \bigl((\rho e +p){v} +{H}{\times}({v}{\times}{H})\bigr)=0,
\end{array}
\right.
\label{mhd1}
\end{equation}
where $\rho$ denotes density, $v\in\mathbb{R}^3$ plasma velocity, $H \in\mathbb{R}^3$ magnetic field, $p=p(\rho,S )$ pressure, $q =p+\frac{1}{2}|{H} |^2$ total pressure, $S$ entropy, $e=E+\frac{1}{2}|{v}|^2$ total energy, and  $E=E(\rho,S )$ internal energy. With a state equation of gas, $\rho=\rho(p ,S)$, and the first principle of thermodynamics, \eqref{mhd1} is a closed system. The system is symmetric hyperbolic provided
$\rho  >0, \rho_p >0. $
System \eqref{mhd1} is supplemented by the divergence constraint
\begin{equation}
{\rm div}\, {H} =0
\label{mhd2}
\end{equation}
on the initial data.

{\it Current-vortex sheets} are weak solutions of \eqref{mhd1} that are smooth on either side of a smooth hypersurface  $\Gamma(t)=\{x_1=\psi(t,x')\}$ in $[0,T]\times\Omega$, where $\Omega\subset\R^3,\, x'=(x_2,x_3)$ and that satisfy
suitable jump conditions at each point of the front $\Gamma (t)$.

Let us denote $\Omega^\pm(t)=\{x_1\gtrless \psi(t,x')\}$, where $\Omega=\Omega^+(t)\cup\Omega^-(t)\cup\Gamma(t)$; given any
function $g$ we denote $g^\pm=g$ in $\Omega^\pm(t)$ and $[g]=g^+_{|\Gamma}-g^-_{|\Gamma}$ the jump across
$\Gamma (t)$.

One looks for smooth solutions $(v^\pm,H^\pm,p^\pm,S^\pm)$ of \eqref{mhd1} in $\Omega^\pm(t)$ such that $\Gamma (t)$
is a tangential discontinuity, namely the plasma does not flow through the discontinuity front and the magnetic
field is tangent to $\Gamma (t)$, see e.g. \cite{landau}, so that the boundary conditions take the form
\begin{equation}
\label{RH}
\dt \psi =v^\pm \cdot N \, ,\quad H^\pm \cdot N=0 \, ,\quad [q]=0 \quad {\rm on }\;\Gamma (t) \, ,
\end{equation}
with $N:=(1,-\ddue \psi,-\dtre \psi)$. Because of the possible jump in the tangential velocity and magnetic fields, there is a concentration of vorticity and current along the discontinuity $\Gamma (t)$. Notice that the function $\psi$ describing the discontinuity front is part of
the unknown of the problem, i.e. this is a free boundary problem. The well-posedness of the nonlinear problem \eqref{mhd1}--\eqref{RH} is shown in \cite{ChenWang,trakhinin09arma} under the assumption of the structural stability condition $|H^+\times H^-|>0$ on $\Gamma (t)$.

After a change of independent variables that \lq\lq flattens\rq\rq the boundary, a linearization around a suitable basic state and some reductions, Trakhinin \cite{trakhinin05,trakhinin09arma} (see also \cite{ChenWang}) gets a linearized problem for $u=(v^\pm,H^\pm,p^\pm,S^\pm)$ of the form \eqref{sbvp} with $\mathcal{L}_\gamma$ as in \eqref{pdo}, $b_\gamma$ as in \eqref{oprt_bordo1}, $\mathcal{M}_\gamma$  as in \eqref{oprt_bordo2} but with $M_2=M_3=0$, that is the boundary operator has order zero in $u$. Moreover, because of the special reductions, the boundary data are zero, i.e. $g=0$ in \eqref{sbvp2}, and $F$ in \eqref{sbvp1} is such that the solution satisfies some additional constraints.

It is proved that the solution of the linearized problem satisfies an a priori estimate similar to \eqref{h1estimate} (with $g=0$).
Instead, the linearized problem with general data $F$ and $g\not=0$ admits an a priori estimate with a loss of two derivatives, see \cite{trakhinin09arma} for details.
\newline
Analogous results for incompressible current-vortex sheets are obtained in \cite{catania13} and \cite{morandotrakhinintrebeschi}.


\subsection{Plasma-vacuum 1}\label{appB.2}

Using the previous notations, let $\Omega^+(t)$ and $\Omega ^-(t)$ be space-time domains occupied by the plasma and the vacuum respectively. That is, in the domain
$\Omega^+(t)$ we consider system \eqref{mhd1}, \eqref{mhd2} governing the motion of an ideal plasma and in the domain $\Omega^-(t)$ we consider the so-called {\it pre-Maxwell dynamics}
\begin{equation}
\nabla \times \mathcal{H} =0,\qquad {\rm div}\, \mathcal{H}=0,\label{mhd3}
\end{equation}
describing the vacuum magnetic field $\mathcal{H}\in\mathbb{R}^3$, see \cite{Goed}.

The plasma variable $(v,H,p,S)$ is connected with the vacuum magnetic field
$\mathcal{H}$  through the relations
\cite{Goed}
\begin{equation}
\dt \psi =v \cdot N,\quad  H\cdot N=0, \quad  \mathcal{H}\cdot N=0 \, ,\quad [q]=0,\quad \mbox{on}\ \Gamma (t),
\label{mhd4}
\end{equation}
where the jump of the total pressure across the interface is $[q]= q|_{\Gamma}-\frac{1}{2}|\mathcal{H}|^2_{|\Gamma}$.
The well-posedness of the nonlinear problem \eqref{mhd1}, \eqref{mhd2}, \eqref{mhd3}, \eqref{mhd4} is shown in \cite{secchitrakhinin,SeTrNl} under the assumption of the structural stability condition $|H \times \mathcal{H}|>0$ on $\Gamma (t)$.

As in the case of current-vortex sheets, after a change of independent variables that \lq\lq flattens\rq\rq the boundary, a linearization around a suitable basic state and some reductions, the authors obtain a linearized problem for $u=(v,H,p,S,\mathcal{H})$ of the form \eqref{sbvp} with $\mathcal{L}_\gamma$ as in \eqref{pdo}, $b_\gamma$ as in \eqref{oprt_bordo1}, $\mathcal{M}_\gamma$  as in \eqref{oprt_bordo2} with $M_2=M_3=0$, that is the boundary operator has order zero in $u$. Moreover, because of the special reductions, the boundary data are zero, i.e. $g=0$ in \eqref{sbvp2}, and $F$ in \eqref{sbvp1} is such that the solution satisfies some additional constraints.

In \cite{secchitrakhinin} it is proved that the solution of the linearized problem satisfies an a priori estimate similar to \eqref{h1estimate} (with $g=0$). The vacuum magnetic field $\mathcal{H}$ is estimated in the standard Sobolev space $H^1$ with full regularity.
Instead, the linearized problem with general data $F$ and $g\not=0$ admits an a priori estimate similar to \eqref{h1h2estimate}, with loss of one derivative in $F$ and $g$, see \cite{SeTrNl}.
\newline
For similar results in the case of the incompressible plasma - vacuum problem, see \cite{motratrevacuum}.

\subsection{Plasma-vacuum 2}\label{appB.3}

In the domain
$\Omega^+(t)$ we consider system \eqref{mhd1}, \eqref{mhd2} governing the motion of an ideal plasma and in the domain $\Omega^-(t)$ we consider the Maxwell equations
\begin{equation}\label{eq:Maxwell}
\begin{cases}
\dt \mathcal{H} + \nabla \times \mathcal{E} = 0  \,, \\
\dt \mathcal{E} - \nabla \times \mathcal{H} = 0  \,, \\
{\rm div}\, \mathcal{H} \,=  \,
{\rm div}\,  \mathcal{E} = 0 \,,
\end{cases}
\end{equation}
describing the vacuum magnetic and electric fields $\mathcal{H},\mathcal{E}\in\mathbb{R}^3$, see \cite{Goed}.

The plasma variable $(v,H,p,S)$ is connected with the vacuum variable
$(\mathcal{H},\mathcal{E})$  through the relations
\cite{Goed}
\begin{equation}
\dt \psi =v \cdot N,\quad  H\cdot N=0, \quad  \mathcal{H}\cdot N=0 \, ,\quad [q]=0,\quad  N \times \mathcal{E} = (N \cdot v)  \mathcal{H},\quad \mbox{on}\ \Gamma (t),
\label{mhd5}
\end{equation}
where the jump of the total pressure across the interface is $[q]= q|_{\Gamma}-\frac{1}{2}|\mathcal{H}|^2_{|\Gamma}+\frac12|\mathcal{E}|^2_{|\Gamma}$.

The stability of the linearized problem obtained from \eqref{mhd1}, \eqref{mhd2}, \eqref{eq:Maxwell}, \eqref{mhd5} is shown in \cite{CDAS} under suitable stability conditions on $\Gamma (t)$.
The authors obtain a linearized problem for $u=(v,H,p,S,\mathcal{H},\mathcal{E})$ of the form \eqref{sbvp} with $\mathcal{L}_\gamma$ as in \eqref{pdo}, $b_\gamma$ as in \eqref{oprt_bordo1}, $\mathcal{M}_\gamma$  as in \eqref{oprt_bordo2} with $M_2=M_3=0$, that is the boundary operator has order zero in $u$. Moreover, because of the special reductions, the boundary data are zero, i.e. $g=0$ in \eqref{sbvp2}, and $F$ in \eqref{sbvp1} is such that the solution satisfies some additional constraints.
It is proved that the solution of the linearized problem satisfies an a priori estimate similar to \eqref{h1estimate} (with $g=0$).
The vacuum variable
$(\mathcal{H},\mathcal{E})$ is estimated in the standard Sobolev space $H^1$ with full regularity.

\subsection{Contact discontinuities}\label{appB.4}
We consider the equations of ideal compressible MHD \eqref{mhd1} for two-dimensional planar flows with respect to the unknown vector $U=(p, v, H, S)$, with $v(t,x)=(v_1,v_2)\in\mathbb R^2$, $H(t,x)=(H_1,H_2)\in\mathbb R^2$, $x=(x_1,x_2)$. For simplicity, let us assume that the plasma obeys the state equation of a polytropic gas
\begin{equation}\label{polytropic}
\rho(p, S)=Ap^{1/\gamma}e^{-S/\gamma}\,,\quad A>0\,,\,\,\,\gamma>1\,.
\end{equation}
Following the notations already introduced in Section \ref{appB.1}, {\it contact discontinuities} are weak solutions of \eqref{mhd1}, that are smooth on either side of a smooth hypersurface  $\Gamma(t)=\{x_1=\psi(t,x_2)\}$ in $[0,T]\times\R^2$, satisfying at each point of the front $\Gamma (t)$ suitable jump conditions. More precisely, one looks for smooth solutions $U^\pm$ of \eqref{mhd1} in $\Omega^\pm(t):=\{x_1\gtrless\psi(t,x_2)\}$, satisfying on $\Gamma(t)$ the following conditions
\begin{equation}\label{cd}
v^+_N - \partial_t \psi =0\,,\quad \left[v\right]=0\,,\quad \left[H\right]=0\,,\quad H_N^\pm\neq 0\,,\quad \left[p\right]=0\,,
\end{equation}
where $N:=(1, -\partial_2\psi)$ is the space normal to the front $\Gamma(t)$, $H_N=H_1-\partial_2\psi H_2$.
\newline
After a change of independent variables that ``flattens'' the boundary, in \cite{contactmtt2013} the authors perform a linearization of the free-boundary problem \eqref{mhd1}, \eqref{cd} for contact discontinuities, around a suitable sufficiently smooth basic state $(\hat p, \hat v, \hat H, \hat S, \hat\varphi)$, obeying the ``stability'' condition
\begin{equation}\label{rt}
\left[\partial_1\hat p\right]\ge c_0>0\,,\quad \mbox{on}\,\,\{x_1=\hat\varphi(t,x_2)\}\,.
\end{equation}
Under the preceding assumptions, the linearized problem can be recast in the form of \eqref{sbvp} with $\mathcal{L}_\gamma$ as in \eqref{pdo}, $b_\gamma=0$ and $\mathcal{M}_\gamma$ of order one in $U$ as in \eqref{oprt_bordo2}. Moreover, because of the special reductions, the boundary data are zero, i.e. $g=0$ in \eqref{sbvp2}, whereas the only nonzero components of $F$ in \eqref{sbvp1} are the ones corresponding to the equation for $v$.
\newline
In \cite{contactmtt2013} it is proved that the solution of the above linearized problem satisfies an a priori estimate in the Sobolev space $H^1_{tan}$ similar to \eqref{h1estimate}.
\bibliographystyle{plain}

\end{document}